\documentclass[11pt]{article}
\usepackage{amsthm,amsmath,amssymb,bbm}
\usepackage{color}
\def\bbe{\mathbb E}

\def\bbr{\mathbb R}

\def\bbone{{\mathbbm 1}}

\theoremstyle{plain}
\newtheorem{thm}{Theorem}[section]

\newtheorem{cor}{Corollary}[section]
\newtheorem{prop}{Proposition}[section]
\newtheorem{lem}{Lemma}[section]
\theoremstyle{definition}
\newtheorem{rem}{Remark}[section]
\allowdisplaybreaks[1]

\begin{document}
\title{On Stein's Method for Multivariate Self-Decomposable Laws}

\author{Benjamin Arras\thanks{Universit\'e de Lille, Laboratoire Paul Painlev\'e, CNRS U.M.R. 8524, 
59655 Villeneuve d'Ascq, France; arrasbenjamin@gmail.com.~I thank the Institute for Mathematical Sciences of the National University of Singapore for its hospitality and support and the organizers of the $2019$ Symposium in Memory of Charles Stein $[1920 - 2016]$ for their invitation to attend it.}\; and Christian Houdr\'e\thanks{Georgia Institute of Technology, 
School of Mathematics, Atlanta, GA 30332-0160, USA; houdre@math.gatech.edu. 
Research supported in part by the grant \# 524678 from the Simons Foundation.
}}

\maketitle

\vspace{\fill}
\begin{abstract}
This work explores and develops elements of Stein's method of approximation, in the infinitely divisible setting, and its connections to functional analysis.  It is mainly concerned with multivariate self-decomposable laws without finite first moment and, in particular, with $\alpha$-stable ones, $\alpha \in (0,1]$.  At first, several characterizations of these laws via covariance identities are presented.  In turn, these characterizations lead to integro-differential equations which are solved with the help of both semigroup and Fourier methodologies. 
Then, Poincar\'e-type inequalities for self-decomposable laws having finite first moment are revisited. 
In this non-local setting, several algebraic quantities (such as the carr\'e du champs and its iterates) originating 
in the theory of Markov diffusion operators are computed.  Finally, rigidity and stability results for 
the Poincar\'e-ratio functional of the rotationally invariant $\alpha$-stable laws, $\alpha\in (1,2)$, are obtained;  
and as such they recover the classical Gaussian setting as $\alpha \to 2$.  
\end{abstract}
\vspace{\fill}

\section{Introduction}
The present notes form a sequel to the works \cite{AH18_1,AH18_2} where Stein's method for general univariate and multivariate infinitely divisible laws with finite first moment has been initiated. Introduced in \cite{Stein1,Stein2}, Stein's method is a collection of methods allowing to control the discrepancy, in a suitable metric, between probability measures and to provide quantitative rates of convergence in weak limit theorems. Originally developed for the Gaussian and the Poisson laws (\cite{C75}), several non-equivalent investigations have focused on extensions and generalizations of Stein's method outside the classical univariate Gaussian and Poisson settings.~In this regard, let us cite \cite{Ba88,BCL92,Luk94,PR14,GPR17,LRS17,McPekSw18,DoPec18,X19} and \cite{Bar90,Go1991,GR96,RR96,Rai04,ChM08,RR09,Mec09,NPR10,R13,GDVM16,AH18_2} for univariate and multivariate extensions and generalizations. Moreover, for good introductions to the method, let us refer the reader to the standard references and surveys \cite{DH04,BC05,R11,CGS,Ch14}. In all the works just cited, the target probability distribution admits, at the very least, a finite first moment. The very recent \cite{CNXYZ}, where the case of the univariate $\alpha$-stable distributions, $\alpha \in (0,1]$, is studied, seems to be the only instance of the method bypassing the finite first moment assumption. Below, we develop a Stein's method framework for non-degenerate multivariate self-decomposable distributions without Gaussian component. 

Let us recall that self-decomposable distributions form a subclass of infinitely divisible distributions. Moreover, they are weak limits of normalized sums of independent summands and, as such, they naturally generalize the Gaussian/stable distributions. Originally introduced by Paul L\'evy in \cite{L54}, self-decomposable distributions and their properties have been studied in depth by many authors (see, e.g., \cite{Loe78,S}).

The methodology developed here for multivariate self-decomposable distributions relies on a specific semigroup of operators already put forward in our previous analyses \cite{AH18_1,AH18_2}. The generator of this semigroup is an integro-differential operator whose non-local part depends in a subtle way on the L\'evy measure of the target self-decomposable distribution (see Lemma \ref{lem:genSDLaws}). Indeed, the non-local part of this operator differs from the one obtained in \cite{AH18_1,AH18_2} since the Fourier symbols of the associated semigroup of operators do not exhibit $\mathcal{C}^1$-smoothness. However, by exploiting the polar decomposition of the L\'evy measure of the target self-decomposable distribution together with the monotonicity of the associated $k$-function, $\mathcal{C}^1$-regularity is reached and as such natural candidates for the corresponding Stein equation and its solution are put forward. Moreover, this equation reflects the L\'evy-Khintchine representation used to express the characteristic function of the target self-decomposable distribution. This naturally induces three types of equations reminiscent of the following classical distinction between stable laws: $\alpha \in (0,1)$, $\alpha =1$ and $\alpha \in (1,2)$.

With these new findings, we revisit Poincar\'e-type inequalities for self-decomposable distributions with finite first moment. Initially obtained in \cite{Chen} (see also \cite{HPAS}), these Poincar\'e-type inequalities reflect the infinite divisibility of the reference measure (without Gaussian component) and as such put into play a non-local Dirichlet form contrasting with the standard local Dirichlet form associated with the Gaussian measures. Our new proof of these Poincar\'e-type inequalities is based on the semigroup of operators already put forward (and used to solve the Stein equation) in \cite{AH18_2} and is in line with the proof of the Gaussian Poincar\'e inequality based on the differentiation of the variance along the Ornstein-Uhlenbeck semigroup (see, e.g., \cite{BGL14}). Moreover, in this non-local setting, we compute several algebraic quantities (such as the carr\'e du champs and its iterates) originating in Markov diffusion operators theory in order to reach rigidity and stability results for the Poincar\'e-ratio ($U$-) functional defined in \eqref{eq:Ufunctional} and associated with the rotationally invariant $\alpha$-stable distributions. Rigidity results for infinitely divisible distributions with finite second moment were obtained in \cite[Theorem $2.1$]{ChLo87} whereas the corresponding stability results were obtained in \cite[Theorem $4.5$]{AH18_2} through Stein's method and variational techniques inspired by \cite{CFP19}. Here, for the rotationally invariant $\alpha$-stable distribution, $\alpha \in (1,2)$, we revisit the method of \cite{CFP19} using the framework of Dirichlet forms. Coupled with a truncation procedure, rigidity and stability of the Poincar\'e $U$-functional are stated in Corollary \ref{cor:rigidity1}, Corollary \ref{cor:rigidity2} and Theorem \ref{thm:stability}. This truncation procedure allows us to build an optimizing sequence for the $U$-functional. This sequence of functions can be spectrally interpreted as a singular sequence verifying a Weyl-type condition associated with the corresponding Poincar\'e constant (see Conditions \eqref{eq:assrigidity2} and \eqref{Weyltype} below).

Let us further describe the content of our notes: The next section introduces notations and definitions used throughout this work and prove a characterization theorem for multivariate infinitely divisible distributions with finite first moment. In Section \ref{sec:charSDLaws}, using the previous characterization and truncation arguments, we obtain several characterization theorems for self-decomposable laws lacking finite first moment. These results highlight the role of the L\'evy-Khintchine representation of the characteristic function of the target self-decomposable distribution and apply, in particular, to multivariate stable laws with stability index in $(0,1]$. In Section \ref{sec:SteinEqSD}, a Stein equation for non-degenerate multivariate self-decomposable distributions without finite first moment is at first put forward. It is then solved, under a low moment condition, via a combination of semigroup techniques and Fourier analysis. In the last section, Poincar\'e-type inequalities for self-decomposable laws with finite first moment are looked a new. Several algebraic quantities originating in Markov diffusion operators theory are computed in this non-local setting. In particular, for the rotationally invariant $\alpha$-stable laws with $\alpha\in (1,2)$, a Bakry-\'Emery criterion is shown to hold, recovering as $\alpha \rightarrow 2$, the classical Gaussian theory involving the carr\'e du champs and its iterates. Finally, rigidity and stability results for the Poincar\'e $U$-functional of the rotationally invariant $\alpha$-stable distributions, $\alpha \in (1,2)$, are obtained using elements of spectral analysis and Dirichlet form theory.~A technical appendix finishes our manuscript.

\section{Notations and Preliminaries}
\label{sec:notprel}
Throughout, let $\|\cdot\|$ and $\langle \cdot;\cdot \rangle$ be respectively the Euclidean norm and inner product on $\mathbb{R}^d$, $d\geq 1$. Let also ${\cal S}(\mathbb{R}^d)$ be the Schwartz space of infinitely differentiable rapidly decreasing real-valued functions defined on $\mathbb R^d$, and finally let $\mathcal{F}$ be the Fourier transform operator given, for $f\in {\cal S}(\mathbb{R}^d)$, by 
\begin{align*}
\mathcal{F}(f)(\xi)=\int_{\mathbb{R}^d}f(x)e^{-i \langle \xi; x \rangle}dx, \quad \xi \in \mathbb{R}^d.
\end{align*}
On ${\cal S}(\mathbb{R}^d)$, the Fourier transform is an isomorphism and the following inversion formula is well known
\begin{align*}
f(x)=\int_{\mathbb{R}^d}\mathcal{F}(f)(\xi)e^{i \langle \xi; x \rangle}\frac{d\xi}{(2\pi)^d}, \quad x\in \mathbb{R}^d.
\end{align*}
Next, ${\cal C}_b(\mathbb{R}^d)$ is the space of bounded continuous functions on $\mathbb{R}^d$ endowed with the uniform norm $\|f\|_\infty=\sup_{x\in \mathbb{R}^d}|f(x)|$, $f\in {\cal C}_b(\mathbb{R}^d)$.~For $\mu$ a probability measure on $\bbr^d$ and for $1\leq p<+\infty$, $L^p(\mu)$ is the Banach space of equivalence classes of functions defined $\mu$-a.e.on $\bbr^d$ such that $\|f\|^p_{L^p(\mu)}=\int_{\bbr^d} |f(x)|^p\mu(dx)<+\infty$, $f\in L^p(\mu)$. Similarly, $L^\infty(\mu)$ is the space of equivalence classes of functions bounded everywhere and $\mu$-measurable. For any bounded linear operator, $T$, from a Banach space $({\cal X}, \|\cdot\|_{{\cal X}})$ to another Banach space $({\cal Y}, \|\cdot\|_{{\cal Y}})$ the operator norm is, as usual,
\begin{align}
\|T\|_{{\cal X}\rightarrow {\cal Y}}=\underset{f\in {\cal X},\, \|f\|_{{\cal X}}\ne 0}{\sup}\dfrac{\|T(f)\|_{{\cal Y}}}{\|f\|_{{\cal X}}}.
\end{align}
More generally, for any $r$-multilinear form $F$ from $(\mathbb{R}^d)^r$, $r\geq 1$, to $\mathbb{R}$, the operator norm of $F$ is
\begin{align}
\|F\|_{op}:=\sup \left\{|F(v_1,...,v_r)|:\, v_j \in \mathbb{R}^d,\, \|v_j\|=1,\, j=1,...,r\right\}.
\end{align}
Through the whole text, a L\'evy measure is a positive Borel measure on $\mathbb{R}^d$ such that $\nu(\{0\})=0$ and $\int_{\mathbb{R}^d} (1\wedge \|u\|^2)\nu(du)<+\infty$. An $\mathbb{R}^d$-valued random vector $X$ is infinitely divisible with triplet $(b,\Sigma,\nu)$ (written $X\sim ID(b, \Sigma,\nu)$), if its characteristic function $\varphi$ writes, for all $\xi\in\mathbb{R}^d$, as
\begin{align}\label{eq:characID}
\varphi(\xi)=\exp\left(i \langle b;\xi \rangle-\frac{1}{2}\langle \xi;\Sigma\xi \rangle +\int_{\mathbb{R}^d}\left(e^{i \langle \xi; u\rangle}-1-i\langle \xi;u \rangle\bbone_{D}(u)\right)\nu(du)
\right),
\end{align}
with $b\in\mathbb{R}^d$, $\Sigma$ a symmetric positive semi-definite $d\times d$ matrix, $\nu$ a L\'evy measure on $\mathbb{R}^d$ and $D$ the closed Euclidean unit ball of 
$\mathbb{R}^d$. 

The representation \eqref{eq:characID} is mainly the one to be used, 
from start to finish, with the (unique) generating triplet
$(b,\Sigma,\nu)$.  However, other types of representations are also possible 
and two of them are presented
next.  First, if $\nu$ is such that $\int_{\|u\|\le 1} \|u\|\nu (du)<+\infty$,
then \eqref{eq:characID} becomes
\begin{equation}\label{eq:characID2}
\varphi(\xi)=\exp\left(i \langle b_0;\xi \rangle-\frac{1}{2}\langle \xi;\Sigma\xi \rangle +\int_{\mathbb{R}^d}\left(e^{i \langle \xi; u\rangle}-1\right)\nu(du)
\right),
\end{equation}
where $b_0 = b-\int_{\|u\|\le 1}u\nu (du)$ is called the {\it drift} of $X$.
This representation is cryptically
expressed as $X\sim ID(b_0,\Sigma,\nu)_0$.  Second, if $\nu$ is
such that $\int_{\|u\|>1} \|u\|\nu(du)<+\infty$, then
\eqref{eq:characID} becomes
\begin{equation}\label{eq:characID3}
\varphi(\xi)=\exp\left(i \langle b_1;\xi \rangle-\frac{1}{2}\langle \xi;\Sigma\xi \rangle +\int_{\mathbb{R}^d}\left(e^{i \langle \xi; u\rangle}-1-i\langle \xi;u \rangle\right)\nu(du)
\right),
\end{equation}
where $b_1=b+\int_{\|u\|>1} u\nu (du)$ is called the {\it center} of $X$. 
In turn,
this last representation is now cryptically written as 
$X\sim ID(b_1,\Sigma,\nu)_1$. In fact, $b_1=\bbe X$  as, for any
$p>0$, $\bbe \|X\|^p<+\infty$ is equivalent to $\int_{\|u\|>1}\|u\|^p\nu (du)
<+\infty$.  Also, for any $r>0$, $\bbe e^{r\|X\|}<+\infty$ is equivalent to
$\int_{\|u\|>1} e^{r\|u\|}\nu (du)<+\infty$.

In the sequel, we are also interested in some distinct classes of infinitely divisible distributions, namely the stable ones and the self-decomposable ones.  
Recall that an ID random vector $X$ is $\alpha$-stable, $0<\alpha < 2$, if $b\in \bbr^d$, if $\Sigma=0$ and if its L\'evy measure $\nu$ admits the following polar decomposition
\begin{align}\label{eq:PolarStable}
\nu(du)=\bbone_{(0,+\infty)}(r)\bbone_{\mathbb{S}^{d-1}}(x)\frac{dr}{r^{\alpha+1}}\sigma(dx),
\end{align}
where $\sigma$ is a finite positive measure on $\mathbb{S}^{d-1}$, the Euclidean unit sphere of $\mathbb{R}^d$. When $\alpha\in (0,1)$, then $\int_{\|u\|\leq 1} |u_j|\nu(du)<+\infty$, for all $1\leq j\leq d$, $\varphi$ and so
\begin{align}\label{eq:characStable}
\varphi\left(\xi\right)=\exp\left(i\langle \xi;b_0 \rangle+\int_{\mathbb{R}^d} \left(e^{i \langle \xi; u \rangle}-1\right)\nu(du)\right)\quad \xi \in \bbr^d,
\end{align}
with, again, $b_0=b-\int_{\|u\|\leq 1} u\nu(du)$.

Now, recall that an ID random vector $X$ is self-decomposable (SD) if $b\in \bbr^d$, 
if $\Sigma=0$, and if its L\'evy measure $\nu$ admits the polar decomposition
\begin{align}\label{eq:PolarSD1}
\nu(du)=\bbone_{(0,+\infty)}(r)\bbone_{\mathbb{S}^{d-1}}(x)\frac{k_x(r)}{r}dr\sigma(dx),
\end{align}
where $\sigma$ is a positive finite measure on $\mathbb{S}^{d-1}$ and where $k_x(r)$ is a function which is nonnegative, decreasing in $r$, ($k_x(r_1)\leq k_x(r_2)$, for $0<r_2\leq r_1$) and measurable in $x$. In the sequel, without loss of generality, $k_x(r)$ is assumed to be right-continuous in $r\in (0,+\infty)$, to admit a left-limit at each $r\in (0,+\infty)$ and $\int_0^{+\infty} (1\wedge r^2) k_x(r)dr/r$ is independent of $x$.

Next, (see, e.g., \cite[Chapter $12$]{PM91}) let us denote by $V^b_a(g)$ the variation of a function $g$ over the interval $[a,b]\subsetneq (0,+\infty)$,
\begin{align}
V^b_a(g)=\underset{\mathcal{P}}{\sup}\sum_{i=1}^{n} \left|g(x_{i})-g(x_{i-1})\right|,
\end{align}
where the supremum is taken over all subdivisions $\mathcal{P}=\{a=x_0<x_1<\cdots< x_n=b\}$ of $[a,b]$.

Since $k_x(r)$ is of bounded variation in $r$ on any $(a,b)\subsetneq (0,+\infty)$, $a>0,b>0$ and $a\leq b$, and right-continuous in $r\in (0,+\infty)$, the following integration by parts formula holds true
\begin{align}\label{eq:IPPradial}
\int_a^b k_x(r)f'(r)dr=-\int_a^b f(r) dk_x(r), \quad x\in \mathbb{S}^{d-1},
\end{align}
for all $f$ continuously differentiable on $(a,b)$ such that $\underset{r\rightarrow a^+}{\lim} f(r)k_x(r)=0$ and $\underset{r\rightarrow b^-}{\lim} f(r)k_x(r)=0$, $x\in \mathbb{S}^{d-1}$. 

Let us now introduce some natural distances between probability measures on $\mathbb{R}^d$. Let $\mathbb{N}^d$ be the space of multi-indices of dimension $d$.  For any $\alpha\in \mathbb{N}^d$, $|\alpha|=\sum_{i=1}^d |\alpha_i|$ and $D^{\alpha}$ denote the partial derivatives operators defined on smooth enough functions $f$, by $D^{\alpha}(f)(x_1,...,x_d)=\partial^{\alpha_1}_{x_1}...\partial^{\alpha_d}_{x_d}(f)(x_1,...,x_d)$, for all $(x_1,...,x_d)\in \mathbb{R}^d$. Moreover, for any $r$-times continuously differentiable function, $h$, on $\mathbb{R}^d$, viewing its $\ell$th-derivative $\mathbf{D}^{\ell}(h)$ as a $\ell$-multilinear form, for $1\leq \ell\leq r$, let
\begin{align}
M_{\ell}(h):=\underset{x\in \mathbb{R}^d}{\sup} \|\mathbf{D}^{\ell}(h)(x)\|_{op}=\underset{x\ne y}{\sup}\dfrac{\|\mathbf{D}^{\ell-1}(h)(x)-\mathbf{D}^{\ell-1}(h)(y)\|_{op}}{\|x-y\|}.
\end{align}
For $r\geq 0$, $\mathcal{H}_r$ is the space of bounded continuous functions defined on $\mathbb{R}^d$ which are continuously differentiable up to (and including) the order $r$ and such that, for any such function $f$,
\begin{align}
\max_{0\leq \ell \leq r}M_\ell(f)\leq 1
\end{align}
with $M_0(f):=\sup_{x\in\mathbb{R}^d}|f(x)|$. Then, the smooth Wasserstein distance of order $r$, between two random vectors $X$ and $Y$ having respective laws $\mu_X$ and $\mu_Y$, is defined by
\begin{align}
d_{W_r}(\mu_X,\mu_Y)=\underset{h\in \mathcal{H}_r}{\sup} \left|\bbe h(X)-\bbe h(Y)\right|.
\end{align}
Moreover, for $r\geq 1$, $d_{W_r}$ admits the following representation (see \cite[Lemma A.2.]{AH18_2})
\begin{align}
d_{W_r}(\mu_X,\mu_Y)=\underset{h\in \mathcal{H}_r\cap \mathcal{C}^{\infty}_c(\mathbb{R}^d)}{\sup} \left|\bbe h(X)-\bbe h(Y)\right|,
\end{align}
where $\mathcal{C}^{\infty}_c(\mathbb{R}^d)$ is the space of infinitely differentiable compactly supported functions on $\mathbb{R}^d$. In particular, for $r\geq 1$,
\begin{align}\label{ineq:wasser}
d_{W_r}(\mu_X,\mu_Y) \leq d_{W_1}(\mu_X,\mu_Y).
\end{align}
As usual, for two probability measures, $\mu_1$ and $\mu_2$, on $\mathbb{R}^d$, $\mu_1$ is said to be absolutely continuous with respect to $\mu_2$, denoted by $\mu_1<<\mu_2$, if for any Borel set, $B$,  such that $\mu_2(B)=0$, it follows that $\mu_1(B)=0$.

To end this section, let us state the following characterization result of ID random vectors with finite first moment, valid, for example, for stable random vector with stability index $\alpha \in (1,2)$ has its origin in the univariate result \cite[Theorem 3.1]{AH18_1}.

\begin{thm}\label{thm:multid}
Let $X$ be a random vector such that $\bbe |X_i|<+\infty$, for all $i\in \{1,\dots, d\}$. Let $\nu$ be a L\'evy measure on $\bbr^d$ such that $\int_{\|u\|\geq 1}\|u\|\nu(du)<+\infty$. Then,
\begin{align*}
\bbe Xf(X)=\bbe X \bbe f(X)+\bbe \int_{\bbr^d} \left(f(X+u)-f(X)\right) u\nu(du),
\end{align*}
for all $f$ bounded Lipschitz function on $\bbr^d$, if and only if $X$ is an ID random vector with L\'evy measure $\nu$ (and $b = \bbe X -\int_{\|u\|>1}u\nu(du)$).
\end{thm}

\begin{proof}
Let us assume that $X$ is an ID random vector with finite first moment and with L\'evy measure $\nu$. Then, from 
\cite[Proposition $2$]{HPAS}, for all $f$ and $g$ bounded Lipschitz functions on $\bbr^d$,
\begin{align*}
\operatorname{Cov}\left(f(X),g(X)\right)=\int_0^1 \bbe \left(\int_{\bbr^d} \left(f(X_z+u)-f(X_z)\right) \left(g(Y_z+u)-g(Y_z)\right) \nu(du)\right)\!dz,
\end{align*}
where $(X_z,Y_z)$ is an ID random vector in $\bbr^{2d}$ defined through an interpolation scheme as in \cite[Equation $(2.7)$]{HPAS}. Now, since $X$ has finite first moment, one can take for $g$ the function $g_t(x)=\langle t;x \rangle$, for all $x\in \bbr^d$ and for some $t\in \bbr^d$. Then, by linearity
\begin{align*}
\langle t; \bbe X f(X) \rangle=\langle t; \bbe X \bbe f(X) + \bbe \int_{\bbr^d} \left(f(X+u)-f(X)\right) u\nu(du)\rangle,
\end{align*}
since $X_z=_dX$, where $=_d$ stands for equality in distribution. This concludes the direct part of the proof. Conversely, let us assume that 
\begin{align*}
\bbe Xf(X)=\bbe X \bbe f(X)+\bbe \int_{\bbr^d} \left(f(X+u)-f(X)\right) u\nu(du),
\end{align*}
holds true in $\bbr^d$ for all $f$ bounded Lipschitz function on $\bbr^d$. Consider the function $\varphi_t$ defined, for all $(x,\xi)\in \bbr^d\times \bbr^d$ and for all $t\in \bbr$, by
\begin{align*}
\varphi_t(x,\xi)=e^{it \langle x;\xi \rangle}.
\end{align*}
Then, for all $\xi \in \bbr^d$ and all $t\in \bbr$,
\begin{align*}
\bbe X \varphi_t(X,\xi)=\bbe X \bbe \varphi_t(X,\xi)+\bbe \int_{\bbr^d} \left(\varphi_t(X+u,\xi)-\varphi_t(X,\xi)\right) u\nu(du),
\end{align*}
where the equality is understood to be in $\bbr^d$. In particular, one has
\begin{align*}
\bbe \langle \xi; X\rangle \varphi_t(X,\xi)=\langle \xi;\bbe X\rangle \bbe \varphi_t(X,\xi)+\langle \xi ;\bbe \int_{\bbr^d} \left(\varphi_t(X+u,\xi)-\varphi_t(X,\xi)\right) u\nu(du)\rangle.
\end{align*}
Denoting by $\Phi_t$ the function defined by $\Phi_t(\xi)=\bbe \left(e^{i t \langle X;\xi \rangle}\right)$, for all $\xi \in \bbr^d$, the previous equality boils down to
\begin{align*}
\dfrac{d}{dt} \left(\Phi_t(\xi)\right)=i \Phi_t(\xi)\left(\langle \xi; \bbe X\rangle+ \langle \xi; \int_{\bbr^d} \left(e^{i t \langle \xi; u\rangle}-1\right) u\nu(du) \rangle\right).
\end{align*}
Moreover, one notes that $\Phi_0(\xi)=1$. Then, for all $\xi \in \bbr^d$ and all $t\in \bbr$,
\begin{align*}
\Phi_t(\xi)=\exp\left(i t \langle \xi; \bbe X\rangle+\int_{\bbr^d} \left(e^{it \langle u;\xi \rangle}-1-i t\langle \xi; u \rangle\right)\nu(du)\right).
\end{align*}
Taking $t=1$, the characteristic function of $X$ is then given, for all $\xi \in \bbr^d$, by
\begin{align*}
\varphi(\xi)=\exp\left(i \langle \xi; \bbe X\rangle+\int_{\bbr^d} \left(e^{i \langle u;\xi \rangle}-1-i \langle \xi; u \rangle\right)\nu(du)\right),
\end{align*}
namely, $X$ is ID with Levy measure $\nu$ (and $b = \bbe X -\int_{\|u\|>1}u\nu(du)$).
\end{proof}

\begin{rem}\label{rem:sec2}
(i) Let $\varepsilon \in (0, 1]$ and let $X_\varepsilon\sim ID(b_\varepsilon,0,\nu_\varepsilon)$,  
with L\'evy measure $\nu_{\varepsilon}$ such 
that $\int_{\bbr^d} \|u\|^2 \nu_\varepsilon(du)<+\infty$ and with 
$b_\varepsilon=-\int_{\|u\|> 1}u \nu_\varepsilon(du)$.  Then, $X_\varepsilon$ is a centered 
random vector of $\bbr^d$ with covariance matrix $\Sigma_\varepsilon=\int_{\bbr^d}uu^{t} \nu_\varepsilon(du)$, where $u^t$ is the transpose of $u\in \bbr^d$.   
Assume further that  $\Sigma_\varepsilon$, $\varepsilon > 0$, is non-singular, equivalently (see, e.g., \cite[Lemma $2.1$]{CR07}) that $\nu_{\varepsilon}$ is not concentrated on any proper linear subspace of $\bbr^d$, 
equivalently that the law of $X_\varepsilon$ is not concentrated on any proper linear hyperplane 
of $\bbr^d$.   

Then, by \cite[Theorem $2.2$]{CR07},  
whenever $\Sigma_\varepsilon$, is non-singular for every $\varepsilon \in (0, 1]$, the following two conditions are equivalent:  

(a) As $\varepsilon \to 0^+$, 
$\tilde{X}_\varepsilon=\Sigma_\varepsilon^{-1/2}X_\varepsilon$ converges in distribution to a centered multivariate Gaussian random vector with identity covariance matrix.

(b) For every $\kappa>0$
\begin{align}\label{CNS}
\int_{\langle \Sigma_\varepsilon^{-1} u;u \rangle >\kappa}\langle \Sigma_\varepsilon^{-1} u;u \rangle \nu_\varepsilon(du) \underset{\varepsilon \rightarrow 0^+}{\longrightarrow} 0.
\end{align}
(We refer the reader to \cite{CR07} for the above requirements as well as for 
sufficient conditions, on $\nu_\varepsilon$, ensuring their non-vacuity.)

 Moreover, by Theorem~\ref{thm:multid}, for all $f\in\mathcal{S}(\bbr^d)$ and all $\varepsilon>0$,
\begin{align}\label{eq:Steinlemma}
\bbe \tilde{X}_\varepsilon f(\tilde{X}_\varepsilon) = \bbe \int_{\bbr^d} (f(\tilde{X}_\varepsilon+u)-f(\tilde{X}_\varepsilon))u\tilde{\nu}_\varepsilon(du),
\end{align}
with $\tilde{\nu}_\varepsilon$ the pushforward of $\nu_\varepsilon$ by $\Sigma_\varepsilon^{-1/2}$. Note that $\int_{\bbr^d}uu^{t} \tilde{\nu}_\varepsilon(du)=I_d$, where $I_d$ is the $d\times d$ identity matrix, and that, in view of \eqref{CNS},
\begin{align}\label{ConsCNS}
\tilde{\nu}_\varepsilon\left(\|u\|\geq \kappa\right)\underset{\varepsilon\rightarrow 0^+}{\longrightarrow} 0,\quad \left\|I_d
-\int_{\|u\|\leq \kappa}uu^t \tilde{\nu}_\varepsilon(du)\right\|_{op}\underset{\varepsilon\rightarrow 0^+}{\longrightarrow} 0, \quad \kappa >0.
\end{align}
Now, set $F_\varepsilon(z)=\int_{\bbr^d} (f(z+u)-f(z))u\tilde{\nu}_\varepsilon(du)$ and $F(z)=\nabla(f)(z)$, for all $z\in \bbr^d$ and all $\varepsilon\in (0,1)$. Observe that, for all $z\in \bbr^d$, all $\varepsilon>0$ and all $\kappa>0$,
\begin{align*}
\|F_\varepsilon(z)-F(z)\|&\leq\left\|\int_{\|u\|\leq \kappa } \left(f(z+u)-f(z)-\langle\nabla(f)(z); u\rangle\right)u\tilde{\nu}_{\varepsilon}(du)\right\|,\\
&\quad\quad +\left\|\int_{\|u\|\geq \kappa } \left(f(z+u)-f(z)-\langle\nabla(f)(z); u\rangle\right)u\tilde{\nu}_{\varepsilon}(du)\right\|\\
&\leq \frac{\kappa}{2}M_2(f) \int_{\|u\|\leq \kappa} \|u\|^2\tilde{\nu}_{\varepsilon}(du)+ 2M_1(f) \int_{\|u\|\geq \kappa} \|u\|^2\tilde{\nu}_{\varepsilon}(du).
\end{align*}
Letting first $\varepsilon\rightarrow 0^+$ and then $\kappa \rightarrow 0^+$, one sees that 
$F_\varepsilon$ converges, uniformly to $F$, as $\varepsilon \rightarrow 0^+$. 
Then, since $\tilde{X}_\varepsilon$ converges in distribution to $Z\sim \mathcal{N}\left(0,I_d\right)$, (see also, \cite[Remark $4$]{AV} for a univariate result) the identity \eqref{eq:Steinlemma} is preserved as $\varepsilon\rightarrow 0^+$, giving, 
for all $f\in\mathcal{S}(\bbr^d)$, the classical Gaussian characterizing identity,
\begin{align*}
\bbe Z f(Z)=\bbe \nabla(f)(Z).
\end{align*}
(ii) Let $\nu_\alpha$ be the L\'evy measure given, for $\alpha\in (1,2)$, by
\begin{align}\label{eq:nustable}
\nu_\alpha(du)=\bbone_{(0,+\infty)}(r)\bbone_{\mathbb{S}^{d-1}}(x)\dfrac{c_{\alpha,d}}{r^{\alpha+1}}dr\sigma(dx),
\end{align}
with $c_{\alpha,d}>0$, a normalizing constant specified later, and with $\sigma$ the uniform measure on ${\mathbb S}^{d-1}$.  Then, $X_\alpha \sim ID(b_\alpha, 0, \nu_\alpha)$, 
with $b_\alpha=-\int_{\|u\|\geq 1}u \nu_\alpha(du)$, is a rotationally invariant $\alpha$-stable random vector  
with corresponding characteristic function $\varphi_\alpha$
\begin{align*}
\varphi_\alpha(\xi)=\exp\left(-\|\xi\|^{\alpha}/2\right),\quad \xi\in \bbr^d,
\end{align*}
for $c_{\alpha,d}$ given by 
\begin{align}\label{eq:cstenorm}
c_{\alpha,d}= \dfrac{-\alpha (\alpha-1)\Gamma((\alpha+d)/2)}{4\cos(\alpha\pi/2)\Gamma((\alpha+1)/2)\pi^{(d-1)/2}\Gamma(2-\alpha)}.
\end{align}
Clearly, as $\alpha\rightarrow 2^-$, $X_\alpha$ converges in distribution 
to a centered Gaussian random vector $Z$ with identity covariance matrix.  
Next, by Theorem~\ref{thm:multid}, for all $f$ bounded Lipschitz function on $\bbr^d$,
\begin{align}\label{eq:SteinLemmaStable}
\bbe X_\alpha f(X_\alpha)=\bbe \int_{\bbr^d} \left(f(X_\alpha+u)-f(X_\alpha)\right) u\nu_\alpha(du), 
\end{align}
and observe, at first, that for all $f\in \mathcal{S}(\bbr^d)$,
\begin{align*}
\underset{\alpha \rightarrow 2^{-}}{\lim}\bbe X_\alpha f(X_\alpha) = \bbe Z f(Z).
\end{align*}
Next, let $D_\alpha(f)(z)=\int_{\bbr^d} (f(z+u)-f(z)) u\nu_\alpha(du)$, 
and observe now that the Fourier symbol, $\sigma_\alpha$, of this operator satisfies, 
for all $\xi \in \bbr^d$,
\begin{align*}
\langle \sigma_\alpha(\xi); i\xi\rangle&=ic_{\alpha,d}\int_{\bbr^d}\left(e^{i \langle u;\xi \rangle}-1\right)\dfrac{\langle u;\xi\rangle du}{\|u\|^{d+\alpha}}\\
&=-\dfrac{\alpha}{2} \|\xi\|^{\alpha}.
\end{align*}
Finally, for all $f\in \mathcal{S}(\bbr^d)$
\begin{align*}
\bbe D_\alpha(f)(X_\alpha)=\int_{\bbr^d} \mathcal{F}(f)(\xi) \sigma_\alpha(\xi) \exp\left(-\|\xi\|^{\alpha}/2\right) \frac{d\xi}{(2\pi)^d}\underset{\alpha \rightarrow 2^-}{\longrightarrow} \bbe \nabla(f)(Z),
\end{align*}
 so that the characterizing identity \eqref{eq:SteinLemmaStable} is preserved when 
 passing to the limit, converging, again, for all $f\in \mathcal{S}(\bbr^d)$, to 
 \begin{align*}
 \bbe Z f(Z)=\bbe \nabla(f)(Z).
 \end{align*}
\end{rem}


\section{Characterizations of Self-Decomposable Laws}
\label{sec:charSDLaws}
\noindent
In this section, we provide various characterization results, for stable distributions and some 
self-decomposable ones, not covered by 
Theorem~\ref{thm:multid}. However, the direct parts of these results are 
simple consequences of Theorem~\ref{thm:multid} together 
with truncation and discretization arguments. The stable results recover, in particular, the one-dimensional results independently obtained in \cite{CNXYZ}. Below, and throughout, 
we will make use of the transformation $T_c$ applied to positive (L\'evy) 
measures and defined for all $c>0$ and all Borel sets, $B$, of $\bbr^d$ by
$$T_c(\nu)(B)=\nu(B/c).$$

\begin{thm}\label{thm:caracSteinStable2}
Let $X$ be a random vector in $\bbr^d$. Let $b\in \bbr^d$, $\alpha \in (0,1)$ and let $\nu$ 
be a L\'evy measure such that, for all $c>0$, 
\begin{align}\label{transform}
\nu(du)=c^{-\alpha} T_c(\nu)(du).  
\end{align}
Then,
\begin{align*}
\bbe \langle X; \nabla(f)(X)\rangle=\bbe \langle b_0; \nabla(f)(X)\rangle+\alpha \int_{\bbr^d} \left(f(X+u)-f(X)\right)\nu(du),
\end{align*}
where $b_0 = b - \int_{\|u\|\le 1}u\nu(du)$, for all $f\in\mathcal{S}(\bbr^d)$  
if and only if $X$ is a stable random vector with parameter $b$, 
stability index $\alpha$ and L\'evy measure $\nu$.
\end{thm}

\begin{proof}
Let us first assume that $X$ is a stable random vector in $\bbr^d$ with parameters
 $b\in \bbr^d$, stability index $\alpha\in (0,1)$ and L\'evy measure $\nu$. 
 Then, \cite[Theorem $14.3$, (ii)]{S}, $\nu$ is given by
\begin{align*}
\nu(du)=\bbone_{(0,+\infty)}(r) \bbone_{\mathbb{S}^{d-1}}(x) \frac{dr}{r^{1+\alpha}} \sigma(dx),
\end{align*}
where $\sigma$ is a finite positive measure on the Euclidean unit sphere of $\bbr^d$, and let $b_0 = b - \int_{\|u\|\le 1}u\nu(du)$.
Next, let $R>1$,
\begin{align*}
\nu_R(du):=\bbone_{(0,R)}(r)\bbone_{\mathbb{S}^{d-1}}(x) \dfrac{dr}{r^{\alpha+1}}\sigma(dx).
\end{align*}
and, let $X_R$ be the ID random vector defined through its characteristic function by
\begin{align*}
\varphi_R(\xi):=\exp\left(i\langle \xi; b_0 \rangle+\int_{\bbr^d}\left(e^{i \langle \xi; u \rangle}-1\right)\nu_R(du)\right), \quad \xi\in \bbr^d.
\end{align*}
Note, in particular, that $X_R$ is such that $\bbe \|X_R\|<+\infty$. Then, by Theorem \ref{thm:multid}, for all $g\in\mathcal{S}(\bbr^d)$,
\begin{align*}
\bbe X_R g(X_R)=b_0 \bbe g(X_R)+\bbe\int_{\bbr^d} g(X_R+u)u \nu_R(du).
\end{align*}
Now, choosing $g=\partial_i(f)$ for some $f\in \mathcal{S}(\bbr^d)$ and for $i\in\{1,\dots,d\}$, it follows that
\begin{align}\label{eq:3.5}
\bbe X_R \partial_i(f)(X_R)=b_0 \bbe \partial_i(f)(X_R)+\bbe\int_{\bbr^d} \partial_i(f)(X_R+u)u \nu_R(du).
\end{align}
To continue, project the vectorial equality \eqref{eq:3.5} onto the 
direction $e_i=(0,\dots,0,1,0,\dots,0)$, to get, for all $i\in \{1,\dots, d\}$,
\begin{align}\label{eq:3.6}
\bbe X_{R,i} \partial_i(f)(X_R)=b_{0,i} \bbe \partial_i(f)(X_R)+\bbe\int_{\bbr^d} \partial_i(f)(X_R+u)u_i \nu_R(du), 
\end{align}
where $X_{R,i}$ and $b_{0,i}$ are the i-th coordinates of $X_R$ and of $b_0$ respectively.  
Adding-up these last identities, for $i \in \{1,\dots,d\}$, leads to 
\begin{align*}
\bbe \langle X_R; \nabla(f)(X_R) \rangle=\langle b_0 ; \bbe \nabla (f)(X_R)\rangle+\bbe \int_{\bbr^d} \langle \nabla(f)(X_R+u); u \rangle \nu_R(du).
\end{align*}
Now, observe that $X_R$ converges in distribution towards $X$ since by 
the Lebesgue dominated convergence theorem, $\varphi_R(\xi) \underset{R\rightarrow+\infty}{\longrightarrow} \varphi(\xi)$ , for all $\xi \in \bbr^d$.  Hence,
\begin{align*}
\underset{R\rightarrow +\infty}{\lim} \bbe \langle X_R; \nabla(f)(X_R) \rangle =  \bbe \langle X; \nabla(f)(X) \rangle,\quad \underset{R\rightarrow +\infty}{\lim} \langle b_0 ; \bbe \nabla (f)(X_R)\rangle = \langle b_0 ; \bbe \nabla (f)(X)\rangle.
\end{align*}
Moreover, from the polar decomposition of the L\'evy measure $\nu_R$,
\begin{align*}
\bbe \int_{\bbr^d} \langle \nabla(f)(X_R+u); u \rangle \nu_R(du)=\bbe \int_{(0,R)\times \mathbb{S}^{d-1}} \langle \nabla(f)(X_R+rx); x \rangle \dfrac{dr}{r^\alpha}\sigma(dx).
\end{align*}
Next, for all $z\in \bbr^d$
\begin{align*}
\int_{(0,R)\times \mathbb{S}^{d-1}} \langle \nabla(f)(z+rx); x \rangle \dfrac{dr}{r^\alpha}\sigma(dx)=\int_0^R \left(\int_{\mathbb{S}^{d-1}} \langle \nabla(f)(z+rx); x \rangle \sigma(dx)\right) \dfrac{dr}{r^\alpha}.
\end{align*}
Set $H_z(r)= \int_{\mathbb{S}^{d-1}} f(z+rx)\sigma(dx)$, for all $r>0$ and all $z\in\bbr^d$. Moreover, for all $r>0$ 
\begin{align*}
\frac{d}{dr} \left(H_z(r)\right)=\int_{\mathbb{S}^{d-1}} \langle \nabla(f)(z+rx); x \rangle \sigma(dx).
\end{align*}
Thus,
\begin{align*}
\int_{(0,R)\times \mathbb{S}^{d-1}} \langle \nabla(f)(z+rx); x \rangle \dfrac{dr}{r^\alpha}\sigma(dx)=\int_{(0,R)} \frac{d}{dr} \left(H_z(r)-H_z(0)\right) \dfrac{dr}{r^\alpha}.
\end{align*}
A standard integration by parts argument, combined with $\alpha\in (0,1)$, implies that
\begin{align*}
\int_{(0,R)\times \mathbb{S}^{d-1}} \langle \nabla(f)(z+rx); x \rangle \dfrac{dr}{r^\alpha}\sigma(dx)&=\dfrac{(H_z(R)-H_z(0))}{R^\alpha}+\alpha \int_{0}^R \left(H_z(r)-H_z(0)\right) \dfrac{dr}{r^{\alpha+1}}\\
&=\dfrac{(H_z(R)-H_z(0))}{R^\alpha}+\alpha \int_{\bbr^d} \left(f(z+u)-f(z)\right) \nu_R(du).
\end{align*}
Next, integrating with respect to the law of $X_R$, one gets that
\begin{align*}
\bbe \int_{\bbr^d} \langle \nabla(f)(X_R+u); u \rangle \nu_R(du)&= \frac{1}{R^\alpha}\bbe \int_{\mathbb{S}^{d-1}} \left(f(X_R+Rx)-f(X_R)\right)\sigma(dx)\\
&\quad\quad+\alpha \bbe\int_{\bbr^d}  \left(f(X_R+u)-f(X_R)\right) \nu_R(du).
\end{align*}
Again, since $\alpha\in (0,1)$, $f\in \mathcal{S}(\bbr^d)$ and $\sigma\left(\mathbb{S}^{d-1}\right)<+\infty$,
\begin{align*}
\underset{R\rightarrow +\infty}{\lim} \frac{1}{R^\alpha}\bbe \int_{\mathbb{S}^{d-1}} \left(f(X_R+Rx)-f(X_R)\right)\sigma(dx)=0.
\end{align*}
Finally, to conclude the direct implication, one needs to prove that
\begin{align*}
\underset{R\rightarrow +\infty}{\lim} \bbe\int_{\bbr^d}  \left(f(X_R+u)-f(X_R)\right) \nu_R(du)=\bbe\int_{\bbr^d}  \left(f(X+u)-f(X)\right) \nu(du).
\end{align*}
To this end, for all $R>1$ and all $z\in \bbr^d$, set $F_R(z)=\int_{\bbr^d} \left(f(z+u)-f(z)\right) \nu_R(du)$ and $F(z)=\int_{\bbr^d} \left(f(z+u)-f(z)\right) \nu(du)$. Since $\alpha\in (0,1)$ and $f\in \mathcal{S}(\bbr^d)$, it is clear that both functions are well-defined, bounded and continuous on $\bbr^d$. Moreover, for all $R>1$ and all $z\in \bbr^d$
\begin{align*}
\left|F_R(z)-F(z) \right|&=\left| \int_{\bbr^d} \left(f(z+u)-f(z)\right) \bbone_{\{\|u\|\geq R\}}\nu(du)\right|\leq 2 \|f \|_\infty \int_{\|u\|\geq R}\nu(du).
\end{align*}
Thus, $F_R$ converges uniformly on $\bbr^d$ towards $F$. Finally, since $X_R$ converges in distribution to $X$, 
\begin{align*}
\underset{R\rightarrow+\infty}{\lim} \bbe \int_{\bbr^d} \left(f(X_R+u)-f(X_R)\right) \nu_R(du)=\bbe \int_{\bbr^d} \left(f(X+u)-f(X)\right) \nu(du),
\end{align*}
which concludes the first part of the proof.
To prove the converse implication, let us assume that, for all $f\in\mathcal{S}(\bbr^d)$,
\begin{align}\label{eq:SteinStableM}
\bbe \langle X; \nabla(f)(X)\rangle=\bbe \langle b_0; \nabla(f)(X)\rangle+\alpha \int_{\bbr^d} \left(f(X+u)-f(X)\right)\nu(du).
\end{align}
Denoting $\varphi_X$ the characteristic function of $X$, the equality \eqref{eq:SteinStableM} can be rewritten as 
\begin{align*}
\int_{\bbr^d} \mathcal{F}(\langle x;\nabla(f)\rangle)(\xi) \varphi_X(\xi) d\xi&= \int_{\bbr^d} \mathcal{F}(\langle b_0;\nabla(f)\rangle)(\xi) \varphi_X(\xi) d\xi\\
&\quad+\alpha \int_{\bbr^d} \mathcal{F}(f)(\xi)\left(\int_{\bbr^d} \left(e^{i \langle u; \xi \rangle}-1\right) \nu(du)\right)\varphi_X(\xi) d\xi.
\end{align*}
Using standard Fourier arguments and the fact that $f\in \mathcal{S}(\bbr^d)$,
\begin{align*}
\int_{\bbr^d} \mathcal{F}(f)(\xi)\langle \xi ; \nabla(\varphi_X)(\xi)\rangle d\xi=  \int_{\bbr^d} \mathcal{F}(f)(\xi) \left(i\langle b_0;\xi \rangle+\alpha\int_{\bbr^d} \left(e^{i \langle u;\xi\rangle }-1\right) \nu(du)\right)\varphi_X(\xi) d\xi,
\end{align*}
where the left-hand side has to be understood as a duality bracket between the Schwartz function $\mathcal{F}(f)$ and the tempered distribution $\langle \xi ; \nabla(\varphi_X)\rangle$. 
Since $\varphi_X$ is continuous on $\bbr^d$, for all $\xi \in \bbr^d$ with $\xi\ne 0$
\begin{align*}
\langle \xi ; \nabla(\varphi_X)(\xi)\rangle 
=\left(i\langle b_0;\xi \rangle+\alpha\int_{\bbr^d} \left(e^{i \langle u;\xi\rangle }-1\right) \nu(du)\right) \varphi_X(\xi).
\end{align*}
Moreover, $\varphi_X(0)=1$. Now, in order to solve the previous linear partial differential equation of order one, let us change the coordinates system $(\xi_1,\dots,\xi_d)$ into the hyper-spherical one $(r,\theta_1,\dots, \theta_{d-1})$ where $r>0$, $\theta_i\in [0, \pi]$, for all $i \in \{1,\dots,d-2\}$ and $\theta_{d-1}\in [0, 2\pi)$. Noting that
\begin{align*}
\sum_{i=1}^d \xi_i \dfrac{\partial \theta_j}{\partial \xi_i}=0, \quad j\in \{1,\dots,d-1\},
\end{align*}
and using the scaling property of the L\'evy measure $\nu$, i.e., \eqref{transform}, one gets
\begin{align*}
r \partial_r \left(\varphi_X\right)(r x) =\left(r i\langle b_0;x \rangle+\alpha r^\alpha \int_{\bbr^d} \left(e^{i \langle u; x\rangle }-1\right) \nu(du)\right) \varphi_X(r x),\quad r>0, x\in \mathbb{S}^{d-1}.
\end{align*}
For any fixed $x\in \mathbb{S}^{d-1}$, this linear differential equation admits a unique solution which is given by
\begin{align*}
\varphi_X(rx)= \exp\left( (i\langle b_0, rx\rangle+ \int_{\bbr^d} \left(e^{i \langle u; rx \rangle}-1\right)\nu(du)\right),\quad r>0,
\end{align*}
since $\varphi_X(0)=1$. Then, $X$ is a stable random vector in $\bbr^d$ with parameter $b$, stability index $\alpha$ and L\'evy measure $\nu$.
\end{proof}
\noindent
This ensuing result deals with the Cauchy case.

\begin{thm}\label{thm:caracSteinStable3}
Let $X$ be a random vector in $\bbr^d$. Let $b\in \bbr^d$ and let $\nu$ be a L\'evy measure on $\bbr^d$ such that, for all $c>0$
\begin{align*}
\nu(du)=c^{-1} T_c(\nu)(du).
\end{align*}
Moreover, let $\sigma$, the spherical part of $\nu$, be such that
\begin{align*}
\int_{\mathbb{S}^{d-1}} x \sigma(dx)=0.
\end{align*}
Then,
\begin{align}\label{eq:caraccauchy}
\bbe \langle X; \nabla(f)(X)\rangle&=\bbe \langle b; \nabla(f)(X)\rangle
+ \int_{\bbr^d} \left(f(X+u)-f(X)-\langle \nabla(f)(X);u\rangle\bbone_{\|u\|\leq 1} \right)\nu(du),
\end{align}
for all $f\in\mathcal{S}(\bbr^d)$ if and only if $X$ is a stable random vector in $\bbr^d$ with parameter $b$, stability index $\alpha=1$ and L\'evy measure $\nu$.
\end{thm}

\begin{proof}
The proof is similar to the one of Theorem~\ref{thm:caracSteinStable2}. The direct part goes with a double truncation procedure together with an integration by parts and, then, passing to the limit. Let us first assume that $X$ is stable with parameter $b$, stability index $\alpha=1$, L\'evy measure $\nu$ and $\sigma$ the spherical part. Then, \cite[Theorem $14.3$, (ii)]{S},
\begin{align*}
\nu(du)=\bbone_{(0,+\infty)}(r) \bbone_{\mathbb{S}^{d-1}}(x) \frac{dr}{r^{2}} \sigma(dx).
\end{align*}
Let $R>1$ be a truncation parameter, let
\begin{align*}
\nu_R(du):=\bbone_{(\frac{1}{R},R)}(r)\bbone_{\mathbb{S}^{d-1}}(x) \dfrac{dr}{r^{2}}\sigma(dx),
\end{align*}
and, let $X_R$ be the ID random vector defined through its characteristic function by
\begin{align*}
\varphi_R(\xi):=\exp\left(i\langle \xi; b \rangle+\int_{\bbr^d}\left(e^{i \langle \xi; u \rangle}-1-i\langle \xi; u\rangle\bbone_{\|u\|\leq 1}\right)\nu_R(du)\right), \quad \xi\in \bbr^d.
\end{align*}
Note, in particular, that $X_R$ is such that $\bbe \|X_R\|<+\infty$. Then, by Theorem \ref{thm:multid}, for all $g\in\mathcal{S}(\bbr^d)$,
\begin{align*}
\bbe X_R g(X_R)=b \bbe g(X_R)+\bbe\int_{\bbr^d} \left(g(X_R+u)-g(X_R)\bbone_{\|u\|\leq 1}\right)u \nu_R(du).
\end{align*}
Performing computations similar to those in the proof of Theorem~\ref{thm:caracSteinStable2}, for all $f\in\mathcal{S}\left(\bbr^d\right)$,
\begin{align}\label{eq:direct}
\bbe \langle X_R; \nabla(f)(X_R) \rangle=\langle b ; \bbe \nabla (f)(X_R)\rangle+\bbe \int_{\bbr^d} \langle \nabla(f)(X_R+u)-\nabla(f)(X_R)\bbone_{\|u\|\leq 1}; u \rangle \nu_R(du).
\end{align}
Now, since $X_R$ converges in distribution towards $X$, as $R$ tends to $+\infty$,
\begin{align}\label{lim:1}
\underset{R\rightarrow +\infty}{\lim} \bbe \langle X_R; \nabla(f)(X_R) \rangle =  \bbe \langle X; \nabla(f)(X) \rangle,\quad \underset{R\rightarrow +\infty}{\lim} \langle b ; \bbe \nabla (f)(X_R)\rangle = \langle b ; \bbe \nabla (f)(X)\rangle.
\end{align}
Next, let us study the second term on the right-hand side of \eqref{eq:direct}. First, since $R>1$, 
\begin{align*}
\bbe \int_{\bbr^d} &\langle \nabla(f)(X_R+u)-\nabla(f)(X_R)\bbone_{\|u\|\leq 1}; u \rangle \nu_R(du)\\
&=\bbe \int_{\|u\|\leq 1} \langle \nabla(f)(X_R+u)-\nabla(f)(X_R); u \rangle \nu_R(du)+\int_{\|u\|\geq 1}\langle \nabla(f)(X_R+u); u \rangle \nu_R(du)\\
&=-\bbe \int_{\|u\|\leq 1} \langle \nabla(f)(X_R); u \rangle \nu_R(du)+\int_{\bbr^d}\langle \nabla(f)(X_R+u); u \rangle \nu_R(du).
\end{align*}
From the polar decomposition of the L\'evy measure $\nu_R$,
\begin{align*}
\bbe \int_{\bbr^d} \langle \nabla(f)(X_R+u); u \rangle \nu_R(du)=\bbe \int_{(\frac{1}{R},R)\times \mathbb{S}^{d-1}} \langle \nabla(f)(X_R+rx); x \rangle \dfrac{dr}{r}\sigma(dx).
\end{align*}
Then, for all $z\in \bbr^d$,
\begin{align*}
\int_{(\frac{1}{R},R)\times \mathbb{S}^{d-1}} \langle \nabla(f)(z+rx); x \rangle \dfrac{dr}{r}\sigma(dx)=\int_{\frac{1}{R}}^R \left(\int_{\mathbb{S}^{d-1}} \langle \nabla(f)(z+rx); x \rangle \sigma(dx)\right) \dfrac{dr}{r}.
\end{align*}
Setting $H_z(r)= \int_{\mathbb{S}^{d-1}} f(z+rx)\sigma(dx)$, for all $r>0$ and all $z\in \bbr^d$, it follows that 
\begin{align*}
\frac{d}{dr} \left(H_z(r)\right)=\int_{\mathbb{S}^{d-1}} \langle \nabla(f)(z+rx); x \rangle \sigma(dx).
\end{align*}
Thus,
\begin{align*}
\int_{(\frac{1}{R},R)\times \mathbb{S}^{d-1}} \langle \nabla(f)(z+rx); x \rangle \dfrac{dr}{r}\sigma(dx)=\int_{(\frac{1}{R},R)} \frac{d}{dr} \left(H_z(r)-H_z\left(0\right)\right) \dfrac{dr}{r}.
\end{align*}
A standard integration by parts argument implies that
\begin{align*}
\int_{(\frac{1}{R},R)\times \mathbb{S}^{d-1}} \langle \nabla(f)(z+rx); x \rangle \dfrac{dr}{r}\sigma(dx)&=\dfrac{\left(H_z(R)-H_z\left(0\right)\right)}{R}-R\left(H_z\left(\frac{1}{R}\right)-H_z\left(0\right)\right)\\
&\quad\quad+\int_{\frac{1}{R}}^R \left(H_z(r)-H_z\left(0\right)\right) \dfrac{dr}{r^{2}}\\
&=\dfrac{\left(H_z(R)-H_z\left(0\right)\right)}{R}-R\left(H_z\left(\frac{1}{R}\right)-H_z\left(0\right)\right)\\
&\quad\quad+\int_{(\frac{1}{R},R)\times \mathbb{S}^{d-1}} \left(f(z+rx)-f\left(z\right)\right)\frac{dr}{r^2}\sigma(dx).
\end{align*}
Integrating with respect to the law of $X_R$, one gets
\begin{align*}
\bbe \int_{\bbr^d} \langle \nabla(f)(X_R+u); u \rangle \nu_R(du)&= \frac{1}{R}\bbe \int_{\mathbb{S}^{d-1}} \left(f(X_R+Rx)-f\left(X_R\right)\right)\sigma(dx)\\
&\quad\quad- R\,\bbe \int_{\mathbb{S}^{d-1}} \left(f\left(X_R+\frac{x}{R}\right)-f\left(X_R\right)\right)\sigma(dx)\\
&\quad\quad+\bbe\int_{(\frac{1}{R},R)\times \mathbb{S}^{d-1}} \left(f(X_R+rx)-f\left(X_R\right)\right) \frac{dr}{r^2}\sigma(dx).
\end{align*}
Then, since $f\in \mathcal{S}(\bbr^d)$ and $\sigma\left(\mathbb{S}^{d-1}\right)<+\infty$,
\begin{align*}
\underset{R\rightarrow +\infty}{\lim} \frac{1}{R}\bbe \int_{\mathbb{S}^{d-1}} \left(f(X_R+Rx)-f(X_R)\right)\sigma(dx)=0.
\end{align*}
Moreover, 
\begin{align*}
\underset{R\rightarrow +\infty}{\lim} R\bbe \int_{\mathbb{S}^{d-1}} \left(f\left(X_R+\frac{x}{R}\right)-f\left(X_R\right)\right)\sigma(dx)&=\bbe \int_{\mathbb{S}^{d-1}} \langle \nabla(f)(X);x\rangle \sigma(dx)=0.
\end{align*}
Let us now study the convergence, as $R\rightarrow +\infty$, of
\begin{align*}
\bbe\int_{(\frac{1}{R},R)\times \mathbb{S}^{d-1}} \left(f(X_R+rx)-f\left(X_R\right)- \langle \nabla(f)(X_R); rx \rangle\bbone_{r\leq 1} \right) \frac{dr}{r^2}\sigma(dx).
\end{align*}
To this end, let $F_R$ and $F$ be the bounded and continuous functions on $\mathbb{R}^d$ respectively defined, by
\begin{align*}
F_R(z)=\int_{(\frac{1}{R},R)\times \mathbb{S}^{d-1}} \left(f(z+rx)-f\left(z\right)- \langle \nabla(f)(z); rx \rangle\bbone_{r\leq 1} \right) \frac{dr}{r^2}\sigma(dx),\quad z\in\bbr^d,
\end{align*}
and by
\begin{align*}
F(z)=\int_{\bbr^d} \left(f(z+u)-f\left(z\right)- \langle \nabla(f)(z); u \rangle\bbone_{\|u\|\leq 1} \right) \nu(du), \quad z\in\bbr^d.
\end{align*}
Now, note that, for all $z\in \bbr^d$ and all $R>1$,
\begin{align*}
F(z)-F_R(z)= I+II,
\end{align*}
where,
\begin{align*}
I&:=\int_{\bbr^d} \left(f(z+u)-f\left(z\right)- \langle \nabla(f)(z); u \rangle\bbone_{\|u\|\leq 1} \right)\bbone_{0<\|u\|\leq \frac{1}{R}} \nu(du),\\
II&:= \int_{\bbr^d} \left(f(z+u)-f\left(z\right) \right)\bbone_{\|u\|\geq R} \nu(du).
\end{align*}
Then, by standard inequalities, since $f\in\mathcal{S}(\bbr^d)$ and $\sigma(\mathbb{S}^{d-1})<+\infty$, 
\begin{align*}
\left|I\right|\leq \dfrac{\sigma(\mathbb{S}^{d-1})}{2R}M_2(f),\quad
\left|II\right|\leq \frac{2}{R} \|f\|_{\infty} \sigma(\mathbb{S}^{d-1}),
\end{align*}
which implies that $F_R$ converges uniformly to $F$, as $R\rightarrow +\infty$. Thus,
\begin{align*}
\underset{R\rightarrow +\infty}{\lim}\bbe F_R(X_R)=\bbe F(X),
\end{align*}
and also
\begin{align}\label{lim:3}
\underset{R\rightarrow +\infty}{\lim}&\,\bbe \int_{\bbr^d} \langle \nabla(f)(X_R+u)-\nabla(f)(X_R)\bbone_{\|u\|\leq 1}; u \rangle \nu_R(du)=\nonumber\\
&\bbe\, \int_{\bbr^d} \bigg(f(X+u)-f\left(X\right)- \langle \nabla(f)(X); u \rangle\bbone_{\|u\|\leq 1} \bigg) \nu(du).
\end{align}
Combining \eqref{eq:direct}, \eqref{lim:1} and \eqref{lim:3}, one obtains
\begin{align*}
 \bbe \langle X; \nabla(f)(X) \rangle&= \langle b ; \bbe \nabla (f)(X)\rangle+\bbe\, \int_{\bbr^d} \bigg(f(X+u)-f\left(X\right)- \langle \nabla(f)(X); u \rangle\bbone_{\|u\|\leq 1} \bigg) \nu(du),
\end{align*}
which is the direct part of the theorem.
To prove the converse, assume that, for all $f\in\mathcal{S}(\bbr^d)$,\small{
\begin{align}\label{eq:SteinCauchy}
\bbe \langle X; \nabla(f)(X) \rangle&= \langle b ; \bbe \nabla (f)(X)\rangle+\bbe\, \int_{\bbr^d} \bigg(f(X+u)-f\left(X\right)- \langle \nabla(f)(X); u \rangle\bbone_{\|u\|\leq 1} \bigg) \nu(du).
\end{align}
}
Denoting by  $\varphi_X$ the characteristic function of $X$, the identity \eqref{eq:SteinCauchy} 
can then be rewritten as 
\begin{align*}
\int_{\bbr^d} \mathcal{F}(\langle x;\nabla(f)\rangle)(\xi) \varphi_X(\xi) d\xi&= \int_{\bbr^d} \mathcal{F}(\langle b;\nabla(f)\rangle)(\xi) \varphi_X(\xi) d\xi\\
&\quad+ \int_{\bbr^d} \mathcal{F}(f)(\xi)\left(\int_{\bbr^d} \left(e^{i \langle u; \xi \rangle}-1-i\langle\xi;u\rangle\bbone_{\|u\|\leq 1}\right) \nu(du)\right)\varphi_X(\xi) d\xi.
\end{align*}
Reasoning as in the proof of Theorem~\ref{thm:caracSteinStable2} gives, for all $r>0$ and 
all $x\in \mathbb{S}^{d-1}$, 
\begin{align*}
r\partial_{r}\left(\varphi_X\right)(rx)=\left(i\langle b;r x \rangle+\int_{\bbr^d} \left(e^{i \langle u; rx \rangle}-1-i\langle u; rx\rangle\bbone_{\|u\|\leq 1}\right) \nu(du)\right) \varphi_X(r x).
\end{align*}
To conclude, note that the previous equality can be interpreted as an ordinary differential equation in the radial variable. Its solution is given, for all $r\geq 0$ and all $x\in \mathbb{S}^{d-1}$, by 
\begin{align*}
\varphi_X(rx)= \exp\left(i \langle b; rx\rangle + \int_0^r G(R,x) dR+\int_{0}^r J(R,x) dR \right),
\end{align*}
where $G$ and $J$ are defined, for all $R>0$ and all $x\in\mathbb{S}^{d-1}$, by 
\begin{align*}
&G(R,x)= \int_{(0, R)\times \mathbb{S}^{d-1}} \left(e^{i\langle \rho y; x\rangle}-1-i\langle \rho y; x\rangle\right)\dfrac{d\rho}{\rho^2}\sigma(dy)\\
&J(R,x)=\int_{(R,+\infty)\times \mathbb{S}^{d-1}} \left(e^{i\langle \rho y; x\rangle}-1\right)\dfrac{d\rho}{\rho^2}\sigma(dy).
\end{align*}
Straightforward computations, and the fact that $\int_{\mathbb{S}^{d-1}}x \sigma(dx)=0$, 
finally imply that
\begin{align*}
\varphi_X(\xi)=\exp\left(i\langle b;\xi \rangle+\int_{\bbr^d} \left(e^{i \langle u; \xi \rangle}-1-i\langle u; \xi\rangle\bbone_{\|u\|\leq 1}\right) \nu(du)\right), \quad \xi\in \bbr^d,
\end{align*}
which concludes the proof.
\end{proof}
\noindent
\begin{rem}
The quantity $\int_{\mathbb{S}^{d-1}}x\sigma(dx)$ reflects the asymmetry of the L\'evy measure $\nu$. In case $\int_{\mathbb{S}^{d-1}}x\sigma(dx)\ne 0$, a careful inspection of the proof of Theorem \ref{thm:caracSteinStable3} reveals that the identity \eqref{eq:caraccauchy} becomes, for all $f\in \mathcal{S}(\bbr^d)$, 
\begin{align*}
\bbe \langle X; \nabla(f)(X)\rangle&=\bbe \langle b; \nabla(f)(X)\rangle
-\bbe \int_{\mathbb{S}^{d-1}} \langle \nabla(f)(X);x\rangle \sigma(dx)\\
&\quad\quad+ \int_{\bbr^d} \left(f(X+u)-f(X)-\langle \nabla(f)(X);u\rangle\bbone_{\|u\|\leq 1} \right)\nu(du).
\end{align*}
\end{rem}
\noindent
The next results provide extensions of both Theorem \ref{thm:caracSteinStable2} and Theorem \ref{thm:caracSteinStable3} to subclasses of self-decomposable distributions with regular radial part, on $(0,+\infty)$, and some specific asymptotic behaviors at the edges of $(0,+\infty)$ in any directions of $\mathbb{S}^{d-1}$.

\begin{thm}\label{thm:caracSD1}
Let $X$ be a random vector in $\bbr^d$. Let $b\in \bbr^d$, let $\nu$ be a L\'evy measure with $\int_{\|u\|\leq 1}\|u\| \nu(du)<+\infty$, and with polar decomposition
\begin{align}\label{eq:polardecSD1}
\nu(du)=\bbone_{(0,+\infty)}(r) \bbone_{\mathbb{S}^{d-1}}(x)  \frac{k_x(r)}{r}dr\sigma(dx),
\end{align}
where $\sigma$ is a finite positive measure on $\mathbb{S}^{d-1}$ and where $k_x(r)$ is a nonnegative continuous function decreasing in $r\in(0,+\infty)$, continuous in $x\in \mathbb{S}^{d-1}$ and such that 
\begin{align*}
&\underset{\varepsilon \rightarrow 0^+}{\lim} \varepsilon k_x(\varepsilon)=0,\quad \underset{R \rightarrow +\infty}{\lim} k_x(R)=0,\quad x\in \mathbb{S}^{d-1},\\
&\quad\quad \int_{0}^{+\infty} (1\wedge r) \underset{x\in \mathbb{S}^{d-1}}{\max} (k_x(r))\frac{dr}{r}<+\infty.
\end{align*}
Moreover, assume that, for all $(x_n)_{n\geq 1}\in (\mathbb{S}^{d-1})^{\mathbb{N}}$ converging to $x\in\mathbb{S}^{d-1}$,
\begin{align}\label{eq:totvar}
\underset{n\rightarrow+\infty}{\lim} \underset{R\rightarrow+\infty}{\lim} V_a^R (k_{x_n}-k_x)=0,\quad a>0.
\end{align}
Let $\tilde{\nu}$ be the positive measure on $\bbr^d$ defined by 
\begin{align*}
\tilde{\nu}(du)=\bbone_{(0,+\infty)}(r) \bbone_{\mathbb{S}^{d-1}}(x)(-dk_x(r))\sigma(dx),
\end{align*}
with,
\begin{align}\label{eq:condtilde1}
\int_{\bbr^d}(1 \wedge \|u\|) \tilde{\nu}(du)<+\infty.
\end{align}
Then,
\begin{align*}
\bbe \langle X; \nabla(f)(X)\rangle=\bbe \langle b_0; \nabla(f)(X)\rangle+ \int_{\bbr^d} \left(f(X+u)-f(X)\right)\tilde{\nu}(du),
\end{align*}
where $b_0=b-\int_{\|u\|\leq 1}u \nu(du)$, for all $f\in\mathcal{S}(\bbr^d)$ if and only if $X$ is self-decomposable with parameter $b$, $\Sigma=0$ and L\'evy measure $\nu$. 
\end{thm}

\begin{proof}
Let us start with the direct part. Let $X$ be a SD random vector of $\bbr^d$ with parameter $b$ and L\'evy measure $\nu$ such that $\int_{\|u\|\leq 1}\|u\| \nu(du)<+\infty$ and whose polar decomposition is given by \eqref{eq:polardecSD1}. Let $R>1$ and let $(\sigma_n)_{n\geq 1}$ be a sequence of positive linear combinations of Dirac measures which converges weakly to $\sigma$, the spherical component of $\nu$. Then, for all $R>1$ and all $n\geq 1$, let
\begin{align*}
\nu_{R,n}(du)=\bbone_{(0,R)}(r) \bbone_{\mathbb{S}^{d-1}}(x) \dfrac{k_x(r)}{r}dr \sigma_n(dx),
\end{align*}
and denote by $X_{R,n}$ the SD random vector with parameter $b$ and L\'evy measure $\nu_{R,n}$. Similarly, let, for all $n\geq 1$, 
\begin{align*}
\nu_{n}(du):=\bbone_{(0,+\infty)}(r) \bbone_{\mathbb{S}^{d-1}}(x) \dfrac{k_x(r)}{r}dr \sigma_n(dx),
\end{align*} 
and denote by $X_{n}$ the SD random vector with parameter $b$ and L\'evy measure $\nu_{n}$. Performing computations similar to those in the proof of Theorem \ref{thm:caracSteinStable2}, for all $f\in \mathcal{S}(\bbr^d)$, all $R>1$ and all $n\geq 1$
\begin{align*}
\bbe \langle X_{R,n}; \nabla(f)(X_{R,n}) \rangle=\langle b_0 ; \bbe \nabla (f)(X_{R,n})\rangle+\bbe \int_{\bbr^d} \langle \nabla(f)(X_{R,n}+u); u \rangle \nu_{R,n}(du).
\end{align*}
Now, since, as $R\rightarrow+\infty$, $X_{R,n}$ converges in distribution to $X_n$, for all $n\geq 1$,
\begin{align*}
\underset{R\rightarrow +\infty}{\lim} \bbe \langle X_R; \nabla(f)(X_R) \rangle =  \bbe \langle X_n; \nabla(f)(X_n) \rangle,\quad \underset{R\rightarrow +\infty}{\lim} \langle b_0 ; \bbe \nabla (f)(X_R)\rangle = \langle b_0 ; \bbe \nabla (f)(X_n)\rangle.
\end{align*}
Moreover, from the polar decomposition of the L\'evy measure $\nu_{R,n}$, \textit{mutatis mutandis},
\begin{align*}
\bbe \int_{\bbr^d} \langle \nabla(f)(X_{R,n}+u); u \rangle \nu_{R,n}(du)&= \bbe \int_{\mathbb{S}^{d-1}} \left(f(X_{R,n}+Rx)-f(X_{R,n})\right)k_x(R)\sigma_n(dx)\\
&\quad\quad+ \bbe\int_{\bbr^d}  \left(f(X_{R,n}+u)-f(X_{R,n})\right) \tilde{\nu}_{R,n}(du),
\end{align*}
where, for all $R>1$ and all $n\geq 1$,
\begin{align*}
\tilde{\nu}_{R,n}(du):=\bbone_{(0,R)}(r) \bbone_{\mathbb{S}^{d-1}}(x)\left(-dk_x(r)\right)\sigma_n(dx).
\end{align*}
Then, since $\underset{R \rightarrow +\infty}{\lim} k_x(R)=0$, $x\in \mathbb{S}^{d-1}$, $f\in \mathcal{S}(\bbr^d)$ and $\sigma_n\left(\mathbb{S}^{d-1}\right)<+\infty$, $n\geq 1$,
\begin{align*}
\underset{R\rightarrow +\infty}{\lim} \bbe \int_{\mathbb{S}^{d-1}} \left(f(X_{R,n}+Rx)-f(X_{R,n})\right)k_x(R)\sigma_n(dx)=0.
\end{align*}
Next, one needs to prove that
\begin{align*}
\underset{R\rightarrow +\infty}{\lim} \bbe\int_{\bbr^d}  \left(f(X_{R,n}+u)-f(X_{R,n})\right)  \tilde{\nu}_{R,n}(du)=\bbe\int_{\bbr^d}  \left(f(X_n+u)-f(X_n)\right) \tilde{\nu}_n(du),
\end{align*}
where $\tilde{\nu}_{n}$ is given, for all $R>1$ and all $n\geq 1$, by
\begin{align*}
\tilde{\nu}_{n}(du)=\bbone_{(0,+\infty)}(r) \bbone_{\mathbb{S}^{d-1}}(x)(-dk_x(r))\sigma_n(dx).
\end{align*}
To this end, for all $R>1$, all $n\geq 1$ and all $z\in \bbr^d$, set $F_{R,n}(z)=\int_{\bbr^d} \left(f(z+u)-f(z)\right) \tilde{\nu}_{R,n}(du)$ and $F_n(z)=\int_{\bbr^d} \left(f(z+u)-f(z)\right)\tilde{\nu}_n(du)$. From \eqref{eq:condtilde1}, and since $f\in \mathcal{S}(\bbr^d)$, it is clear that both functions are well-defined, bounded and continuous on $\bbr^d$. Moreover, 
\begin{align*}
\left|F_{R,n}(z)-F_n(z) \right|&=\left| \int_{\bbr^d} \left(f(z+u)-f(z)\right) \bbone_{\{\|u\|\geq R\}}\tilde{\nu}_n(du)\right|\leq 2 \|f \|_\infty \int_{\|u\|\geq R}\tilde{\nu}_n(du).
\end{align*}
Thus, as $R$ tends to $+\infty$, $F_{R,n}$ converges to $F_n$ uniformly on $\bbr^d$, for all $n\geq 1$. Finally, since $X_{R,n}$ converges in distribution to $X_n$, for all $n\geq 1$,
\begin{align*}
\underset{R\rightarrow+\infty}{\lim} \bbe \int_{\bbr^d} \left(f(X_{R,n}+u)-f(X_{R,n})\right) \tilde{\nu}_{R,n}(du)=\bbe \int_{\bbr^d} \left(f(X_n+u)-f(X_n)\right) \tilde{\nu}_n(du).
\end{align*}
Then, for all $n\geq 1$
\begin{align*}
\bbe \langle X_n; \nabla(f)(X_n)\rangle=\bbe \langle b_0; \nabla(f)(X_n)\rangle+\bbe \int_{\bbr^d} \left(f(X_n+u)-f(X_n)\right)\tilde{\nu}_n(du).
\end{align*}
Now, observe that, $(X_n)_{n\geq 1}$ converges in distribution to $X$ since $(\sigma_n)_{n\geq 1}$ converges weakly to $\sigma$ and since $ \int_{0}^{+\infty} (1\wedge r) \underset{x\in \mathbb{S}^{d-1}}{\max} k_x(r)dr/r<+\infty$. Hence,
\begin{align*}
\underset{n\rightarrow +\infty}{\lim}\bbe \langle X_n; \nabla(f)(X_n)\rangle=\bbe \langle X; \nabla(f)(X)\rangle,\quad \underset{n\rightarrow +\infty}{\lim}\bbe \langle b_0; \nabla(f)(X_n)\rangle=\bbe \langle b_0; \nabla(f)(X)\rangle.
\end{align*}
To conclude the proof of the direct part of the theorem, let us study the convergence of:
\begin{align*}
\bbe \int_{\bbr^d} \left(f(X_n+u)-f(X_n)\right) \tilde{\nu}_n(du).
\end{align*}
Since $f\in \mathcal{S}(\bbr^d)$, for all $n\geq 1$,
\begin{align*}
\bbe \int_{\bbr^d} \left(f(X_n+u)-f(X_n)\right) \tilde{\nu}_n(du)=\frac{1}{(2\pi)^d}\int_{\bbr^d} \mathcal{F}(f)(\xi) \varphi_n(\xi) \left(\int_{\bbr^d} (e^{i \langle u;\xi \rangle}-1)\tilde{\nu}_n(du)\right)d\xi.
\end{align*}
Now, since $(X_n)_{n\geq 1}$ converges in distribution to $X$, then $\underset{n\rightarrow+\infty}{\lim} \varphi_n(\xi)=\varphi(\xi)$, for all $\xi \in \bbr^d$. In turn, let us prove the following:
\begin{align}\label{eq:weak}
\underset{n\rightarrow+\infty}{\lim}\int_{\bbr^d} (e^{i \langle u;\xi \rangle}-1)\tilde{\nu}_n(du)=\int_{\bbr^d} (e^{i \langle u;\xi \rangle}-1)\tilde{\nu}(du).
\end{align}
Observe that, for all $\xi \in \bbr^d$ and all $n\geq 1$
\begin{align*}
\int_{\bbr^d} (e^{i \langle u;\xi \rangle}-1)\tilde{\nu}_n(du)=\int_{\mathbb{S}^{d-1}} \sigma_n(dx) \left(\int_0^{+\infty}\left(e^{i \langle rx;\xi \rangle}-1\right)(-dk_x(r))\right).
\end{align*}
Since $(\sigma_n)_{n\geq 1}$ converges weakly to $\sigma$, let us prove that the function $H(x,\xi)=\int_0^{+\infty}\left(e^{i \langle rx;\xi \rangle}-1\right)dk_x(r)$ is continuous in $x\in \mathbb{S}^{d-1}$, for all $\xi \in \bbr^d$. Let $(x_n)_{n\geq 1}$ be a sequence of elements of $\mathbb{S}^{d-1}$ converging to $x\in \mathbb{S}^{d-1}$. Then, consider, for all $n\geq 1$ and all $\xi \in \bbr^d$
\begin{align}\label{eq:cont1}
H(x,\xi)-H(x_n,\xi)&=\int_0^{+\infty} \left(e^{i \langle rx;\xi \rangle}-1\right)dk_x(r)-\int_0^{+\infty} \left(e^{i\langle rx_n;\xi \rangle}-1\right)dk_{x_n}(r),\nonumber\\
&=\int_0^{+\infty} \left(e^{i \langle rx_n;\xi \rangle}-1\right)d\left(k_x-k_{x_n}\right)(r)+\int_0^{+\infty} \left(e^{i \langle rx;\xi \rangle}-e^{i \langle rx_n;\xi \rangle}\right) dk_x(r).
\end{align}
The second term on the right-hand side of \eqref{eq:cont1} converges to $0$ as $n$ tends to $+\infty$, by the Lebesgue dominated convergence theorem since $\int_{0}^{+\infty} (1\wedge r) dk_x(r)<+\infty$. For the first term of \eqref{eq:cont1}, observe that
\begin{align}\label{eq:cont2}
\int_0^{+\infty} \left(e^{i \langle rx_n;\xi \rangle}-1\right)d\left(k_x-k_{x_n}\right)(r)&=\int_0^{1} \left(e^{i \langle rx_n;\xi \rangle}-1\right)d\left(k_x-k_{x_n}\right)(r)\nonumber\\
&\quad\quad +\int_1^{+\infty} \left(e^{i \langle rx_n;\xi \rangle}-1\right)d\left(k_x-k_{x_n}\right)(r).
\end{align}
For the second term on the right-hand side of \eqref{eq:cont2}, for all $n\geq 1$,
\begin{align*}
\left|\int_1^{+\infty} \left(e^{i \langle rx_n;\xi \rangle}-1\right)d\left(k_x-k_{x_n}\right)(r)\right| \leq 2 \underset{R\rightarrow+\infty}{\lim} V_1^R(k_{x}-k_{x_n}),
\end{align*}
so that by \eqref{eq:totvar}, this term converges to $0$. Finally, integrating by parts, for all $n\geq 1$,
\begin{align}\label{eq:intbp}
\int_0^{1} \left(e^{i \langle rx_n;\xi \rangle}-1\right)d\left(k_x-k_{x_n}\right)(r)&=-i\langle  x_n;\xi\rangle\int_0^{1} e^{i \langle rx_n;\xi \rangle} (k_x(r)-k_{x_n}(r))dr\nonumber\\
&\quad\quad+\left(e^{i \langle x_n;\xi \rangle}-1\right)(k_x(1)-k_{x_n}(1)).
\end{align}
Now, the second term on the right-hand side of \eqref{eq:intbp} converges to $0$, as $n$ tends to $+\infty$ and, by the Lebesgue dominated convergence theorem, the first term does converges to $0$, as $n$ tends to $+\infty$. This proves that $\underset{n\rightarrow +\infty}{\lim}H(x_n,\xi)=H(x,\xi)$, for all $\xi \in \bbr^d$, so that \eqref{eq:weak} is indeed verified.
To prove the converse part, assume that, for all $f\in\mathcal{S}(\bbr^d)$,
\begin{align}\label{eq:SteinSd}
\bbe \langle X; \nabla(f)(X)\rangle=\bbe \langle b_0; \nabla(f)(X)\rangle+ \int_{\bbr^d} \left(f(X+u)-f(X)\right)\tilde{\nu}(du).
\end{align}
Now, reasoning as in the second part of the proof of Theorem \ref{thm:caracSteinStable2}, 
\begin{align}\label{eq:radeq}
r \partial_r \left(\varphi_X\right)(r x) =\left(r i\langle b_0;x \rangle+\int_{\bbr^d} \left(e^{i \langle u; rx\rangle }-1\right) \tilde{\nu}(du)\right) \varphi_X(r x),\quad r>0, x\in \mathbb{S}^{d-1}.
\end{align}
Let us develop the second term inside the above parenthesis a bit more. First,
\begin{align*}
\int_{\bbr^d} \left(e^{i \langle u; rx\rangle }-1\right) \tilde{\nu}(du)&=\int_{(0,+\infty)\times \mathbb{S}^{d-1}}\left(e^{i \langle \rho y; rx\rangle }-1\right)(-dk_y(\rho))\sigma(dy),\\
&=\int_{(0,+\infty)\times \mathbb{S}^{d-1}}\left(e^{i \langle \rho y; x\rangle }-1\right)\left(-dk_y\left(\frac{\rho}{r}\right)\right)\sigma(dy).
\end{align*}
The radial equation \eqref{eq:radeq} then becomes, for all $r>0$ and all $x\in \mathbb{S}^{d-1}$,
\begin{align*}
 \partial_r \left(\varphi_X\right)(r x) =\left( i\langle b_0;x \rangle+\int_{(0,+\infty)\times \mathbb{S}^{d-1}}\left(e^{i \langle \rho y; x\rangle }-1\right) \frac{1}{r}\left(-dk_y\left(\frac{\rho}{r}\right)\right)\sigma(dy)\right) \varphi_X(r x).
\end{align*}
For any fixed $x\in \mathbb{S}^{d-1}$, this linear differential equation admits a unique solution which is given by
\begin{align*}
\varphi_X(rx)= \exp\left( i\langle b_0, rx\rangle+ \int_{\bbr^d} \left(e^{i \langle u; rx \rangle}-1\right)\nu(du)\right),\quad r>0,
\end{align*}
since $\varphi_X(0)=1$. Then, $X$ is a SD random vector with parameter $b$ and L\'evy measure $\nu$.
\end{proof}
\noindent
The next result is the SD pendant of the Cauchy characterization obtained in Theorem \ref{thm:caracSteinStable3}.

\begin{thm}\label{thm:caracSD2}
Let $X$ be a random vector in $\bbr^d$. Let $b\in \bbr^d$ and let $\nu$ be a L\'evy measure on $\bbr^d$ with polar decomposition
\begin{align*}
\nu(du)=\bbone_{(0,+\infty)}(r) \bbone_{\mathbb{S}^{d-1}}(x)  \frac{k_x(r)}{r}dr\sigma(dx),
\end{align*}
where $\sigma$ is a finite positive measure on $\mathbb{S}^{d-1}$ and where $k_x(r)$ is a nonnegative continuous function decreasing in $r\in(0,+\infty)$, continuous in $x\in \mathbb{S}^{d-1}$, and such that 
\begin{align*}
&\underset{\varepsilon \rightarrow 0^+}{\lim} \varepsilon k_x(\varepsilon)=k_x(1),\quad \underset{R \rightarrow +\infty}{\lim} k_x(R)=0, \quad x\in \mathbb{S}^{d-1},\\
&\quad\quad \int_{0}^{+\infty} (1\wedge r^2) \underset{x\in \mathbb{S}^{d-1}}{\max} (k_x(r))\frac{dr}{r}<+\infty.
\end{align*}
Moreover, assume that, for all $(x_n)_{n\geq 1}\in (\mathbb{S}^{d-1})^{\mathbb{N}}$ converging to $x\in\mathbb{S}^{d-1}$,
\begin{align}\label{eq:totvar2}
\underset{n\rightarrow+\infty}{\lim} \underset{R\rightarrow+\infty}{\lim} V_a^R (k_{x_n}-k_x)=0,\quad a>0.
\end{align}
Let $\tilde{\nu}$ be the positive measure on $\bbr^d$ defined by 
\begin{align*}
\tilde{\nu}(du)=\bbone_{(0,+\infty)}(r) \bbone_{\mathbb{S}^{d-1}}(x)(-dk_x(r))\sigma(dx),
\end{align*}
with,
\begin{align}\label{eq:condtilde2}
\int_{\bbr^d}(1 \wedge \|u\|^2) \tilde{\nu}(du)<+\infty.
\end{align}
Then,
\begin{align}\label{eq:caracSD2}
\bbe \langle X; \nabla(f)(X)\rangle&=\bbe \langle b; \nabla(f)(X)\rangle- \bbe \int_{\mathbb{S}^{d-1}} \langle \nabla(f)(X); x \rangle k_x(1) \sigma(dx)\nonumber\\
&\quad +\bbe \int_{\bbr^d} \left(f(X+u)-f(X)-\langle \nabla (f)(X); u \rangle\bbone_{\|u\|\leq 1}\right)\tilde{\nu}(du),
\end{align}
for all $f\in\mathcal{S}(\bbr^d)$ if and only if $X$ is self-decomposable with parameter $b$, $\Sigma=0$ and L\'evy measure $\nu$. 
\end{thm}

\begin{proof}
The proof is a direct extension of the proof of Theorem \ref{thm:caracSteinStable3} so that it is only outlined by highlighting the main differences. Let us start with the direct part. Let $X$ be a SD random vector with parameter $b$ and L\'evy measure $\nu$. Let $R>1$ and let $(\sigma_n)_{n\geq 1}$ be a sequence of positive linear combinations of Dirac measures converging weakly to $\sigma$, the spherical component of $\nu$. Then, for all $R>1$ and all $n\geq 1$, let
\begin{align*}
\nu_{R,n}(du):=\bbone_{(1/R,R)}(r) \bbone_{\mathbb{S}^{d-1}}(x) \dfrac{k_x(r)}{r}dr \sigma_n(dx),
\end{align*}
and denote by $X_{R,n}$ the SD random vector with parameter $b$ and L\'evy measure $\nu_{R,n}$. Similarly, for all $n\geq 1$, let
\begin{align*}
\nu_{n}(du):=\bbone_{(0,+\infty)}(r) \bbone_{\mathbb{S}^{d-1}}(x) \dfrac{k_x(r)}{r}dr \sigma_n(dx),
\end{align*} 
and denote by $X_{n}$ the SD random vector with parameter $b$ and L\'evy measure $\nu_{n}$. As in the proof of Theorem \ref{thm:caracSteinStable3}, for all $f\in\mathcal{S}\left(\bbr^d\right)$ and all $R>1$,
\begin{align*}
\bbe \langle X_{R,n}; \nabla(f)(X_{R,n}) \rangle=\langle b ; \bbe \nabla (f)(X_{R,n})\rangle+\bbe \int_{\bbr^d} \langle \nabla(f)(X_{R,n}+u)-\nabla(f)(X_{R,n})\bbone_{\|u\|\leq 1}; u \rangle \nu_{R,n}(du).
\end{align*}
Now, since, as $R\rightarrow+\infty$, $X_{R,n}$ converges in distribution to $X_n$, for all $n\geq 1$
\begin{align*}
\underset{R\rightarrow +\infty}{\lim} \bbe \langle X_{R,n}; \nabla(f)(X_{R,n}) \rangle= \bbe \langle X_n; \nabla(f)(X_n) \rangle,\quad \underset{R\rightarrow +\infty}{\lim} \bbe \langle b; \nabla(f)(X_{R,n}) \rangle= \bbe \langle b; \nabla(f)(X_n) \rangle.
\end{align*}
Moreover, for all $R>1$ and all $n\geq 1$,
\begin{align*}
\bbe \int_{\bbr^d} \langle \nabla(f)(X_{R,n}+u); u \rangle \nu_{R,n}(du)&= \bbe \int_{\mathbb{S}^{d-1}} \left(f(X_R+Rx)-f\left(X_R\right)\right)k_x(R)\sigma_n(dx)\\
&\quad\quad-\bbe \int_{\mathbb{S}^{d-1}} \left(f\left(X_{R,n}+\frac{x}{R}\right)-f\left(X_{R,n}\right)\right)k_x\left(\frac{1}{R}\right)\sigma_n(dx)\\
&\quad\quad+\bbe\int_{(\frac{1}{R},R)\times \mathbb{S}^{d-1}} \left(f(X_{R,n}+rx)-f\left(X_{R,n}\right)\right) (-dk_x(r))\sigma_n(dx).
\end{align*}
From the limiting behavior of $k_x$ at $+\infty$ and at $0^+$, for all $n\geq 1$,
\begin{align*}
\underset{R\rightarrow+\infty}{\lim}\bbe \int_{\mathbb{S}^{d-1}} k_x(R)\left(f(X_{R,n}+Rx)-f\left(X_{R,n}\right)\right)\sigma_n(dx)=0,
\end{align*}
and,
\begin{align*}
\underset{R\rightarrow+\infty}{\lim}\bbe \int_{\mathbb{S}^{d-1}} \left(f\left(X_{R,n}+\frac{x}{R}\right)-f\left(X_{R,n}\right)\right)k_x\left(\frac{1}{R}\right)\sigma_n(dx)= \bbe \int_{\mathbb{S}^{d-1}} \langle \nabla(f)(X_n); x \rangle k_x(1)\sigma_n(dx).
\end{align*}
Next, consider the term defined, for all $z\in \bbr^d$ and all $n\geq 1$, by
\begin{align*}
\int_{\bbr^d}\langle \nabla(f)(z)\bbone_{\|u\|\leq 1}; u \rangle \nu_{R,n}(du)=\langle \nabla(f)(z);\int_{\mathbb{S}^{d-1}} x \left(\int_{1/R}^{1}k_x(r)dr\right) \sigma_n(dx)\rangle.
\end{align*}
By a standard integration by parts, for all $n\geq 1$,
\begin{align*}
\int_{\bbr^d}\langle \nabla(f)(z)\bbone_{\|u\|\leq 1}; u \rangle \nu_{R,n}(du)&=\langle \nabla(f)(z);\int_{\mathbb{S}^{d-1}} x\left(k_x(1)-k_x(1/R)/R\right) \sigma_n(dx)\rangle \\
&\quad\quad+\langle \nabla(f)(z);\int_{\mathbb{S}^{d-1}} x \left( \int_{1/R}^{1}r(-dk_x(r))\right)\sigma_n(dx)\rangle.
\end{align*}
Then, observe that, for all $x\in \mathbb{S}^{d-1}$, $\underset{\varepsilon\rightarrow 0^+}{\lim}\varepsilon k_x(\varepsilon)=k_x(1)$, and, for all $n\geq 1$,
\begin{align*}
\underset{R\rightarrow +\infty}{\lim} \bbe \langle \nabla(f)(X_{R,n});\int_{\mathbb{S}^{d-1}} x\left(k_x(1)-k_x(1/R)/R\right) \sigma_n(dx)\rangle =0.
\end{align*}
Finally, for all $n\geq 1$,
\begin{align*}
\underset{R\rightarrow +\infty}{\lim}\,&\bbe \int_{\bbr^d} \langle \nabla(f)(X_{R,n}+u)-\nabla(f)(X_{R,n})\bbone_{\|u\|\leq 1}; u \rangle \nu_{R,n}(du)=\\
&\bbe\, \int_{\bbr^d} \bigg(f(X_n+u)-f\left(X_n\right)- \langle \nabla(f)(X_n); u \rangle\bbone_{\|u\|\leq 1} \bigg) \tilde{\nu}_n(du)- \bbe \int_{\mathbb{S}^{d-1}} \langle \nabla(f)(X_n); x \rangle k_x(1)\sigma_n(dx),
\end{align*}
so that,
\begin{align*}
\bbe \langle X_n; \nabla(f)(X_n) \rangle&=\bbe \langle b; \nabla(f)(X_n) \rangle- \bbe \int_{\mathbb{S}^{d-1}} \langle \nabla(f)(X_n); x \rangle k_x(1)\sigma_n(dx)+\\
&\bbe\, \int_{\bbr^d} \bigg(f(X_n+u)-f\left(X_n\right)- \langle \nabla(f)(X_n); u \rangle\bbone_{\|u\|\leq 1} \bigg) \tilde{\nu}_n(du).
\end{align*}
Now, since $(\sigma_n)_{n\geq 1}$ converges weakly to $\sigma$ and since $\int_{0}^{+\infty} (1\wedge r^2) \underset{x\in \mathbb{S}^{d-1}}{\max} (k_x(r))dr/r<+\infty$, $(X_n)_{n\geq 1}$ converges in distribution to $X$. Hence, 
\begin{align*}
\underset{n\rightarrow +\infty}{\lim}\bbe \langle X_n; \nabla(f)(X_n) \rangle=\bbe \langle X; \nabla(f)(X) \rangle,\quad
\underset{n\rightarrow +\infty}{\lim}\bbe \langle b; \nabla(f)(X_n) \rangle=\bbe \langle b; \nabla(f)(X) \rangle.
\end{align*}
To conclude the direct part of the proof, let us consider the following terms:
\begin{align*}
\bbe \int_{\mathbb{S}^{d-1}} \langle \nabla(f)(X_n); x \rangle k_x(1)\sigma_n(dx), \quad \bbe\, \int_{\bbr^d} \bigg(f(X_n+u)-f\left(X_n\right)-\langle \nabla(f)(X_n); u \rangle\bbone_{\|u\|\leq 1} \bigg) \tilde{\nu}_n(du).
\end{align*}
First, for all $n\geq 1$
\begin{align*}
\bbe \int_{\mathbb{S}^{d-1}} \langle \nabla(f)(X_n); x \rangle k_x(1)\sigma_n(dx)=\int_{\bbr^d} \mathcal{F}(f)(\xi) \varphi_n(\xi) \langle i\xi ; \int_{\mathbb{S}^{d-1}}x k_x(1)\sigma_n(dx) \rangle \frac{d\xi}{(2\pi)^d}. 
\end{align*}
Since $(X_n)_{n\geq 1}$ converges in distribution to $X$, as $n$ tends to $+\infty$, $(\varphi_n(\xi))_{n\geq 1}$ converges to $\varphi(\xi)$, for all $\xi \in \bbr^d$. Moreover,
\begin{align*}
\underset{n\rightarrow+\infty}{\lim} \int_{\mathbb{S}^{d-1}} x k_x(1)\sigma_n(dx)= \int_{\mathbb{S}^{d-1}} x k_x(1)\sigma(dx).
\end{align*}
Then, by the Lebesgue dominated convergence theorem, 
\begin{align*}
\underset{n\rightarrow+\infty}{\lim}\bbe \int_{\mathbb{S}^{d-1}} \langle \nabla(f)(X_n); x \rangle k_x(1)\sigma_n(dx)=\bbe \int_{\mathbb{S}^{d-1}} \langle \nabla(f)(X); x \rangle k_x(1)\sigma(dx).
\end{align*}
Similarly, for all $n\geq 1$,
\begin{align*}
&\bbe\, \int_{\bbr^d} \bigg(f(X_n+u)-f\left(X_n\right)-\langle \nabla(f)(X_n); u \rangle\bbone_{\|u\|\leq 1} \bigg) \tilde{\nu}_n(du)\\
&\quad=\int_{\bbr^d} \mathcal{F}(f)(\xi)\varphi_n(\xi) \left(\int_{\bbr^d} (e^{i \langle u;\xi \rangle}-1-i\langle u;\xi \rangle\bbone_{\|u\|\leq 1})\tilde{\nu}_n(du)\right) \frac{d\xi}{(2\pi)^d},
\end{align*}
and proceeding as in the proof of Theorem \ref{thm:caracSD1},
\begin{align*}
&\underset{n\rightarrow+\infty}{\lim}\bbe\, \int_{\bbr^d} \bigg(f(X_n+u)-f\left(X_n\right)-\langle \nabla(f)(X_n); u \rangle\bbone_{\|u\|\leq 1} \bigg) \tilde{\nu}_n(du)\\
&\quad\quad=\bbe\, \int_{\bbr^d} \bigg(f(X+u)-f\left(X\right)-\langle \nabla(f)(X); u \rangle\bbone_{\|u\|\leq 1} \bigg) \tilde{\nu}(du).
\end{align*}
The direct part of the theorem is proved. For the converse part, \textit{mutatis mutandis}, based on \eqref{eq:caracSD2}, for all $r>0$ and all $x\in \mathbb{S}^{d-1}$
\begin{align*}
\varphi_X(rx)= \exp\left(i \langle b-\int_{\mathbb{S}^{d-1}}zk_z(1)\sigma(dz); rx\rangle + \int_0^r \tilde{G}(R,x) dR+\int_{0}^r \tilde{J}(R,x) dR \right),
\end{align*}
where $\tilde{G}$ and $\tilde{J}$ are respectively defined, for all $R>0$ and all $x\in\mathbb{S}^{d-1}$, by
\begin{align*}
&\tilde{G}(R,x)= \int_{(0, R)\times \mathbb{S}^{d-1}} \left(e^{i\langle \rho y; x\rangle}-1-i\langle \rho y; x\rangle\right) \frac{1}{R}\left(-dk_y\left(\frac{\rho}{R}\right)\right)\sigma(dy)\\
&\tilde{J}(R,x)=\int_{(R,+\infty)\times \mathbb{S}^{d-1}} \left(e^{i\langle \rho y; x\rangle}-1\right) \frac{1}{R}\left(-dk_y\left(\frac{\rho}{R}\right)\right)\sigma(dy).
\end{align*}
Finally, straightforward computations together with Fubini's Theorem and the fact that $\underset{\varepsilon \rightarrow 0^+}{\lim} \varepsilon k_x(\varepsilon)=k_x(1)$, for all $x\in \mathbb{S}^{d-1}$, concludes the proof. 
\end{proof}
\noindent
\begin{rem}\label{rem:1D}
(i) Let us recast the previous results in dimension one, i.e., for $d=1$. In this case, the L\'evy measure of a SD law is absolutely continuous with respect to the Lebesgue measure and is given by
\begin{align*}
\nu(du)=\frac{k(u)}{|u|}du,
\end{align*} 
where $k$ is a nonnegative function increasing on $(-\infty,0)$ and decreasing on $(0,+\infty)$. Now, assume, for simplicity only, that $k$ is continuously differentiable on $(-\infty,0)$ and on $(0,+\infty)$ and that 
\begin{align}\label{eq:limit1Dreg}
&\underset{\varepsilon \rightarrow 0^+}{\lim} \varepsilon k(\varepsilon)=k(1),\quad \underset{\varepsilon \rightarrow 0^{-}}{\lim} \varepsilon k(\varepsilon)=-k(-1),\quad \underset{|x| \rightarrow +\infty}{\lim} k(x)=0.
\end{align}
Theorem \ref{thm:caracSD2} gives the following characterizing identity when $X$ is a SD random variable with parameter $b\in \bbr$ and L\'evy measure $\nu$:
\begin{align*}
\bbe Xf'(X)= b\bbe f'(X)+(k(-1)-k(1))\bbe f'(X)+\bbe \int_{\bbr} \left(f(X+u)-f(X)-f'(X)u\bbone_{|u|\leq 1}\right) \tilde{\nu}(du),
\end{align*}
for all $f\in \mathcal{S}(\bbr)$, where $\tilde{\nu}(du)=(-k'(u))\bbone_{(0,+\infty)}(u)du+k'(u)\bbone_{(-\infty,0)}(u)du$. In a similar fashion, it is possible to provides a characterization result for SD random variables with L\'evy measure $\nu$ such that $\int_{|u|\leq 1}|u|\nu(du)<+\infty$ and such that $k$ is continuously differentiable on $(-\infty,0)$ and on $(0,+\infty)$ with 
\begin{align}\label{eq:limit1Dnoreg}
&\underset{\varepsilon \rightarrow 0^+}{\lim} \varepsilon k(\varepsilon)=0,\quad \underset{\varepsilon \rightarrow 0^{-}}{\lim} \varepsilon k(\varepsilon)=0,\quad \underset{|x| \rightarrow +\infty}{\lim} k(x)=0,
\end{align}
via Theorem \ref{thm:caracSD1}.\\
(ii) From \cite[Theorem $28.4$]{S}, under the assumptions that the function $k$ is continuously differentiable on $(-\infty,0)$ and on $(0,+\infty)$ and satisfies \eqref{eq:limit1Dreg} with $k(1)>0, k(-1)>0$, the quantity $c:=k(0^+)+k(0^-)$ is infinite. Then, the associated SD distribution admits a Lebesgue  density infinitely differentiable on $\bbr$. If the function $k$ is continuously differentiable on $(-\infty,0)$ and on $(0,+\infty)$ and satisfies \eqref{eq:limit1Dnoreg}, then $c$ can be either finite or infinite, implying different types of regularity for the Lebesgue density of the associated SD distribution.\\
(iii) Let $X$ be a SD random vector with L\'evy measure $\nu$ as in Theorem \ref{thm:caracSD2} and such that $\int_{\|u\|\geq 1}\|u\|\nu(du)<+\infty$. Then, integrating by parts, for all $f\in\mathcal{S}(\bbr^d)$,
\begin{align*}
\bbe \langle X; \nabla(f)(X)\rangle&=\bbe \langle b; \nabla(f)(X)\rangle- \bbe \int_{\mathbb{S}^{d-1}} \langle \nabla(f)(X); x \rangle k_x(1) \sigma(dx)\\
&\quad +\bbe \int_{\bbr^d} \left(f(X+u)-f(X)-\langle \nabla (f)(X); u \rangle\bbone_{\|u\|\leq 1}\right)\tilde{\nu}(du),\\
&=\bbe \langle \bbe X; \nabla(f)(X)\rangle- \bbe \int_{\mathbb{S}^{d-1}} \langle \nabla(f)(X); x \rangle k_x(1) \sigma(dx)\\
&\quad +\bbe \int_{\bbr^d} \left(f(X+u)-f(X)-\langle \nabla (f)(X); u \rangle\right)\tilde{\nu}(du),\\
&=\bbe \langle \bbe X; \nabla(f)(X)\rangle+\bbe \int_{\bbr^d} \langle \nabla (f)(X+u) - \nabla (f)(X); u \rangle\nu(du).
\end{align*}
\end{rem}
\noindent
Let us now present a simple example for which Theorem \ref{thm:caracSD1} and Theorem \ref{thm:caracSD2} apply and which is not covered in the relevant existing literature. Rotationally invariant self-decomposable distributions are covered by Theorem \ref{thm:caracSD1} and Theorem \ref{thm:caracSD2}. Indeed, let $\lambda$ be the uniform measure on $\mathbb{S}^{d-1}$ and let
\begin{align*}
\nu(du):=\bbone_{(0,+\infty)}(r) \bbone_{\mathbb{S}^{d-1}}(x)  \frac{k(r)}{r}dr\lambda(dx),
\end{align*}
with $\int_{\|u\|\leq 1}\|u\| \nu(du)<+\infty$ and with $k$ satisfying the assumptions of Theorem \ref{thm:caracSD1}. Then, the corresponding self-decomposable distribution is rotationally invariant. 

\section{The Stein Equation for Self-Decomposable Laws}
\label{sec:SteinEqSD}
Throughout this subsection, $X$ is a non-degenerate self-decomposable random vector in $\bbr^d$, without Gaussian component, with law $\mu_X$, characteristic function $\varphi$ given by \eqref{eq:characID} with parameter $b\in \bbr^d$ and L\'evy measure $\nu$ given by
\begin{align}\label{eq:PolarSD2}
\nu(du)=\bbone_{(0,+\infty)}(r)\bbone_{\mathbb{S}^{d-1}}(x)\frac{k_x(r)}{r}dr\sigma(dx),
\end{align}
where $k_x(r)$ is a nonnegative function decreasing in $r\in (0,+\infty)$ and where $\sigma$ is a finite positive measure on $\mathbb{S}^{d-1}$. \textit{The following assumptions are assumed to hold true throughout this subsection: $k_x(r)$ is continuous in $r\in (0,+\infty)$, continuous in $x\in \mathbb{S}^{d-1}$, with}
\begin{align}\label{eq:condlimk}
\underset{r \rightarrow 0^+}{\lim} r^2 k_x(r)=0,\quad \underset{r \rightarrow +\infty}{\lim} k_x(r)=0,\quad x\in \mathbb{S}^{d-1}.
\end{align}
These assumptions insure that the positive measure $\tilde{\nu}$ given by
\begin{align}\label{eq:tilde}
\tilde{\nu}(du)=\bbone_{(0,+\infty)}(r)\bbone_{\mathbb{S}^{d-1}}(x)\left(-dk_x(r)\right)\sigma(dx),
\end{align}
is a well defined L\'evy measure on $\bbr^d$. Let us introduce next a collection of ID random vectors, $X_t$, $t \geq 0$, defined through their characteristic function, for all $t\geq 0$ and all $\xi\in \bbr^d$, by
\begin{align}\label{eq:characXt}
\varphi_t(\xi)=&\exp\bigg(i\langle b;\xi \rangle(1-e^{-t})+\int_{(0,+\infty)\times \mathbb{S}^{d-1}} \left(e^{i r\langle x;\xi \rangle}-1-i\langle rx;\xi \rangle\bbone_{r\leq 1}\right) \dfrac{k_x(r)-k_x(e^t r)}{r}dr \sigma(dx)\nonumber\\
&\quad\quad\quad -\int_{(0,+\infty)\times \mathbb{S}^{d-1}}i \langle rx;\xi \rangle\bbone_{e^{-t}<r\leq 1}\dfrac{k_x(e^t r)}{r}dr \sigma(dx)\bigg).
\end{align}
By changing variables, this function is, for all $\xi\in \bbr^d$ and all $t\geq 0$, equal to
\begin{align*}
\varphi_t(\xi)=\frac{\varphi(\xi)}{\varphi(e^{-t}\xi)},
\end{align*}
which is a well-defined characteristic function since $X$ is SD. Denoting by $\mu_t$ the law of $X_t$, let us introduce the following continuous family of operators $(P^\nu_t)_{t\geq 0}$, defined, for all $t\geq 0$, all $h\in \mathcal{C}_b(\mathbb{R}^d)$ and all $x\in \bbr^d$, by
\begin{align}\label{eq:semSD}
P^\nu_t(h)(x)= \int_{\mathbb{R}^d} h(xe^{- t}+y)\mu_t(dy).
\end{align}
For $t=0$, set $\mu_0=\delta_{0}$, with $\delta_{0}$ the Dirac measure at $0$, so that $P^\nu_0$ is the identity operator. Based on the computations of \cite[Lemma $3.1$]{AH18_2}, observe that the continuous family of operators $(P^\nu_t)_{t\geq 0}$ is a semigroup of operators on $\mathcal{C}_b(\mathbb{R}^d)$ with, for all $h\in \mathcal{C}_b(\mathbb{R}^d)$ and all $x\in \bbr^d$,
\begin{align*}
\int_{\bbr^d} P^\nu_t(h)(x)\mu_X(dx)=\int_{\bbr^d} h(x)\mu_X(dx),\, \underset{t\rightarrow+\infty}{\lim}P^\nu_t(h)(x)=\int_{\bbr^d} h(x)\mu_X(dx),\, \underset{t\rightarrow 0^+}{\lim}P^\nu_t(h)(x)=h(x).
\end{align*}
The next lemma identifies the generator of $(P^\nu_t)_{t\geq 0}$ on $\mathcal{S}\left(\bbr^d\right)$.

\begin{lem}\label{lem:genSDLaws}
Let $(P^\nu_t)_{t\geq 0}$ be the semigroup of operators defined by \eqref{eq:semSD}. Let $\tilde{\nu}$ be the L\'evy measure on $\bbr^d$ given by \eqref{eq:tilde}. 
The generator of $(P^\nu_t)_{t\geq 0}$ is given, for all $f\in\mathcal{S}(\bbr^d)$ and all $x\in \bbr^d$, by
\begin{align*}
\mathcal{A}(f)(x)=\langle b-\int_{\mathbb{S}^{d-1}} k_y(1)y\sigma(dy)-x; \nabla(f)(x)\rangle&+\int_{\bbr^d} \bigg(f(x+u)-f(x)\\
&\quad-\langle \nabla (f)(x); u \rangle\bbone_{\|u\|\leq 1}\bigg)\tilde{\nu}(du).
\end{align*}
\end{lem}

\begin{proof}
Let $f\in\mathcal{S}(\bbr^d)$. By Fourier inversion, for all $x\in \bbr^d$ and all $t\in (0,1)$,
\begin{align*}
\frac{1}{t}\left(P^\nu_t(f)(x)-f(x)\right)=\frac{1}{(2\pi)^d}\int_{\bbr^d} \mathcal{F}(f)(\xi)e^{i\langle x; \xi \rangle}\left(e^{i\langle x;\xi (e^{-t}-1) \rangle}\dfrac{\varphi(\xi)}{\varphi(e^{-t}\xi)}-1\right)\dfrac{d\xi}{t}.
\end{align*}
By Lemma \ref{lem:GenCau} (i), for all $\xi \in \bbr^d$ and all $x\in \bbr^d$,
\begin{align*}
\underset{t\rightarrow 0^+}{\lim}\frac{1}{t}\left(e^{i\langle x;\xi (e^{-t}-1) \rangle}\dfrac{\varphi(\xi)}{\varphi(e^{-t}\xi)}-1\right)&=i\langle b-\int_{\mathbb{S}^{d-1}} k_y(1)y\sigma(dy)-x;\xi\rangle\\
&\quad\quad +\int_{\bbr^d} \left(e^{i\langle u;\xi \rangle}-1-i\langle u;\xi \rangle\bbone_{\|u\|\leq 1}\right) \tilde{\nu}(du).
\end{align*}
Moreover, by Lemma \ref{lem:GenCau} (ii), for all $\xi \in \bbr^d$ and all $x\in \bbr^d$,
\begin{align*}
\underset{t\in (0,1)}{\sup}\frac{1}{t}\left|e^{i\langle x;\xi (e^{-t}-1) \rangle}\dfrac{\varphi(\xi)}{\varphi(e^{-t}\xi)}-1\right|&\leq \left(\|x\|+\|b\|\right)\|\xi\|+\|\xi\|^2 \int_{\mathbb{S}^{d-1}} \sigma(dy) \int_0^1s^2 (-dk_y(s)) +4  \int_{\mathbb{S}^{d-1}}k_y(1)\sigma(dy)\\
&\quad\quad+2 \|\xi\|^2  \int_{\mathbb{S}^{d-1}}k_y(1)\sigma(dy)+2\int_{\mathbb{S}^{d-1}} \sigma(dy) \int_{1}^{+\infty} (-dk_y(s))\\
&\quad\quad  +\|\xi\| \int_{\mathbb{S}^{d-1}} k_y(1)\sigma(dy).
\end{align*}
Then, a direct application of the Lebesgue dominated convergence theorem together with Fourier duality imply that
\begin{align*}
\underset{t\rightarrow 0^+}{\lim}\frac{1}{t}\left(P^\nu_t(f)(x)-f(x)\right)&=\langle b-\int_{\mathbb{S}^{d-1}} yk_y(1)\sigma(dy)-x; \nabla(f)(x)\rangle\\
&\quad+\int_{\bbr^d} \bigg(f(x+u)-f(x)-\langle \nabla (f)(x); u \rangle\bbone_{\|u\|\leq 1}\bigg)\tilde{\nu}(du),
\end{align*}
which concludes the proof of the lemma.
\end{proof}


Based on the previous lemma, it is natural to consider the following Stein equation for self-decomposable distributions with polar decomposition given by \eqref{eq:PolarSD2} (under appropriate assumptions on the function $k_x(r)$) : for all $h\in \mathcal{H}_2\cap \mathcal{C}_c^{\infty}\left(\bbr^d\right)$ and all $x\in \bbr^d$,
\begin{align}\label{eq:SteinEquSD}
\langle \tilde{b}-x; \nabla(f_h)(x)\rangle+\int_{\bbr^d} \bigg(f_h(x+u)-f_h(x)-\langle \nabla (f_h)(x); u \rangle\bbone_{\|u\|\leq 1}\bigg)\tilde{\nu}(du)=h(x)-\bbe h(X),
\end{align}
where $\tilde{b}=b-\int_{\mathbb{S}^{d-1}}k_y(1) y\sigma(dy)$. By semigroup theory, a candidate solution to \eqref{eq:SteinEquSD} is given by,
\begin{align}\label{eq:SolSD}
f_h(x)=-\int_{0}^{+\infty} \left(P^\nu_t(h)(x)-\bbe h(X)\right)dt,\quad x\in\bbr^d.
\end{align}
The next proposition proves the existence of the function $f_h$ given by \eqref{eq:SolSD}, studies its regularity and proves that this function is a strong solution of \eqref{eq:SteinEquSD} on $\bbr^d$. 

\begin{thm}\label{thm:Stein}
Let $X$ be a non-degenerate SD random vector without Gaussian component, with law $\mu_X$, characteristic function $\varphi$, L\'evy measure $\nu$ having polar decomposition given by \eqref{eq:PolarSD2} where the function $k_x(r)$ is continuous in $r\in(0,+\infty)$, continuous in $x\in \mathbb{S}^{d-1}$ and satisfies \eqref{eq:condlimk}. Assume that there exists $\varepsilon\in (0,1)$ such that $\bbe \|X\|^\varepsilon <+\infty$ and that there exist $\beta_1>0,\beta_2>0$ and $\beta_3\in (0,1)$ such that the function $k_x(r)$ in \eqref{eq:PolarSD2} satisfies
\begin{align*}
&\gamma_1=\underset{t\geq 0}{\sup}\, \left(e^{\beta_1 t}\int_{(1,+\infty)\times\mathbb{S}^{d-1}}\frac{k_x(e^t r)}{r}dr\sigma(dx)\right)<+\infty,\nonumber\\
&\gamma_2=\underset{t\geq 0}{\sup}\, \left(e^{\beta_2 t}\int_{(0,1)\times\mathbb{S}^{d-1}} rk_x(e^tr)dr\sigma(dx)\right)<+\infty,
\end{align*}
and that,
\begin{align*}
\gamma_3=\underset{t\geq 0}{\sup} \left(e^{-(1-\beta_3)t} \left\|\int_{\mathbb{S}^{d-1}}x\left(\int_{1}^{e^t} k_x(r)dr\right)\sigma(dx)\right\|\right)<+\infty.
\end{align*}
Let $X_t$, $t\geq 0$, be the random vector defined through the characteristic function $\varphi_t$ given by \eqref{eq:characXt} and assume that,
\begin{align}\label{eq:unifbound}
\underset{t \geq 0}{\sup}\, \bbe\, \|X_t\|^\varepsilon <+\infty.
\end{align}
Let $(P^\nu_t)_{t\geq 0}$ be the semigroup of operators defined by \eqref{eq:semSD}. Then, for any $h\in \mathcal{H}_1$, the function $f_h$, given, for all $x\in \bbr^d$, by 
\begin{align*}
f_h(x)=-\int_0^{+\infty} \left(P_t^\nu(h)(x)-\bbe h(X)\right)dt,
\end{align*}
is well defined and continuously differentiable on $\bbr^d$ with $M_1(f_h)\leq 1$.~Moreover, for any $h\in \mathcal{H}_2\cap \mathcal{C}_c^{\infty}\left(\bbr^d\right)$,  $f_h$ is twice continuously differentiable on $\bbr^d$, $M_2(f_h)\leq 1/2$ and
\begin{align*}
\langle \tilde{b}-x; \nabla(f_h)(x)\rangle&+\int_{\bbr^d} \bigg(f_h(x+u)-f_h(x)-\langle \nabla (f_h)(x); u \rangle\bbone_{\|u\|\leq 1}\bigg)\tilde{\nu}(du),\quad x\in \bbr^d,
\end{align*}
where $\tilde{\nu}$ is given by \eqref{eq:tilde} and where $\tilde{b}=b-\int_{\mathbb{S}^{d-1}}k_y(1) y\sigma(dy)$.
\end{thm}

\begin{proof}
To start with, let us prove that, for any $h\in \mathcal{H}_1$, the function $f_h$ defined by \eqref{eq:SolSD} does exist. For all $x\in \bbr^d$ and all $t>0$,
\begin{align*}
\left| P^\nu_t(h)(x)-\bbe h(X)\right|&=\left|\bbe h\left(xe^{-t}+X_t\right)-\bbe h(X)\right|\\
&\leq e^{-t} \|x\|+\left|\bbe h(X_t)-\bbe h(X) \right|\\
&\leq e^{-t} \|x\|+d_{W_1}(\mu_t,\mu_X)\\
&\leq e^{-t} \|x\|+Ce^{-c t},
\end{align*}
where we have used Proposition \ref{prop:ExpConv} in the last line. Then, the function $f_h$ is well defined on $\bbr^d$. Now, since $h\in \mathcal{H}_1$, $f_h$ is continuously differentiable on $\bbr^d$ with, for all $x\in \bbr^d$,
\begin{align*}
\nabla\left(f_h\right)(x)=-\int_{0}^{+\infty} e^{-t}P^\nu_t(\nabla (h))(x)dt.
\end{align*}
Moreover, reasoning as in \cite[Proposition $3.4$]{AH18_2}, one gets that $M_1(f_h)\leq 1$. Now, if $h\in \mathcal{H}_2\cap \mathcal{C}_c^\infty\left(\bbr^d\right)$, then, $f_h$ is twice continuously differentiable on $\bbr^d$ with, for all $x\in \bbr^d$ and all $i,j\in \{1,\dots,d\}$,
\begin{align*}
\partial^2_{i,j}(f_h)(x)=-\int_{0}^{+\infty} e^{-2 t}P^\nu_t(\partial^2_{i,j}(h))(x)dt,
\end{align*}
and with $M_2(f_h)\leq 1/2$. To conclude let us prove that $f_h$ is a strong solution of \eqref{eq:SteinEquSD} on $\bbr^d$. Set $u(t,x)=P^\nu_t(h)(x)$, for $t\geq 0$ and $x\in \bbr^d$. First, let us prove that, for all $t \geq 0$ and all $x\in \bbr^d$,
\begin{align}\label{eq:HeatEq}
\partial_t(u)(t,x)= \mathcal{A}(u)(t,x).
\end{align}
Since $h\in \mathcal{C}_c^\infty\left(\bbr^d\right)$, by Fourier inversion, for all $t \geq 0$ and all $x\in \bbr^d$,
\begin{align*}
u(t,x)=\frac{1}{(2\pi)^d}\int_{\bbr^d} \mathcal{F}(h)(\xi) e^{i \langle x;\xi e^{-t} \rangle} \frac{\varphi(\xi)}{\varphi(e^{-t}\xi)} d\xi.
\end{align*}
Now, by Lemma \ref{lem:diffallpoints}, for all $x\in \bbr^d$, all $\xi \in \bbr^d$ and all $t\geq 0$,
\begin{align*}
\dfrac{d}{dt}\left(e^{i \langle x;\xi e^{-t} \rangle} \frac{\varphi(\xi)}{\varphi(e^{-t}\xi)}\right)&=-i e^{-t} \langle x; \xi \rangle e^{i e^{-t} \langle x;\xi \rangle} \dfrac{\varphi(\xi)}{\varphi(e^{-t}\xi)}+i e^{-t} \langle b; \xi \rangle e^{i e^{-t} \langle x; \xi\rangle}\dfrac{\varphi(\xi)}{\varphi(e^{-t}\xi)}\\
&\quad -i e^{-t} \langle \int_{\mathbb{S}^{d-1}}k_y(1) y\sigma(dy);\xi \rangle e^{i \langle x;\xi e^{-t} \rangle} \frac{\varphi(\xi)}{\varphi(e^{-t}\xi)}\\
&\quad+ \int_{\mathbb{R}^d} \left(e^{i \langle u;\xi e^{-t}\rangle}-1-i \langle u;\xi e^{-t} \rangle\bbone_{\|u\|\leq 1}\right)\tilde{\nu}(du) e^{i \langle x;\xi e^{-t} \rangle} \frac{\varphi(\xi)}{\varphi(e^{-t}\xi)}
\end{align*}
Moreover, for all $x\in \bbr^d$, all $\xi \in \bbr^d$ and all $t\geq 0$,
\begin{align*}
\left| \dfrac{d}{dt}\left(e^{i \langle x;\xi e^{-t} \rangle} \frac{\varphi(\xi)}{\varphi(e^{-t}\xi)}\right)\right| &\leq e^{-t} \left(\|x\|+\|b\|+\left\|\int_{\mathbb{S}^{d-1}}yk_y(1)\sigma(dy)\right\|\right)\|\xi\|+ e^{-2t} \|\xi\|^2 \int_{\|u\|\leq 1} \|u\|^2 \tilde{\nu}(du)\\
&\quad+2\int_{\|u\|\geq 1} \tilde{\nu}(du)\\
&\leq \left(\|x\|+\|b\|+\left\|\int_{\mathbb{S}^{d-1}}yk_y(1)\sigma(dy)\right\|\right)\|\xi\|+ \|\xi\|^2 \int_{\|u\|\leq 1} \|u\|^2 \tilde{\nu}(du)\\
&\quad+2\int_{\|u\|\geq 1} \tilde{\nu}(du).
\end{align*}
Then, for all $t\geq 0$ and all $x\in \bbr^d$,
\begin{align*}
\partial_t(u)(t,x)&=\frac{1}{(2\pi)^d} \int_{\bbr^d} \mathcal{F}(h)(\xi)\dfrac{d}{dt}\left(e^{i \langle x;\xi e^{-t} \rangle} \frac{\varphi(\xi)}{\varphi(e^{-t}\xi)}\right)d\xi\\
&=\mathcal{A}(u)(t,x),
\end{align*}
where the Fourier symbol of $\mathcal{A}$ and the Fourier representation of $u(t,x)$ have been used in the last equality. To pursue, let $0<T <+\infty$ and let us integrate out the equation \eqref{eq:HeatEq} between $0$ and $T$. Then, for all $x\in \bbr^d$,
\begin{align*}
P^\nu_T(h)(x)-h(x)=\int_0^T \mathcal{A}\left(P^\nu_t(h)\right)(x)dt,
\end{align*}
then, letting $T\rightarrow +\infty$ and the ergodicity of the semigroup $(P_t^\nu)_{t\geq 0}$ give:
\begin{align*}
\underset{T\rightarrow +\infty}{\lim} (P^\nu_T(h)(x)-h(x))=\bbe h(X)-h(x),\quad x\in \bbr^d.
\end{align*}
Next, let us prove that $\int_0^{+\infty} \left| \mathcal{A}(P^\nu_t(h))(x)\right|dt<+\infty$, for all $x\in \bbr^d$. To do so, one needs to estimate $\|\nabla (P^\nu_t(h))(x)\|$ and
\begin{align*}
(I):=\left|\int_{\bbr^d} \left(P^\nu_t(h)(x+u)-P^\nu_t(h)(x)-\langle \nabla(P^\nu_t(h))(x);u\rangle \bbone_{\|u\|\leq 1}\right)\tilde{\nu}(du)\right|.
\end{align*} 
From the commutation relation and the fact that $h\in \mathcal{H}_2$,
\begin{align*}
\|\nabla (P^\nu_t(h))(x)\| \leq e^{-t}.
\end{align*}
Now, let us bound (I). For all $x\in \bbr^d$ and all $t>0$
\begin{align}\label{ineq:boundI}
(I)& \leq \left|\int_{\bbr^d} \left(P^\nu_t(h)(x+u)-P^\nu_t(h)(x)-\langle \nabla(P^\nu_t(h))(x);u\rangle \bbone_{\|u\|\leq e^{t}}\right)\tilde{\nu}(du)\right|\nonumber\\
&\quad\quad+\left|\int_{\bbr^d} \langle \nabla (P_t(h))(x);u\rangle \bbone_{1\leq \|u\|\leq e^t} \tilde{\nu}(du)\right|.
\end{align}
Let us start with the second term on the right-hand side of \eqref{ineq:boundI}. Again, via the commutation relation and an integration by parts
\begin{align}\label{ineq=boundI2}
\left|\int_{\bbr^d} \langle \nabla (P_t(h))(x);u\rangle \bbone_{1\leq \|u\|\leq e^t} \tilde{\nu}(du)\right|&\leq e^{-t} \left|\langle P^\nu(\nabla(h))(x);\int_{\bbr^d} u \bbone_{1<\|u\|\leq e^t} \tilde{\nu}(du)\rangle\right|\nonumber\\
&\leq e^{-t} \left\|\int_{\mathbb{S}^{d-1}}y\left(\int_{1}^{e^t} r (-dk_y(r))\right)\sigma(dy)\right\|\nonumber\\
&\leq e^{-t} \left\|\int_{\mathbb{S}^{d-1}}y \left(\int_1^{e^t} k_y(r)dr +k_y(1)-e^t k_y(e^t)\right)\sigma(dy)\right\|\nonumber\\
&\leq \bigg(\gamma_3 e^{-\beta_3 t} +e^{-t}\left\|\int_{\mathbb{S}^{d-1}}yk_y(1)\sigma(dy)\right\|+ \left\|\int_{\mathbb{S}^{d-1}}yk_y(e^t)\sigma(dy)\right\| \bigg).
\end{align}
Note also that $\int_0^{+\infty} k_y(e^t)dt=\int_1^{+\infty} k_y(r)dr/r<+\infty$, for $y\in \mathbb{S}^{d-1}$. This concludes the bounding of the second term on the right-hand side of \eqref{ineq:boundI}. For the first term, for all $x\in \bbr^d$ and all $t\geq 0$
\begin{align*}
\left|\int_{\bbr^d} \left(P^\nu_t(h)(x+u)-P^\nu_t(h)(x)-\langle \nabla(P^\nu_t(h))(x);u\rangle \bbone_{\|u\|\leq e^{t}}\right)\tilde{\nu}(du)\right| \leq (I_1)+(I_2),
\end{align*}
where,
\begin{align*}
&(I_1):=\left|\int_{\|u\|\leq e^t} \left(P^\nu_t(h)(x+u)-P^\nu_t(h)(x)-\langle \nabla(P^\nu_t(h))(x);u\rangle\right)\tilde{\nu}(du)\right|\\
&(I_2):=\left|\int_{\|u\|>e^t} \left(P^\nu_t(h)(x+u)-P^\nu_t(h)(x)\right)\tilde{\nu}(du)\right|.
\end{align*}
Then, by commutation and a change of variables
\begin{align*}
(I_1)\leq \left|\int_{(0,1)\times\mathbb{S}^{d-1}} \left(P_t^\nu(h)(x+e^t ry)-P^\nu_t(h)(x)-\langle P^\nu_t(\nabla(h))(x);r y \rangle \right) (-dk_y(e^t r)) \sigma(dy)\right|.
\end{align*}
Now, note that, for all $x\in \bbr^d$, all $u=ry\in(0,1)\times\mathbb{S}^{d-1}$ and all $t>0$,
\begin{align*}
\left| P_t^\nu(h)(x+e^t ry)-P^\nu_t(h)(x)-\langle P^\nu_t(\nabla(h))(x);r y \rangle \right| \leq r^2
\end{align*}
so that, integratng by parts, for all $t>0$,
\begin{align}\label{ineq:boundI11}
(I_1)&\leq \int_{\mathbb{S}^{d-1}}\left(\int_0^1 r^2 (-dk_y(e^t r))\right)\sigma(dy)\nonumber\\
&\leq \int_{\mathbb{S}^{d-1}} \left(2\int_0^1 rk_y(e^t r)dr-k_y(e^t)\right)\sigma(dy)\nonumber\\
&\leq \left(2 \gamma_2 e^{-\beta_2 t} -\int_{\mathbb{S}^{d-1}}k_y(e^t)\sigma(dy)\right).
\end{align}
Similarly for $(I_2)$, for all $x\in \bbr^d$, all $u=ry\in(0,1)\times\mathbb{S}^{d-1}$ and all $t>0$
\begin{align}\label{ineq:boundI12}
(I_2)&= \left|\int_{(1,+\infty)\times \mathbb{S}^{d-1}} \left(P^\nu_t(h)(x+e^t r y)-P^\nu_t(h)(x)\right)(-dk_y(e^t r)) \sigma(dy)\right|\nonumber\\
&\leq 2 \int_{\mathbb{S}^{d-1}}\left(\int_{1}^{+\infty} (-dk_y(e^t r))\right)\sigma(dy) \nonumber\\
&\leq 2  \int_{\mathbb{S}^{d-1}}k_y(e^t)\sigma(dy).
\end{align}
Combining \eqref{ineq:boundI} together with \eqref{ineq=boundI2}--\eqref{ineq:boundI12},
\begin{align*}
(I)&\leq \bigg(\gamma_3 e^{-\beta_3 t} +e^{-t}\left\|\int_{\mathbb{S}^{d-1}}yk_y(1)\sigma(dy)\right\|+ \left\|\int_{\mathbb{S}^{d-1}}yk_y(e^t)\sigma(dy)\right\| \bigg)\\
&\quad\quad+\left(2 \gamma_2 e^{-\beta_2 t} -\int_{\mathbb{S}^{d-1}}k_y(e^t)\sigma(dy)\right)+2  \int_{\mathbb{S}^{d-1}}k_y(e^t)\sigma(dy).
\end{align*} 
Then, for all $x\in \bbr^d$
\begin{align*}
\bbe h(X)-h(x)=\int_0^{+\infty} \mathcal{A}(P^\nu_t(h))(x)dt.
\end{align*}
Noting that $\int_0^{+\infty} \mathcal{A}(P^\nu_t(h))(x)dt=-\mathcal{A}(f_h)(x)$, for all $x\in \bbr^d$, concludes the proof of the theorem.
\end{proof}

\begin{rem} (i) Let us discuss the assumptions of Theorem \ref{thm:Stein}. If $X$ is $\alpha$-stable with L\'evy measure given by \eqref{eq:PolarStable} and with $\alpha\in (0,1]$, then for any $\varepsilon\in (0,\alpha)$, $\beta_1\in (0,\alpha)$, $\beta_2\in (0,\alpha)$ and $\beta_3\in (0,\alpha)$ while
\begin{align*}
\bbe \|X\|^\varepsilon <+\infty,\quad \gamma_1<+\infty,\quad \gamma_2<+\infty,\quad \gamma_3<+\infty.
\end{align*}
Next, let us discuss the condition \eqref{eq:unifbound} in the particular case $\alpha \in (0,1)$. (A similar discussion can be performed in the case $\alpha=1$ but requires different estimates. )
Since $\alpha \in (0,1)$, the random vector $X_t$, $t\geq 0$, defined through \eqref{eq:characXt} has the characteristic function given, for all $\xi \in \bbr^d$ and all $t \geq 0$, by
\begin{align*}
\varphi_t(\xi)=\exp\left(i\langle b_0;\xi \rangle(1-e^{-t})+(1-e^{-\alpha t})\int_{\bbr^d} \left(e^{i \langle u; \xi \rangle}-1\right)\nu(du)\right),
\end{align*}
with $\nu$ as in \eqref{eq:PolarStable}. Then, $X_t=_d (1-e^{-t}) b_0+(1-e^{-\alpha t})^{\frac{1}{\alpha}} \tilde{X}$ where $\tilde{X}$ is $\alpha$-stable with $b_0=0$ and $\alpha\in (0,1)$. It is then straightforward to check that $\bbe \|X_t\|^\varepsilon$ is uniformly bounded in $t$ for any $\varepsilon\in (0,\alpha)$.\\
(ii) Now, let $X$ be a non-degenerate SD random vector as in Theorem \ref{thm:Stein} such that
\begin{align*}
\int_{\|u\|\leq 1}\|u\|\nu(du)<+\infty,\quad \underset{\varepsilon\rightarrow 0^+}{\lim}\varepsilon k_y(\varepsilon)=0, \quad y\in \mathbb{S}^{d-1}.
\end{align*}
Let $f_h$ be the solution to the Stein equation \eqref{eq:SteinEquSD} defined by \eqref{eq:SolSD}, for $h\in \mathcal{H}_2\cap \mathcal{C}_c^\infty(\bbr^d)$. Then, by an integration by parts, for all $x\in \bbr^d$,
\begin{align*}
&\langle \tilde{b}+\int_{\mathbb{S}^{d-1}}y\left(\int_0^1 rdk_y(r)\right)\sigma(dy)-x; \nabla(f_h)(x)\rangle+\int_{\bbr^d} \bigg(f_h(x+u)-f_h(x)\bigg)\tilde{\nu}(du)=\\
&\langle b-\int_{\mathbb{S}^{d-1}}y\left(\int_0^1 k_y(r)dr\right)\sigma(dy)-x; \nabla(f_h)(x)\rangle+\int_{\bbr^d} \bigg(f_h(x+u)-f_h(x)\bigg)\tilde{\nu}(du),
\end{align*}
so that, in this case, $f_h$ is a solution to the following Stein equation
\begin{align}\label{eq:SteinEq2}
\langle b_0-x; \nabla(f_h)(x)\rangle+\int_{\bbr^d} \bigg(f_h(x+u)-f_h(x)\bigg)\tilde{\nu}(du)=h(x)-\bbe h(X),\quad x\in \bbr^d.
\end{align}
In particular, if $X$ is $\alpha$-stable with $\alpha\in (0,1)$, then the equation \eqref{eq:SteinEq2} boils down to
\begin{align*}
\langle b_0-x; \nabla(f_h)(x)\rangle+\alpha \int_{\bbr^d} \bigg(f_h(x+u)-f_h(x)\bigg)\nu(du)=h(x)-\bbe h(X),\quad x\in \bbr^d.
\end{align*}
(iii) Next, let $X$ be a non-degenerate SD random vector as in Theorem \ref{thm:Stein} and such that
\begin{align*}
\int_{\|u\|\geq 1}\|u\|\nu(du)<+\infty,\quad \underset{R\rightarrow +\infty}{\lim} R k_y(R)=0, \quad \underset{\varepsilon\rightarrow 0^+}{\lim} \varepsilon^2 k_y(\varepsilon)=0,\quad y\in\mathbb{S}^{d-1},
\end{align*}
Let $f_h$ be the solution to the Stein equation \eqref{eq:SteinEquSD} defined by \eqref{eq:SolSD}, for $h\in \mathcal{H}_2\cap \mathcal{C}_c^\infty(\bbr^d)$. Then, integrating by parts twice, for all $x\in \bbr^d$,
\begin{align*}
&\langle \tilde{b}+\int_{\mathbb{S}^{d-1}}y\left(\int_1^{+\infty} r(-dk_y(r))\right)\sigma(dy)-x; \nabla(f_h)(x)\rangle+\int_{\bbr^d} \bigg(f_h(x+u)-f_h(x)-\langle \nabla (f_h)(x);u \rangle\bigg)\tilde{\nu}(du)=\\
&\langle b+\int_{\mathbb{S}^{d-1}}y\left(\int_1^{+\infty} k_y(r)dr\right)\sigma(dy)-x; \nabla(f_h)(x)\rangle+\int_{\bbr^d} \bigg(f_h(x+u)-f_h(x)-\langle \nabla (f_h)(x);u \rangle\bigg)\tilde{\nu}(du)=\\
&\quad\langle \bbe X-x;\nabla(f_h)(x)\rangle+\int_{\bbr^d} \langle \nabla (f_h)(x+u)-\nabla(f_h)(x);u\rangle \nu(du),
\end{align*} 
so that $f_h$ is a solution to the following Stein equation
\begin{align}\label{eq:SteinEq3}
\langle \bbe X-x;\nabla(f_h)(x)\rangle+\int_{\bbr^d} \langle \nabla (f_h)(x+u)-\nabla(f_h)(x);u\rangle \nu(du)=h(x)-\bbe h(X),\quad x\in \bbr^d.
\end{align}
\end{rem}

\section{Applications to Functional Inequalities for SD Random Vectors}
\label{sec:FASD}
This section discusses Poincar\'e-type inequalities for self-decomposable random vectors, providing in particular new proofs based on the semigroup of operators $\left(P^\nu_t\right)_{t\geq 0}$ defined in \eqref{eq:semSD}. This proof is in line with the standard proof of the Gaussian Poincar\'e inequality based on the differentiation of the variance along the Ornstein-Uhlenbeck semigroup. In the literature, standard references regarding Poincar\'e-type inequalities for infinitely divisible random vectors are \cite{Chen,HPAS}. In \cite{Chen}, the proof is based on stochastic calculus for L\'evy processes and the L\'evy-It\^o decomposition whereas in \cite{HPAS}, the proof is based on a covariance representation for infinitely divisible random vectors. Let us also mention that Poincar\'e-type inequalities for stable random vectors have been obtained in \cite{RW03,WW15}.

\begin{prop}\label{prop:poincSD}
Let $X$ be a centered SD random vector with L\'evy measure $\nu$ such that
\begin{align*}
\int_{\|u\|\geq 1}\|u\|\nu(du)<+\infty,\quad \nu(du)=\bbone_{(0,+\infty)}(r) \bbone_{\mathbb{S}^{d-1}}(x)\frac{k_x(r)}{r}dr\sigma(dx),
\end{align*}
where $\sigma$ is a finite positive measure on $\mathbb{S}^{d-1}$ and where $k_x(r)$ is a nonnegative continuous function decreasing in $r\in(0,+\infty)$, continuous in $x\in \mathbb{S}^{d-1}$ with
\begin{align*}
\underset{r\rightarrow +\infty}{\lim}k_x(r)=0,\quad \underset{r\rightarrow 0^+}{\lim}r^2k_x(r)=0, \quad x\in\mathbb{S}^{d-1}.
\end{align*}
Then, for all $f\in \mathcal{C}_{c}^{\infty}\left(\bbr^d\right)$ with $\bbe f(X)=0$
\begin{align*}
\bbe f(X)^2 \leq \bbe \int_{\bbr^d} \left(f(X+u)-f(X)\right)^2 \nu(du).
\end{align*}
\end{prop}

\begin{proof}
Let $X$ be a SD random vector with characteristic function $\varphi$ and L\'evy measure $\nu$ satisfying the hypotheses of the proposition. Let $(P^\nu_t)_{t\geq 0}$ be the semigroup of operators given by \eqref{eq:semSD}. In particular, on $\mathcal{C}_{c}^{\infty}\left(\bbr^d\right)$, for all $t\geq 0$ and all $x\in \bbr^d$,
\begin{align*}
P^\nu_t(f)(x)=\frac{1}{(2\pi)^d} \int_{\bbr^d} \mathcal{F}(f)(\xi) e^{i \langle x;\xi e^{-t}\rangle}\dfrac{\varphi(\xi)}{\varphi(e^{-t}\xi)} d\xi.
\end{align*}
Now, for all $f\in \mathcal{C}_{c}^{\infty}\left(\bbr^d\right)$ and all $x\in \bbr^d$, let
\begin{align*}
\Delta_\nu (f)(x):=\int_{\bbr^d} \langle\nabla(f)(x+u)-\nabla(f)(x);u\rangle \nu(du).
\end{align*}
This operator admits the Fourier representation, i.e., for all $x\in \bbr^d$,
\begin{align*}
\Delta_\nu (f)(x)=\int_{\bbr^d}\mathcal{F}(f)(\xi) \sigma_{\nu}(\xi)e^{i\langle x; \xi \rangle}\frac{d\xi}{(2\pi)^d},
\end{align*}
with, for all $\xi \in \bbr^d$
\begin{align*}
 \sigma_{\nu}(\xi)=\int_{\bbr^d} i \langle u;\xi \rangle\left(e^{i \langle u;\xi \rangle}-1\right)\nu(du).
\end{align*}
Next, let $f\in \mathcal{C}_{c}^{\infty}\left(\bbr^d\right)$ be such that $\bbe f(X)=0$. Then, for all $t\geq0$,
\begin{align*}
\dfrac{d}{dt}\left(\bbe P^\nu_t(f)(X)^2\right)&=2 \bbe P^\nu_t(f)(X) \frac{d}{dt}\left(P^\nu_t(f)\right)(X)\\
&=2 \bbe P_t^\nu(f)(X) \mathcal{A}\left(P_t^\nu(f)\right)(X),
\end{align*}
where $\mathcal{A}$ is defined, for all $f\in \mathcal{C}_{c}^{\infty}\left(\bbr^d\right)$ and all $x\in \bbr^d$, by
\begin{align*}
\mathcal{A}(f)(x)=-\langle x; \nabla(f)(x)\rangle+\Delta_\nu (f)(x).
\end{align*}
Thus, for all $t\geq 0$
\begin{align*}
\dfrac{d}{dt}\left(\bbe P^\nu_t(f)(X)^2\right)&=-2 \bbe P^\nu_t(f)(X) \langle X;\nabla (P^\nu_t(f))(X)\rangle+2\bbe P^\nu_t(f)(X) \Delta_\nu\left(P^\nu_t(f)\right)(X).
\end{align*}
Next, from Theorem \ref{thm:multid}, observe that, for all $f\in\mathcal{C}_{c}^{\infty}\left(\bbr^d\right)$ and all $t\geq 0$,
\begin{align*}
2\bbe P^\nu_t(f)(X) \langle X;\nabla (P^\nu_t(f))(X)\rangle&=\bbe \langle X ; \nabla \left(P^\nu_t(f)^2\right)(X)\rangle\\
&=\bbe \Delta_\nu \left(P^\nu_t(f)^2\right)(X),
\end{align*}
and so, for all $t\geq 0$,
\begin{align}
\dfrac{d}{dt}\left(\bbe P^\nu_t(f)(X)^2\right)&=-\bbe \bigg( \Delta_\nu \left(P^\nu_t(f)^2\right)(X)-2 P^\nu_t(f)(X) \Delta_\nu\left(P^\nu_t(f)\right)(X)\bigg).
\end{align}
Next, using Fourier arguments as in the proof of \cite[Proposition $4.1$]{LR04}, for all $f\in \mathcal{C}_{c}^{\infty}\left(\bbr^d\right)$ and all $x\in \bbr^d$,
\begin{align*}
\Delta_{\nu} \left(f^2\right)(x)-2f(x)\Delta_{\nu}\left(f\right)(x)&=\int_{\bbr^d} \int_{\bbr^d} \mathcal{F}(f)(\xi)\mathcal{F}(f)(\zeta) e^{i \langle x ;\zeta +\xi\rangle }\sigma_{\nu}(\xi+\zeta) \frac{d\xi d\zeta}{(2\pi)^{2d}} \\
&\quad\quad-2\int_{\bbr^d} \int_{\bbr^d} \mathcal{F}(f)(\xi)\mathcal{F}(f)(\zeta) e^{i \langle x ;\zeta +\xi\rangle }\sigma_{\nu}(\xi) \frac{d\xi d\zeta}{(2\pi)^{2d}} \\
&=\int_{\bbr^d} \int_{\bbr^d} \mathcal{F}(f)(\xi)\mathcal{F}(f)(\zeta) e^{i \langle x ;\zeta +\xi\rangle }\sigma_{\nu}(\zeta+\xi) \frac{d\xi d\zeta}{(2\pi)^{2d}} \\
&\quad\quad-\int_{\bbr^d} \int_{\bbr^d} \mathcal{F}(f)(\xi)\mathcal{F}(f)(\zeta) e^{i \langle x ;\zeta +\xi\rangle }\sigma_{\nu}(\xi) \frac{d\xi d\zeta}{(2\pi)^{2d}}\\
&\quad\quad-\int_{\bbr^d} \int_{\bbr^d} \mathcal{F}(f)(\xi)\mathcal{F}(f)(\zeta) e^{i \langle x ;\zeta +\xi\rangle }\sigma_{\nu}(\zeta)\frac{d\xi d\zeta}{(2\pi)^{2d}}.
\end{align*}
Moreover, an integration by parts in the radial coordinate gives, for all $\xi \in \bbr^d$
\begin{align*}
\sigma_{\nu}(\xi)=\int_{\bbr^d} \left(e^{i \langle u; \xi\rangle}-1-i \langle u ;\xi\rangle \right) \tilde{\nu}(du),
\end{align*}
where $\tilde{\nu}(du)=\bbone_{(0,+\infty)}(r)\bbone_{\mathbb{S}^{d-1}}(x) \left(-dk_x(r)\right)\sigma(dx)$. Then, for all $\xi, \zeta\in \bbr^d$
\begin{align*}
\sigma_{\nu}(\xi+\zeta)-\sigma_{\nu}(\xi)-\sigma_{\nu}(\zeta)=\int_{\bbr^d} \left(e^{i \langle u; \xi\rangle}-1\right)\left(e^{i \langle u; \zeta\rangle}-1\right) \tilde{\nu}(du),
\end{align*}
and thus, for all $x\in \bbr^d$,
\begin{align}\label{eq:laplacienfrac}
\Delta_{\nu} \left(f^2\right)(x)-2f(x)\Delta_{\nu}\left(f\right)(x)&= \int_{\bbr^d} \int_{\bbr^d} \mathcal{F}(f)(\xi)\mathcal{F}(f)(\zeta) e^{i x (\zeta +\xi)}\bigg(\int_{\bbr^d} \left(e^{i \langle u; \xi\rangle}-1\right)\nonumber\\
&\quad\quad\times \left(e^{i \langle u; \zeta \rangle}-1\right) \tilde{\nu}(du)\bigg) \frac{d\xi d\zeta}{(2\pi)^{2d}}\nonumber\\
&=\int_{\bbr^d} \left(f(x+u)-f(x)\right)^2 \tilde{\nu}(du).
\end{align}
Then, for all $t\geq 0$
\begin{align*}
\dfrac{d}{dt}\left(\bbe P^\nu_t(f)(X)^2\right)&=-\bbe \int_{\bbr^d} \left(P^\nu_t(f)(X+u)-P^\nu_t(f)(X)\right)^2 \tilde{\nu}(du).
\end{align*}
But, from a change of variables, Jensen inequality and invariance,
\begin{align*}
\bbe \int_{\bbr^d} \left(P^\nu_t(f)(X+u)-P^\nu_t(f)(X)\right)^2 \tilde{\nu}(du)&=\bbe \int_{\bbr^d}\left(\int_{\bbr^d} \left(f\left(Xe^{-t}+ue^{-t}+y\right)-f\left(Xe^{-t}+y\right)\right)\mu_t(dy)\right)^2\tilde{\nu}(du)\\
&=\bbe \int_{(0,+\infty)}\int_{\mathbb{S}^{d-1}} \left(\int_{\bbr^d} \left(f\left(Xe^{-t}+r x+y\right)-f\left(Xe^{-t}+y\right)\right)\mu_t(dy)\right)^2\\
&\quad\quad\times  \left(-dk_x(e^t r)\right) \sigma(dx)\\
&\leq \bbe \int_{(0,+\infty)\times \mathbb{S}^{d-1}} \left(f\left(X+r x\right)-f\left(X\right)\right)^2 \left(-dk_x(e^t r)\right) \sigma(dx).
 \end{align*}
Thus,
\begin{align*}
-\dfrac{d}{dt}\left(\bbe P^\nu_t(f)(X)^2\right)\leq \bbe \int_{(0,+\infty)\times \mathbb{S}^{d-1}} \left(f\left(X+r x\right)-f\left(X\right)\right)^2 \left(-dk_x(e^t r)\right) \sigma(dx).
\end{align*}
With an integration by parts, observe that, for $x\in \mathbb{S}^{d-1}$,
\begin{align*}
\int_0^{+\infty} \bbe\left(f\left(X+e^{-t}r x\right)-f\left(X\right)\right)^2 \left(-dk_x(r)\right)&=\int_0^{+\infty} k_x(r)2 \bbe\langle \nabla(f)(X+e^{-t}rx);xe^{-t}\rangle\\
&\quad\quad\times \left(f\left(X+e^{-t}r x\right)-f\left(X\right)\right)dr\\
&=\int_0^{+\infty} \frac{k_x(r)}{r}2 \bbe\langle \nabla(f)(X+e^{-t}rx);rxe^{-t}\rangle\\
&\quad\quad\times \left(f\left(X+e^{-t}r x\right)-f\left(X\right)\right)dr\\
&=\int_0^{+\infty} \frac{k_x(r)}{r} \left(-\dfrac{d}{dt}\bbe\left(f\left(X+e^{-t}r x\right)-f\left(X\right)\right)^2\right)dr.
\end{align*}
Finally, integrating with respect to $t$ (between $0$ and $+\infty$) leads to
\begin{align}\label{eq:PoincStable}
\bbe f(X)^2\leq \bbe \int_{\bbr^d} \left(f\left(X+u\right)-f\left(X\right)\right)^2 \nu(du).
\end{align}
\end{proof}
\noindent

\begin{rem}\label{rem:PoincStable}
(i) Let $X$ be a rotationally invariant $\alpha$-stable random vector, $\alpha \in (1,2)$, with characteristic function $\varphi$ given by
\begin{align*}
\varphi(\xi)=\exp\left(-\|\xi\|^\alpha/2\right),\quad \xi \in \bbr^d.
\end{align*}
Then, by Proposition \ref{prop:poincSD},  for all $f\in \mathcal{S}(\bbr^d)$ with $\bbe f(X)=0$
\begin{align*}
\bbe f(X)^2 \leq c_{\alpha, d}\bbe \int_{\bbr^d} (f(X+u)-f(X))^2  \frac{du}{\|u\|^{\alpha+d}},
\end{align*}
where $c_{\alpha, d}=\dfrac{-\alpha (\alpha-1)\Gamma((\alpha+d)/2)}{4\cos(\alpha\pi/2)\Gamma((\alpha+1)/2)\pi^{(d-1)/2}\Gamma(2-\alpha)}$.\\
(ii) Throughout the proof of Proposition \ref{prop:poincSD}, the following integration by parts formula has been obtained and used, for all $f\in \mathcal{C}_c^{\infty}(\bbr^d)$,
\begin{align*}
\int_{\bbr^d} f(x) \left(-\mathcal{A}(f)(x)\right) \mu_X(dx)= \int_{\bbr^d} \Gamma(f,f)(x) \mu_X(dx),
\end{align*}
where $\mu_X$ is the law of $X$ and $\Gamma$ is a bilinear symmetric application defined, for all $f,g\in \mathcal{C}_c^{\infty}(\bbr^d)$ and all $x\in \bbr^d$, by
\begin{align*}
\Gamma(f,g)(x)&=\frac{1}{2}\int_{\bbr^d} \left(f(x+u)-f(x)\right)\left(g(x+u)-g(x)\right)\tilde{\nu}(du),\\
&=\frac{1}{2} \int_{\bbr^d} \int_{\bbr^d} \mathcal{F}(f)(\xi)\mathcal{F}(g)(\zeta) e^{i \langle x; (\zeta +\xi)\rangle }\sigma_{\tilde{\nu}}(\xi,\zeta) \frac{d\xi d\zeta}{(2\pi)^{2d}},
\end{align*}
with $\sigma_{\tilde{\nu}}(\xi,\zeta)= \int_{\bbr^d} \left(e^{i \langle u;\xi \rangle}-1\right)\left(e^{i \langle u;\zeta \rangle}-1\right) \tilde{\nu}(du)$, for $\xi, \zeta \in \bbr^d$. A straightforward computation in the Fourier domain shows that this bilinear symmetric application is the "carr\'e du champs" operator associated with the generator $\mathcal{A}$ of the semigroup $(P^\nu_t)_{t\geq 0}$ (see, e.g., \cite{BGL14} for a thorough exposition of these topics in the setting of Markov diffusion operators). Namely, for all $f,g\in \mathcal{C}_c^{\infty}(\bbr^d)$ and all $x\in \bbr^d$,
\begin{align*}
\Gamma(f,g)(x)= \frac{1}{2}\left(\mathcal{A}(fg)(x)-f(x)\mathcal{A}(g)(x)-g(x)\mathcal{A}(f)(x)\right).
\end{align*}
\end{rem}
\noindent
Standard objects of interest in the setting of Markov diffusion operators are iterated "carr\'e du champs" of any orders $n\geq 1$ defined through the following recursive formula, for all $f,g\in \mathcal{C}_c^{\infty}(\bbr^d)$ and all $x\in \bbr^d$
\begin{align}\label{eq:GammaIte}
\Gamma_n(f,g)(x)=\frac{1}{2}\bigg(\mathcal{A}(\Gamma_{n-1}(f,g))(x)-\Gamma_{n-1}(\mathcal{A}(f),g)(x)-\Gamma_{n-1}(\mathcal{A}(g),f)(x) \bigg),
\end{align}
with the convention that $\Gamma_0(f,g)(x)=f(x)g(x)$ and $\Gamma_1(f,g)(x)=\Gamma(f,g)(x)$, for $f,g\in \mathcal{C}_c^{\infty}(\bbr^d)$, and $x\in \bbr^d$. The forthcoming simple lemma provides a representation of the $\Gamma_2$ as a pseudo-differential operator whose symbol is completely explicit. 

\begin{lem}\label{lem:GammaIte}
Let $\nu$ be a L\'evy measure on $\bbr^d$ such that
\begin{align*}
\int_{\|u\|\geq 1}\|u\|\nu(du)<+\infty,\quad \nu(du)=\bbone_{(0,+\infty)}(r) \bbone_{\mathbb{S}^{d-1}}(x)\frac{k_x(r)}{r}dr\sigma(dx),
\end{align*}
where $\sigma$ is a finite positive measure on $\mathbb{S}^{d-1}$ and where $k_x(r)$ is a nonnegative continuous function decreasing in $r\in(0,+\infty)$, continuous in $x\in \mathbb{S}^{d-1}$ with
\begin{align*}
\underset{r\rightarrow +\infty}{\lim}k_x(r)=0,\quad \underset{r\rightarrow 0^+}{\lim}r^2k_x(r)=0, \quad x\in\mathbb{S}^{d-1}.
\end{align*}
Let $\mathcal{A}$ be the operator defined, for all $f\in \mathcal{S}(\bbr^d)$ and all $x\in \bbr^d$, by 
\begin{align*}
\mathcal{A}(f)(x)= -\langle x;\nabla(f)(x) \rangle + \int_{\bbr^d} \langle \nabla (f)(x+u)- \nabla (f)(x);u\rangle \nu(du) 
\end{align*}
Then, for all $f,g\in \mathcal{S}(\bbr^d)$ and all $x\in \bbr^d$
\begin{align}\label{eq:GammaIten}
\Gamma_2(f,g)(x)=\int_{\bbr^d}\int_{\bbr^d} \mathcal{F}(f)(\xi)\mathcal{F}(g)(\zeta) e^{i \langle x; \xi +\zeta \rangle}\bigg( \dfrac{\sigma_{\tilde{\nu}}(\xi ,\zeta)^2}{4}+ \dfrac{\rho_{\tilde{\nu}}(\xi,\zeta)}{4} \bigg)\dfrac{d\xi d\zeta}{(2\pi)^{2d}},
\end{align}
where $\sigma_{\tilde{\nu}}(\xi ,\zeta)$ and $\rho_{\tilde{\nu}}(\xi,\zeta)$ are given, for all $\xi,\zeta \in \bbr^d$, by
\begin{align*}
&\sigma_{\tilde{\nu}}(\xi ,\zeta) = \int_{\bbr^d} \left(e^{i \langle u; \xi \rangle}-1\right)\left(e^{i \langle u; \zeta \rangle}-1\right) \tilde{\nu}(du),\\
&\rho_{\tilde{\nu}}(\xi,\zeta)=\bigg(\int_{\bbr^d} i\langle u;\zeta \rangle e^{i \langle u; \zeta \rangle} \left(e^{i \langle u;\xi \rangle} -1 \right) \tilde{\nu}(du)\bigg)+\bigg(\int_{\bbr^d} i\langle u;\xi \rangle e^{i \langle u; \xi \rangle} \left(e^{i \langle u;\zeta \rangle} -1 \right) \tilde{\nu}(du)\bigg),
\end{align*}
with $\tilde{\nu}(du)= \bbone_{(0,+\infty)}(r) \bbone_{\mathbb{S}^{d-1}}(x) (-dk_x(r))\sigma(dx)$.
\end{lem}

\begin{proof}
First, by definition, for all $f,g \in \mathcal{S}(\bbr^d)$ and all $x\in \bbr^d$,
\begin{align*}
\Gamma_2(f,g)(x)=\frac{1}{2}\bigg(\mathcal{A}(\Gamma_{1}(f,g))(x)-\Gamma_{1}(\mathcal{A}(f),g)(x)-\Gamma_{1}(\mathcal{A}(g),f)(x) \bigg).
\end{align*} 
Let us compute $\Gamma_{1}(\mathcal{A}(f),g)(x)$. Using the Fourier representation, for all $x\in\bbr^d$,
\begin{align*}
2\Gamma_{1}(\mathcal{A}(f),g)(x)=\int_{\bbr^d}\int_{\bbr^d}\mathcal{F}(g)(\zeta) \mathcal{F}(\mathcal{A}(f))(\xi)\sigma_{\tilde{\nu}}(\xi,\zeta) e^{i \langle x; \xi +\zeta \rangle} \dfrac{d\xi d\zeta}{(2\pi)^{2d}}.
\end{align*}
Now, recall that, for all $x\in \bbr^d$,
\begin{align*}
\mathcal{A}(f)(x)=-\langle x;\nabla(f)(x)\rangle+ \Delta_\nu(f)(x),
\end{align*}
so that, for all $\xi\in \bbr^d$,
\begin{align*}
 \mathcal{F}(\mathcal{A}(f))(\xi)= \mathcal{F}\left(-\langle x;\nabla(f)\rangle\right)(\xi)+\mathcal{F}\left(f\right)(\xi) \sigma_\nu(\xi).
\end{align*}
Thus, for all $x\in \bbr^d$,
\begin{align*}
2\Gamma_{1}(\mathcal{A}(f),g)(x)&=\int_{\bbr^d}\int_{\bbr^d}\mathcal{F}(g)(\zeta)\mathcal{F}\left(f\right)(\xi) \sigma_\nu(\xi)\sigma_{\tilde{\nu}}(\xi,\zeta) e^{i \langle x; \xi +\zeta \rangle} \dfrac{d\xi d\zeta}{(2\pi)^{2d}}\\
&\quad\quad+ \int_{\bbr^d}\int_{\bbr^d}\mathcal{F}(g)(\zeta)\mathcal{F}\left(-\langle x;\nabla(f)\rangle\right)(\xi)\sigma_{\tilde{\nu}}(\xi,\zeta) e^{i \langle x; \xi +\zeta \rangle} \dfrac{d\xi d\zeta}{(2\pi)^{2d}}.
\end{align*}
Similarly, for all $x\in \bbr^d$,
\begin{align*}
2\Gamma_{1}(\mathcal{A}(g),f)(x)&=\int_{\bbr^d}\int_{\bbr^d}\mathcal{F}(g)(\zeta)\mathcal{F}\left(f\right)(\xi) \sigma_\nu(\zeta)\sigma_{\tilde{\nu}}(\xi,\zeta) e^{i \langle x; \xi +\zeta \rangle} \dfrac{d\xi d\zeta}{(2\pi)^{2d}}\\
&\quad\quad+ \int_{\bbr^d}\int_{\bbr^d}\mathcal{F}(f)(\zeta)\mathcal{F}\left(-\langle x;\nabla(g)\rangle\right)(\xi)\sigma_{\tilde{\nu}}(\xi,\zeta) e^{i \langle x; \xi +\zeta \rangle} \dfrac{d\xi d\zeta}{(2\pi)^{2d}}.
\end{align*}
Next, for all $x\in \bbr^d$
\begin{align*}
\mathcal{A}(\Gamma_{1}(f,g))(x)=-\langle x; \nabla\left(\Gamma_{1}(f,g)\right)(x) \rangle+ \Delta_\nu(\Gamma_{1}(f,g))(x).
\end{align*}
At first, observe that,
\begin{align*}
\Delta_\nu(\Gamma_{1}(f,g))(x)&= \int_{\bbr^d} \mathcal{F}(\Gamma_{1}(f,g)) (\xi) e^{i \langle x;\xi \rangle} \sigma_{\nu}(\xi) \frac{d\xi}{(2\pi)^d}\\
&= \frac{1}{2} \int_{\bbr^d}  \int_{\bbr^d} \mathcal{F}(f)(\xi) \mathcal{F}(g)(\zeta) e^{i \langle x;\xi +\zeta \rangle} \sigma_{\tilde{\nu}}(\xi,\zeta)\sigma_{\nu}(\xi +\zeta) \frac{d\xi d\zeta}{(2\pi)^{2d}}.
\end{align*}
Next, by straightforward computations, 
\begin{align*}
-\langle x; \nabla\left(\Gamma_{1}(f,g)\right)(x) \rangle &= \frac{1}{2} \bigg( \int_{\bbr^d} \tilde{\nu}(du) \left(-\langle x; \nabla(f)(x+u)\rangle+\langle x; \nabla(f)(x)\rangle\right)\left( g(x+u)-g(x)\right) \\
&\quad\quad +\int_{\bbr^d} \tilde{\nu}(du) \left(-\langle x; \nabla(g)(x+u)\rangle+\langle x; \nabla(g)(x)\rangle\right)\left( f(x+u)-f(x)\right)\bigg)\\
&=\frac{1}{2} \bigg( \int_{\bbr^d} \tilde{\nu}(du) \left(-\langle x+u; \nabla(f)(x+u)\rangle+\langle x; \nabla(f)(x)\rangle\right)\left( g(x+u)-g(x)\right) \\
&\quad\quad +\int_{\bbr^d} \tilde{\nu}(du) \left(-\langle x+u; \nabla(g)(x+u)\rangle+\langle x; \nabla(g)(x)\rangle\right)\left( f(x+u)-f(x)\right)\\
&\quad \quad+ \int_{\bbr^d} \tilde{\nu}(du) \langle u; \nabla(g)(x+u)\rangle\left( f(x+u)-f(x)\right)\\
&\quad\quad+ \int_{\bbr^d} \tilde{\nu}(du) \langle u; \nabla(f)(x+u)\rangle\left( g(x+u)-g(x)\right) \bigg)\\
&=\frac{1}{2} \bigg(  \int_{\bbr^d} \int_{\bbr^d} \mathcal{F}\left(-\langle x;\nabla(f)\rangle\right)(\xi) \mathcal{F}(g)(\zeta) e^{i \langle x;\xi +\zeta \rangle} \sigma_{\tilde{\nu}}(\xi, \zeta) \dfrac{d\xi d\zeta}{(2\pi)^{2d}}\\
&\quad\quad + \int_{\bbr^d} \int_{\bbr^d} \mathcal{F}\left(-\langle x;\nabla(g)\rangle\right)(\xi) \mathcal{F}(f)(\zeta) e^{i \langle x;\xi +\zeta \rangle} \sigma_{\tilde{\nu}}(\xi, \zeta) \dfrac{d\xi d\zeta}{(2\pi)^{2d}}\\
&\quad\quad  + \int_{\bbr^d} \int_{\bbr^d} \mathcal{F}\left(g\right)(\xi) \mathcal{F}(f)(\zeta) e^{i \langle x; \xi+\zeta \rangle} \bigg(\int_{\bbr^d} i\langle u;\xi \rangle e^{i \langle u; \xi \rangle} \left(e^{i \langle u;\zeta \rangle} -1 \right) \tilde{\nu}(du)\bigg) \dfrac{d\xi d\zeta}{ (2\pi)^{2d}}\\
&\quad\quad+ \int_{\bbr^d} \int_{\bbr^d} \mathcal{F}\left(g\right)(\xi) \mathcal{F}(f)(\zeta) e^{i \langle x; \xi+\zeta \rangle} \bigg(\int_{\bbr^d} i\langle u;\zeta \rangle e^{i \langle u; \zeta \rangle} \left(e^{i \langle u;\xi \rangle} -1 \right) \tilde{\nu}(du)\bigg) \dfrac{d\xi d\zeta}{ (2\pi)^{2d}}\bigg).
\end{align*}
Then, for all $x\in \bbr^d $,
\begin{align*}
\mathcal{A}(\Gamma_{1}(f,g))(x)-\Gamma_{1}(\mathcal{A}(f),g)(x)-\Gamma_{1}(\mathcal{A}(g),f)(x)&= \frac{1}{2} \int_{\bbr^d}  \int_{\bbr^d} \mathcal{F}(f)(\xi) \mathcal{F}(g)(\zeta) e^{i \langle x;\xi +\zeta \rangle} \sigma_{\tilde{\nu}}(\xi,\zeta)^2 \frac{d\xi d\zeta}{(2\pi)^{2d}}\\
&\quad\quad + \frac{1}{2} \int_{\bbr^d}  \int_{\bbr^d} \mathcal{F}(f)(\xi) \mathcal{F}(g)(\zeta) e^{i \langle x;\xi +\zeta \rangle} \rho_{\tilde{\nu}}(\xi,\zeta) \frac{d\xi d\zeta}{(2\pi)^{2d}},
\end{align*}
where, for all $\xi ,\zeta\in \bbr^d$
\begin{align*}
\rho_{\tilde{\nu}}(\xi,\zeta)=\bigg(\int_{\bbr^d} i\langle u;\zeta \rangle e^{i \langle u; \zeta \rangle} \left(e^{i \langle u;\xi \rangle} -1 \right) \tilde{\nu}(du)\bigg)+\bigg(\int_{\bbr^d} i\langle u;\xi \rangle e^{i \langle u; \xi \rangle} \left(e^{i \langle u;\zeta \rangle} -1 \right) \tilde{\nu}(du)\bigg).
\end{align*}
This concludes the proof of the lemma.
\end{proof}

\begin{rem}\label{rem:gamma2stable}
Using elements of the integral calculus on the sphere (see, e.g., \cite[Appendix $D.3$]{G08}), it is possible to compute the symbol of the $\Gamma_2$ operator when $\nu$ is the L\'evy measure of a rotationally invariant $\alpha$-stable random vector with $\alpha \in (1,2)$. Indeed, Let $\alpha \in (1,2)$ and let 
\begin{align*}
\nu_\alpha(du):= \bbone_{(0,+\infty)}(r) \bbone_{\mathbb{S}^{d-1}}(x) \dfrac{c_{\alpha,d}}{r^{1+\alpha}} dr\lambda(dx)
\end{align*}
where $\lambda$ is the uniform measure on $\mathbb{S}^{d-1}$ and where $c_{\alpha,d}$ is given by
\begin{align*}
c_{\alpha,d}= \dfrac{-\alpha (\alpha-1) \Gamma((\alpha+d)/2)}{4\cos (\alpha\pi/2)\Gamma((\alpha+1)/2)\pi^{(d-1)/2}\Gamma(2-\alpha)},
\end{align*}
so, that $\tilde{\nu}_\alpha$ is given by
\begin{align*}
\tilde{\nu}_\alpha(du)=\bbone_{(0,+\infty)}(r) \bbone_{\mathbb{S}^{d-1}}(x)\dfrac{\alpha c_{\alpha,d}}{r^{1+\alpha}}dr\lambda(dx).
\end{align*}
Then, for all $\xi, \zeta \in \bbr^d$
\begin{align*}
\sigma_{\tilde{\nu}_\alpha}(\xi,\zeta)= \dfrac{\alpha}{2} \left(\|\xi\|^{\alpha}+\|\zeta\|^{\alpha}-\|\xi+\zeta\|^{\alpha}\right),\quad \rho_{\tilde{\nu}_\alpha}(\xi,\zeta)=  \dfrac{\alpha^2}{2} \left(\|\xi\|^{\alpha}+\|\zeta\|^{\alpha}-\|\xi+\zeta\|^{\alpha}\right),
\end{align*}
implying that the symbol of the $\Gamma_2$ operator denoted by $\gamma_2$ is given, for all $\xi, \zeta \in \bbr^d$, by
\begin{align*}
\gamma_2(\xi, \zeta)= \dfrac{\alpha^2}{16} \left( \|\xi\|^{\alpha}+\|\zeta\|^{\alpha}-\|\xi+\zeta\|^{\alpha}\right)^2+\dfrac{\alpha^2}{8}  \left( \|\xi\|^{\alpha}+\|\zeta\|^{\alpha}-\|\xi+\zeta\|^{\alpha}\right).
\end{align*}
As $\alpha \longrightarrow 2^-$, $\gamma_2(\xi, \zeta)\rightarrow \langle i \xi;i\zeta \rangle+ \left(\langle i \xi;i\zeta \rangle\right)^2 $, for all $\xi, \zeta \in \bbr^d$, and as such one retrieves the symbol of the iterated "carr\'e du champs" operator of order $2$ associated with the Ornstein-Uhlenbeck (OU) semigroup. Moreover, the Ornstein-Uhlenbeck $\Gamma_2$ operator clearly dominates, as well known, the associated $\Gamma$ operator which is a typical instance of the Bakry-Emery criterion, implying hypercontractivity of the OU semigroup (see, e.g., \cite{BE85,BGL14}).
\end{rem}
\noindent
The next proposition asserts that the Bakry-Emery criterion still holds for the rotationally invariant $\alpha$-stable distribution with $\alpha \in (1,2)$.

\begin{prop}\label{prop:BEstable}
Let $\alpha \in (1,2)$ and let $X_\alpha$ be a rotationally invariant $\alpha$-stable random vector of $\bbr^d$ with law $\mu_\alpha$ and with L\'evy measure given by
\begin{align*}
\nu_\alpha(du)= \bbone_{(0,+\infty)}(r) \bbone_{\mathbb{S}^{d-1}}(x) \dfrac{c_{\alpha,d}}{r^{1+\alpha}} dr\lambda(dx),
\end{align*}
where $\lambda$ is the uniform measure on $\mathbb{S}^{d-1}$ and where
\begin{align*}
c_{\alpha,d}= \dfrac{-\alpha (\alpha-1) \Gamma((\alpha+d)/2)}{4\cos (\alpha\pi/2)\Gamma((\alpha+1)/2)\pi^{(d-1)/2}\Gamma(2-\alpha)}.
\end{align*}
Then, for all $f \in \mathcal{C}_c^{\infty}\left(\bbr^d\right)$
\begin{align*}
\Gamma_2(f,f) \geq \frac{\alpha}{2}\, \Gamma(f,f),
\end{align*}
where $\Gamma$ and $\Gamma_2$ are respectively the "carr\'e du champs" operator and the iterated "carr\'e du champs" operator of order $2$ associated with $\nu_\alpha$. 
\end{prop}

\begin{proof}
By Remark \ref{rem:gamma2stable}, observe that, for all $\xi,\zeta \in \bbr^d$,
\begin{align*}
\rho_{\tilde{\nu}_{\alpha}}(\xi, \zeta) =\alpha \sigma_{\tilde{\nu}_{\alpha}}(\xi, \zeta),
\end{align*}
where,
\begin{align*}
\sigma_{\tilde{\nu}_{\alpha}}(\xi, \zeta) = \int_{\bbr^d} \left(e^{i \langle u;\xi \rangle}-1 \right)\left( e^{i \langle u;\zeta \rangle}-1\right) \tilde{\nu}_{\alpha}(du).
\end{align*}
Then, by Lemma \ref{lem:GammaIte} and Fourier inversion, for all $f\in \mathcal{C}_c^{\infty}\left(\bbr^d\right)$, and all $x\in \bbr^d$,
\begin{align*}
\Gamma_2(f,f)(x)&=\frac{1}{4} \int_{\bbr^d} \int_{\bbr^d} \left(f(x+u+v)-f(x+u)-f(x+v)+f(x)\right)^2\tilde{\nu}_{\alpha}(du)\tilde{\nu}_{\alpha}(du)\\
&\quad\quad+\frac{\alpha}{4} \int_{\bbr^d} (f(x+u)-f(x))^2 \tilde{\nu}_\alpha(du),
\end{align*}
with $\tilde{\nu}_\alpha(du)= \bbone_{(0,+\infty)}(r) \bbone_{\mathbb{S}^{d-1}}(x) \dfrac{\alpha c_{\alpha,d}}{r^{1+\alpha}} dr\lambda(dx)$. Similarly, for all $f\in \mathcal{C}_c^{\infty}\left(\bbr^d\right)$ and all $x\in \bbr^d$,
\begin{align*}
\Gamma(f,f)(x)=\frac{1}{2} \int_{\bbr^d} (f(x+u)-f(x))^2 \tilde{\nu}_\alpha(du).
\end{align*}
Thus, for all $f\in \mathcal{C}_c^{\infty}\left(\bbr^d\right)$ and all $x\in \bbr^d$,
\begin{align*}
\Gamma_2(f,f)(x) \geq \frac{\alpha}{2} \Gamma(f,f)(x).
\end{align*}
\end{proof}
\noindent
Let us study rigidity and stability phenomena for the rotationally invariant $\alpha$-stable distributions with $\alpha \in (1,2)$ based on the Poincar\'e-type inequality of Proposition \ref{prop:poincSD}. To reach such results let us adopt a spectral point of view.~This is a natural strategy to obtain sharp forms of geometric and functional inequalities as done, e.g., in \cite{BrPr12,PhiFig17,BrPh17}. First, observe that, since $\alpha \in (1,2)$ and since $X_\alpha$ considered in Proposition \ref{prop:BEstable} is centered, the function $g(x)=x$, $x\in \bbr^d$, is an eigenfunction of the semigroup of operators $(P^\nu_t)_{t\geq 0}$ given in \eqref{eq:semSD} with $\nu=\nu_\alpha$ as in \eqref{eq:nustable}. Indeed, for all $x\in \bbr^d$ and all $t\geq 0$
\begin{align*}
P^\nu_t(g)(x)=e^{-t} x.
\end{align*}
Then, by its very definition, $\mathcal{A}(g)(x)=-g(x)$, for $x\in \bbr^d$, so that $g$ is an eigenfunction of $\mathcal{A}$ with associated eigenvalue $-1$. However, since $\alpha \in (1,2)$, $g$ does not belong to $L^2(\mu_\alpha)$, with $\mu_\alpha$ being the law of $X_\alpha$. To circumvent this fact, let us build an optimizing sequence by a smooth truncation procedure. For all $j \in \{1,\dots, d\}$ and all $R\geq 1$, let $g_{R,j}$ be defined, for all $x\in \bbr^d$, by
\begin{align}\label{eq:gRj}
g_{R,j}(x)=x_j \psi\left(\frac{x}{R}\right),
\end{align}
with $\psi\in\mathcal{S}(\bbr^d)$, $\psi(0)=1$ and $0\leq \psi(x)\leq 1$, for $x\in \bbr^d$. Take, for instance, $\psi(x)=\exp(-\|x\|^2)$, for $x\in \bbr^d$. Now, let us state some straightforward facts about the functions $g_{R,j}$: for all $j \in \{1,\dots, d\}$, $\bbe g_{R,j}(X_\alpha)= 0$ and, as $R\rightarrow +\infty$,
\begin{align*}
\bbe g_{R,j}^2(X_\alpha)\longrightarrow +\infty,\quad  \bbe \int_{\bbr^d} |g_{R,j}(X_\alpha+u)-g_{R,j}(X_\alpha)|^2 \tilde{\nu}_\alpha(du) \longrightarrow +\infty.
\end{align*}
Next, by studying precisely the rate at which both the last two terms diverge, we intend to prove that, for all $j \in \{1, \dots, d\}$,
\begin{align*}
\dfrac{\operatorname{Var}(g_{R,j}(X_\alpha))}{\bbe \int_{\bbr^d} |g_{R,j}(X_\alpha+u)-g_{R,j}(X_\alpha)|^2  \tilde{\nu}_\alpha(du)} \underset{R\rightarrow +\infty}{\longrightarrow} \frac{1}{\alpha}.
\end{align*}
The first technical lemma investigate the rate at which $\bbe g_{R,j}(X_\alpha)^2$ diverges as $R$ tends to $+\infty$.

\begin{lem}\label{lem:rate1}
Let $\psi(x)= \exp(-\|x\|^2)$, for $x\in \bbr^d$. Let $\alpha \in (1,2)$ and $X_\alpha$ be a rotationally invariant $\alpha$-stable random vector of $\bbr^d$ with characteristic function $\varphi$ given, for all $\xi \in \bbr^d$, by
\begin{align*}
\varphi(\xi)=\exp\left(-\dfrac{\|\xi\|^\alpha}{2}\right).
\end{align*}
Then, for all $j\in \{1, \dots, d\}$, as $R$ tends to $+\infty$,
\begin{align}\label{eq:lemrate1}
\dfrac{\bbe g_{R,j}^2(X_\alpha)}{R^{2-\alpha}}\longrightarrow \dfrac{\alpha}{2}\int_{\bbr^d}\mathcal{F}(\psi^2)(\xi)\bigg(\|\xi\|^{\alpha-2}+\xi_j^2 (\alpha-2)\|\xi\|^{\alpha-4}\bigg)\frac{d\xi }{(2\pi)^d},
\end{align}
where $g_{R,j}(x)=x_j \psi(x/R)$, for $x\in \bbr^d$.
\end{lem}

\begin{proof}
First, for all $R\geq 1$, set $\psi_R(x)=\psi(x/R)$, for $x\in \bbr^d$, and, for all $j\in \{1, \dots, d\}$,
\begin{align*}
\bbe\, g_{R,j}^2(X_\alpha)=\bbe X_{\alpha, j}^2 \psi^2_R\left(X_\alpha\right),
\end{align*}
where $X_\alpha$ is a rotationally invariant $\alpha$-stable random vector with $\alpha \in (1,2)$ and $X_{\alpha, j}$ is its $j$-th coordinate. By Fubini's theorem, standard Fourier analysis, two integrations by parts and a change of variables, it follows that
\begin{align}\label{eq:rate1}
\bbe\, g_{R,j}^2(X_\alpha) &= \frac{1}{(2\pi)^d} \int_{\bbr^d} \mathcal{F} \left(x_j^2 \psi^2_R\right)(\xi) e^{-\frac{\|\xi\|^{\alpha}}{2}}d\xi\nonumber\\
&=\frac{1}{(2\pi)^d} \int_{\bbr^d} -\partial^2_{\xi_j} \left(\mathcal{F}(\psi^2_R)\right) (\xi)e^{-\frac{\|\xi\|^{\alpha}}{2}}d\xi\nonumber\\
&=-\frac{1}{(2\pi)^d} \int_{\bbr^d}\mathcal{F}(\psi^2_R)(\xi)\left(-\dfrac{\alpha}{2} \|\xi\|^{\alpha-2}-\frac{\alpha}{2}\xi_j^2 (\alpha-2)\|\xi\|^{\alpha-4}+\frac{\alpha^2}{4}\xi_j^2 \|\xi\|^{2(\alpha-2)}\right) e^{-\frac{\|\xi\|^{\alpha}}{2}}d\xi\nonumber\\
&=-\frac{R^d}{(2\pi)^d} \int_{\bbr^d}\mathcal{F}(\psi^2)(R\xi)\left(-\dfrac{\alpha}{2} \|\xi\|^{\alpha-2}-\frac{\alpha}{2}\xi_j^2 (\alpha-2)\|\xi\|^{\alpha-4}+\frac{\alpha^2}{4}\xi_j^2 \|\xi\|^{2(\alpha-2)}\right) e^{-\frac{\|\xi\|^{\alpha}}{2}}d\xi\nonumber\\
&=-\frac{1}{(2\pi)^d} \int_{\bbr^d}\mathcal{F}(\psi^2)(\xi)\bigg(-\dfrac{\alpha}{2}R^{2-\alpha} \|\xi\|^{\alpha-2}-\frac{\alpha}{2}R^{2-\alpha}\xi_j^2 (\alpha-2)\|\xi\|^{\alpha-4}\nonumber\\
&\quad\quad+\frac{\alpha^2}{4}R^{2-2\alpha}\xi_j^2 \|\xi\|^{2(\alpha-2)}\bigg) e^{-\frac{\|\xi\|^{\alpha}}{(2R^\alpha)}}d\xi\nonumber\\
&=\frac{R^{2-\alpha}}{(2\pi)^d}\int_{\bbr^d}\mathcal{F}(\psi^2)(\xi)\bigg(\dfrac{\alpha}{2}\|\xi\|^{\alpha-2}+\frac{\alpha}{2}\xi_j^2 (\alpha-2)\|\xi\|^{\alpha-4}-\frac{\alpha^2}{4}R^{-\alpha}\xi_j^2 \|\xi\|^{2(\alpha-2)}\bigg) e^{-\frac{\|\xi\|^{\alpha}}{(2R^\alpha)}}d\xi,
\end{align}
where $\partial^2_{\xi_j}$ is the partial derivative of order $2$ in the $\xi_j$ coordinate. Moreover, since $\alpha \in (1,2)$ and since $\psi^2 \in \mathcal{S}(\bbr^d)$, all the following integrals converge
\begin{align*}
\int_{\bbr^d}\mathcal{F}(\psi^2)(\xi) \|\xi\|^{\alpha-2} d\xi <+\infty, \quad \int_{\bbr^d}\mathcal{F}(\psi^2)(\xi)|\xi_j|^2\|\xi\|^{\alpha-4}d\xi <+\infty, \quad \int_{\bbr^d}\mathcal{F}(\psi^2)(\xi) |\xi_j|^2 \|\xi\|^{2(\alpha-2)} d\xi <+\infty.
\end{align*}
Hence, as $R\longrightarrow+\infty$,
\begin{align*}
\dfrac{\bbe g_{R,j}^2(X_\alpha)}{R^{2-\alpha}}\longrightarrow \frac{1}{(2\pi)^d}\int_{\bbr^d}\mathcal{F}(\psi^2)(\xi)\bigg(\dfrac{\alpha}{2}\|\xi\|^{\alpha-2}+\frac{\alpha}{2}\xi_j^2 (\alpha-2)\|\xi\|^{\alpha-4}\bigg)d\xi,
\end{align*}
which concludes the proof of the lemma.
\end{proof}
\noindent
This second technical lemma provides the rate of divergence, as $R$ tends to $+\infty$, of
\begin{align*}
\bbe \int_{\bbr^d} |g_{R,j}(X_\alpha+u)-g_{R,j}(X_\alpha)|^2\tilde{\nu}_\alpha(du)=2 \bbe\, \Gamma\left(g_{R,j},g_{R,j}\right)(X_\alpha),
\end{align*}
for all $j \in \{1,\dots,d\}$.

\begin{lem}\label{lem:rate2}
Let $\psi(x)= \exp(-\|x\|^2)$, for $x\in \bbr^d$. Let $\alpha \in (1,2)$ and $X_\alpha$ be a rotationally invariant $\alpha$-stable random vector of $\bbr^d$ with characteristic function $\varphi$ given, for all $\xi \in \bbr^d$, by
\begin{align*}
\varphi(\xi)=\exp\left(-\dfrac{\|\xi\|^\alpha}{2}\right).
\end{align*}
Then, for all $j\in \{1, \dots, d\}$, as $R$ tends to $+\infty$,
\begin{align}\label{eq:lemrate2}
\dfrac{\bbe\, \Gamma(g_{R,j},g_{R,j})(X_\alpha)}{R^{2-\alpha}}\longrightarrow \dfrac{\alpha^2}{4} \int_{\bbr^d}\int_{\bbr^d} \mathcal{F}(\psi)(\xi) \mathcal{F}(\psi)(\zeta)\bigg(\|\xi +\zeta\|^{\alpha-2}+(\xi_j+\zeta_j)^2 (\alpha-2) \|\xi +\zeta\|^{\alpha-4}\bigg) \dfrac{d\xi d\zeta}{(2\pi)^{2d}},
\end{align}
where $g_{R,j}(x)=x_j \psi(x/R)$, for $x\in \bbr^d$, and where $\Gamma$ is the "carr\'e du champs" operator associated with $X_\alpha$.
\end{lem}

\begin{proof}
First, by Remark \ref{rem:gamma2stable}, for all $x\in \bbr^d$, all $R\geq 1$ and all $j\in \{1,\dots, d\}$,
\begin{align*}
\Gamma(g_{R,j},g_{R,j})(x)= \dfrac{1}{2} \int_{\bbr^d}\int_{\bbr^d} \mathcal{F}(g_{R,j})(\xi)\mathcal{F}(g_{R,j})(\zeta) e^{i \langle x; \xi +\zeta \rangle} \dfrac{\alpha}{2} \left(\|\xi\|^{\alpha}+\|\zeta\|^{\alpha}-\|\xi+\zeta\|^{\alpha}\right)\dfrac{d\xi d\zeta}{(2\pi)^{2d}}.
\end{align*}
Then, for all $R \geq 1$ and all $j\in \{1,\dots, d\}$,
\begin{align*}
\frac{1}{2}\bbe \int_{\bbr^d} |g_{R,j}(X_\alpha+u)-g_{R,j}(X_\alpha)|^2 \tilde{\nu}_\alpha(du)&= \bbe\, \Gamma(g_{R,j},g_{R,j})(X_\alpha)\\
&= \dfrac{\alpha}{4} \int_{\bbr^d}\int_{\bbr^d} \mathcal{F}(g_{R,j})(\xi)\mathcal{F}(g_{R,j})(\zeta) e^{-\frac{\|\xi +\zeta\|^\alpha}{2}}\\
&\quad\quad\quad\times \left(\|\xi\|^{\alpha}+\|\zeta\|^{\alpha}-\|\xi+\zeta\|^{\alpha}\right)\dfrac{d\xi d\zeta}{(2\pi)^{2d}}\\
&= -\dfrac{\alpha}{4} \int_{\bbr^d}\int_{\bbr^d} \partial_{\xi_j}(\mathcal{F}(\psi_R))(\xi) \partial_{\zeta_j}(\mathcal{F}(\psi_R))(\zeta) e^{-\frac{\|\xi +\zeta\|^\alpha}{2}}\\
&\quad\quad\quad\times \left(\|\xi\|^{\alpha}+\|\zeta\|^{\alpha}-\|\xi+\zeta\|^{\alpha}\right)\dfrac{d\xi d\zeta}{(2\pi)^{2d}}\\
&= -\dfrac{\alpha}{4} \int_{\bbr^d}\int_{\bbr^d} \mathcal{F}(\psi_R)(\xi) \mathcal{F}(\psi_R)(\zeta) \partial^2_{\xi_j,\zeta_j}\bigg(e^{-\frac{\|\xi +\zeta\|^\alpha}{2}}\\
&\quad\quad \times\bigg(\|\xi\|^{\alpha}+\|\zeta\|^{\alpha}-\|\xi+\zeta\|^{\alpha}\bigg)\bigg)\dfrac{d\xi d\zeta}{(2\pi)^{2d}},
\end{align*}
where $\partial^2_{\xi_j,\zeta_j}$ is the second order partial derivative in the coordinates $\xi_j$ and $\zeta_j$. Now, for all $\xi, \zeta \in \bbr^d \setminus \{0\}$ and all $j\in \{1,\dots, d\}$,
\begin{align*}
\dfrac{d}{d\xi_j}\bigg(e^{-\frac{\|\xi +\zeta\|^\alpha}{2}}\bigg(\|\xi\|^{\alpha}+\|\zeta\|^{\alpha}-\|\xi+\zeta\|^{\alpha}\bigg)\bigg)&=e^{-\frac{\|\xi +\zeta\|^\alpha}{2}} \bigg(\dfrac{-\alpha}{2}\left(\xi_j+\zeta_j\right) \|\xi +\zeta\|^{\alpha-2}\bigg(\|\xi\|^{\alpha}+\|\zeta\|^{\alpha}-\|\xi+\zeta\|^{\alpha}\bigg)\\
&\quad\quad +\alpha \bigg(\xi_j \|\xi\|^{\alpha-2}-(\xi_j+\zeta_j)\|\xi +\zeta\|^{\alpha-2}\bigg)\bigg).
\end{align*}
Then, for all $\xi, \zeta \in \bbr^d \setminus \{0\}$ and all $j\in \{1,\dots, d\}$,
\begin{align*}
\dfrac{d^2}{d\xi_j\zeta_j}\bigg(e^{-\frac{\|\xi +\zeta\|^\alpha}{2}}\bigg(\|\xi\|^{\alpha}+\|\zeta\|^{\alpha}-\|\xi+\zeta\|^{\alpha}\bigg)\bigg)&=e^{-\frac{\|\xi +\zeta\|^\alpha}{2}} \bigg(-\dfrac{\alpha}{2}\left(\xi_j+\zeta_j\right) \|\xi +\zeta\|^{\alpha-2}\bigg(\|\xi\|^{\alpha}+\|\zeta\|^{\alpha}-\|\xi+\zeta\|^{\alpha}\bigg)\\
&\quad\quad +\alpha \bigg(\xi_j \|\xi\|^{\alpha-2}-(\xi_j+\zeta_j)\|\xi +\zeta\|^{\alpha-2}\bigg)\bigg)\\
&\quad\quad \times \bigg(\dfrac{-\alpha}{2}\left(\xi_j+\zeta_j\right) \|\xi +\zeta\|^{\alpha-2}\bigg)+e^{-\frac{\|\xi +\zeta\|^\alpha}{2}} \bigg(-\frac{\alpha}{2} \|\xi +\zeta\|^{\alpha-2} \\
&\quad\quad \times\bigg(\|\xi\|^{\alpha}+\|\zeta\|^{\alpha}-\|\xi+\zeta\|^{\alpha}\bigg)-\frac{\alpha}{2} \|\xi +\zeta\|^{\alpha-4} (\xi_j+\zeta_j)^2 \\
&\quad\quad \times\bigg(\|\xi\|^{\alpha}+\|\zeta\|^{\alpha}-\|\xi+\zeta\|^{\alpha}\bigg)(\alpha-2)-\alpha \|\xi +\zeta\|^{\alpha-2}\\
&\quad\quad\quad -\alpha (\xi_j+\zeta_j)^2 (\alpha-2) \|\xi +\zeta\|^{\alpha-4}\bigg).
\end{align*}
For all $\xi,\zeta\in \bbr^d \setminus \{0\}$ and all $j\in\{1,\dots, d\}$, set
\begin{align*}
m_\alpha(\xi,\zeta)&:=\bigg(-\dfrac{\alpha}{2}\left(\xi_j+\zeta_j\right) \|\xi +\zeta\|^{\alpha-2}\bigg(\|\xi\|^{\alpha}+\|\zeta\|^{\alpha}-\|\xi+\zeta\|^{\alpha}\bigg)+\alpha \bigg(\xi_j \|\xi\|^{\alpha-2}-(\xi_j+\zeta_j)\|\xi +\zeta\|^{\alpha-2}\bigg)\bigg) \\
&\quad\quad \times \bigg(-\dfrac{\alpha}{2}\left(\xi_j+\zeta_j\right) \|\xi +\zeta\|^{\alpha-2}\bigg)+\bigg(-\frac{\alpha}{2} \|\xi +\zeta\|^{\alpha-2}\bigg(\|\xi\|^{\alpha}+\|\zeta\|^{\alpha}-\|\xi+\zeta\|^{\alpha}\bigg)\\
&\quad\quad -\frac{\alpha}{2} \|\xi +\zeta\|^{\alpha-4} (\xi_j+\zeta_j)^2\bigg(\|\xi\|^{\alpha}+\|\zeta\|^{\alpha}-\|\xi+\zeta\|^{\alpha}\bigg)(\alpha-2)-\alpha \|\xi +\zeta\|^{\alpha-2}\\
&\quad\quad\quad -\alpha (\xi_j+\zeta_j)^2 (\alpha-2) \|\xi +\zeta\|^{\alpha-4}\bigg).
\end{align*}
Thus, by a change of variables, for all $R\geq 1$ and all $j\in \{1,\dots, d\}$
\begin{align}\label{eq:rate2}
\bbe\, \Gamma(g_{R,j},g_{R,j})(X_\alpha)&= -\dfrac{\alpha}{4} \int_{\bbr^d}\int_{\bbr^d} \mathcal{F}(\psi_R)(\xi) \mathcal{F}(\psi_R)(\zeta) e^{-\frac{\|\xi +\zeta\|^\alpha}{2}}m_\alpha(\xi, \zeta)\dfrac{d\xi d\zeta}{(2\pi)^{2d}}\nonumber\\
&= -\dfrac{\alpha R^{2d}}{4} \int_{\bbr^d}\int_{\bbr^d} \mathcal{F}(\psi)(R\, \xi) \mathcal{F}(\psi)(R\, \zeta) e^{-\frac{\|\xi +\zeta\|^\alpha}{2}}m_\alpha(\xi, \zeta)\dfrac{d\xi d\zeta}{(2\pi)^{2d}}\nonumber\\
&=-\dfrac{\alpha}{4} \int_{\bbr^d}\int_{\bbr^d} \mathcal{F}(\psi)(\xi) \mathcal{F}(\psi)(\zeta) e^{-\frac{\|\xi +\zeta\|^\alpha}{2R^\alpha}}m_\alpha\left(\frac{\xi}{R}, \frac{\zeta}{R}\right)\dfrac{d\xi d\zeta}{(2\pi)^{2d}}.
\end{align}
Finally, normalizing the left-hand side of \eqref{eq:rate2} by $R^{2-\alpha}$ implies, as $R\longrightarrow+\infty$,
\begin{align*}
\dfrac{\bbe\, \Gamma(g_{R,j},g_{R,j})(X_\alpha)}{R^{2-\alpha}}\longrightarrow \dfrac{\alpha^2}{4} \int_{\bbr^d}\int_{\bbr^d} \mathcal{F}(\psi)(\xi) \mathcal{F}(\psi)(\zeta)\bigg(\|\xi +\zeta\|^{\alpha-2}+(\xi_j+\zeta_j)^2 (\alpha-2) \|\xi +\zeta\|^{\alpha-4}\bigg) \dfrac{d\xi d\zeta}{(2\pi)^{2d}},
\end{align*}
which concludes the proof of the lemma.
\end{proof}
\noindent
From the above lemma, and from a spectral point of view, the correct functional to observe rigidity phenomenon for the rotationally invariant $\alpha$-stable distribution, $\alpha \in (1,2)$, is the functional defined, for all $\mu \in M_1(\bbr^d)$ ($M_1(\bbr^d)$ is the set of probability measures on $\bbr^d$), by
\begin{align}\label{eq:Ufunctional}
U_\alpha(\mu):= \underset{f\in \mathcal{H}_\alpha}{\sup} \dfrac{\operatorname{Var}(f(X))}{\bbe \int_{\bbr^d} |f(X+u)-f(X)|^2  \nu_\alpha(du)},
\end{align}
where $X\sim \mu$ and where $\mathcal{H}_\alpha$ is the set of functions $f$ from $\bbr^d$ to $\bbr$ such that $\operatorname{Var}(f(X))<+\infty$ and $0<\bbe \int_{\bbr^d} |f(X+u)-f(X)|^2  \nu_\alpha(du)<+\infty$. Therefore, the next result is a direct consequence of the Poincar\'e-type inequality for the rotationally invariant $\alpha$-stable distribution, $\alpha \in (1,2)$, and of the existence of an optimizing sequence as built above.

\begin{cor}\label{cor:rigidity1} 
Let $X_\alpha$ be a rotationally invariant $\alpha$-stable random vector, $\alpha \in (1,2)$, with law $\mu_\alpha$ and with characteristic function $\varphi$ given by
\begin{align*}
\varphi(\xi)=\exp\left(-\frac{\|\xi\|^\alpha}{2}\right), \quad \xi \in \bbr^d.
\end{align*}
Let $U_\alpha$ be the functional, on $M_1(\bbr^d)$, defined in \eqref{eq:Ufunctional}. Then, 
\begin{align*}
U_\alpha(\mu_\alpha)=1.
\end{align*}
\end{cor}

\begin{proof}
First, by the Poincar\'e-type inequality of Proposition \ref{prop:poincSD}, $U_\alpha(\mu_\alpha)\leq 1$. Moreover, for all $j \in \{1,\dots, d\}$ and for $g_{R,j}$ as in \eqref{eq:gRj} (which clearly belongs to $\mathcal{H}_\alpha$),
\begin{align*}
U_\alpha(\mu_\alpha)\geq \dfrac{\operatorname{Var}(g_{R,j}(X_\alpha))}{\bbe \int_{\bbr^d} |g_{R,j}(X_\alpha+u)-g_{R,j}(X_\alpha)|^2  \nu_\alpha(du)} \underset{R\rightarrow +\infty}{\longrightarrow}1.
\end{align*}
This concludes the proof of the corollary.
\end{proof}
\noindent
To continue, let us state and prove a converse to the above corollary. In particular, note that, for all $j \in \{1,\dots, d\}$,
\begin{align*}
\underset{R\longrightarrow+\infty}{\lim} \left(\operatorname{Var}(g_{R,j}(X_\alpha))-\bbe \int_{\bbr^d} |g_{R,j}(X_\alpha+u)-g_{R,j}(X_\alpha)|^2  \nu_\alpha(du)\right)=0.
\end{align*}
Indeed, this is a direct consequence of \eqref{eq:rate1} and \eqref{eq:rate2} since the divergent terms cancel out and the remaining terms converge to $0$ as $R \rightarrow +\infty$. 

\begin{cor}\label{cor:rigidity2}
Let $\psi(x)= \exp(-\|x\|^2)$, for $x\in \bbr^d$, and let $g_{R,j}$ be given by \eqref{eq:gRj}, for all $R\geq 1$ and all $j\in \{1, \dots, d\}$. Let $X$ be a centered random vector of $\bbr^d$ with finite first moment, with law $\mu$ and such that, for all $j \in \{1,\dots, d\}$,
\begin{align}\label{eq:assrigidity2}
\underset{R\longrightarrow+\infty}{\lim} \left(\operatorname{Var}(g_{R,j}(X))-\bbe \int_{\bbr^d} |g_{R,j}(X+u)-g_{R,j}(X)|^2  \nu_\alpha(du)\right)=0.
\end{align}
If $U_\alpha(\mu)=1$, then $\mu=\mu_\alpha$, where $\mu_\alpha$ is the law of a rotationally invariant $\alpha$-stable random vector with $\alpha \in (1,2)$.
\end{cor}

\begin{proof}
Let $R\geq 1$, let $f$ be a bounded Lipschitz function on $\bbr^d$ and let $\varepsilon \in \bbr$. Since $U_\alpha(\mu)=1$, $\alpha \in (1,2)$ and since $g_{R,j}+\varepsilon f\in \mathcal{H}_\alpha$, for $j\in \{1,\dots, d\}$,
\begin{align*}
\operatorname{Var}(g_{R,j}(X)) +2\varepsilon \operatorname{Cov}(g_{R,j}(X), f(X))&+\varepsilon^2 \operatorname{Var}(f(X))\leq \bbe \int_{\bbr^d} |g_{R,j}(X+u)-g_{R,j}(X)|^2  \nu_\alpha(du)\\
&\quad+2\varepsilon \bbe \int_{\bbr^d} (g_{R,j}(X+u)-g_{R,j}(X))(f(X+u)-f(X))  \nu_\alpha(du)\\
&\quad + \varepsilon^2 \bbe \int_{\bbr^d} |f(X+u)-f(X)|^2  \nu_\alpha(du).
\end{align*}
Now, observe that, for all $j\in \{1,\dots, d\}$,
\begin{align*}
\underset{R\longrightarrow+\infty}{\lim} \bbe \int_{\bbr^d} (g_{R,j}(X+u)-g_{R,j}(X))&(f(X+u)-f(X))  \nu_\alpha(du)= \bbe \int_{\bbr^d} u_j (f(X+u)-f(X))  \nu_\alpha(du),\\
&\quad \underset{R\longrightarrow+\infty}{\lim}\operatorname{Cov}(g_{R,j}(X), f(X)) =\operatorname{Cov}(X_j, f(X)).
\end{align*}
Then, \eqref{eq:assrigidity2} implies that,
\begin{align*}
2\varepsilon \operatorname{Cov}(X_j, f(X))+\varepsilon^2 \operatorname{Var}(f(X))&\leq 2\varepsilon \bbe \int_{\bbr^d} u_j(f(X+u)-f(X)) \nu_\alpha(du)\\
&\quad + \varepsilon^2 \bbe \int_{\bbr^d} |f(X+u)-f(X)|^2  \nu_\alpha(du).
\end{align*}
Since the above inequality is true for all $\varepsilon \in \bbr$, the following covariance representation holds,
\begin{align*}
 \operatorname{Cov}(X_j, f(X))= \bbe \int_{\bbr^d} u_j(f(X+u)-f(X)) \nu_\alpha(du),\quad  j\in \{1,\dots, d\}.
\end{align*}
Theorem \ref{thm:multid} concludes the proof of the corollary.
\end{proof}
\noindent
To end this section, let us investigate stability results for rotationally invariant $\alpha$-stable laws. A natural strategy to reach stability put forward in \cite{CFP19,AH18_2} is to use Stein kernels. This strategy relies on the Lax-Milgram theorem to ensure the existence of Stein kernels under appropriate assumptions. More precisely, the Stein kernel is seen as the solution to a variational problem linked to the covariance identity characterizing the target probability measure. In the sequel, we develop an approach based on Dirichlet forms to obtain the existence of Stein kernels. Adopting the notations, the definitions and the terminology of \cite[Chapter $1$]{FOT10}, let us start with an abstract result which then leads to the existence of Stein kernels in known and in new situations. Note that this result as well as its geometric generalizations and consequences will be further analyzed in the ongoing work \cite{AG19_1}.

\begin{thm}\label{thm:abstract}
Let $H$ be a real Hilbert space with inner product $\langle \cdot ;\cdot \rangle_H$ and induced norm $\| \cdot \|_H$. Let $\mathcal{E}$ be a closed symmetric non-negative definite bilinear form in the sense of \cite{FOT10} with dense linear domain $\mathcal{D}(\mathcal{E})$. Let $\{G_\alpha:\, \alpha>0\}$ and $\{P_t:\, t>0\}$ be, respectively, the strongly continuous resolvent and the strongly continuous semigroup on $H$ associated with $\mathcal{E}$. Moreover, assume that, there exists a closed linear subspace $H_0 \subset H$ such that, for all $t>0$ and all $u \in H_0$,
\begin{align}\label{ineq:Poinc}
\|P_t(u)\|_H \leq e^{-\frac{t}{C_P}}\|u\|_H,
\end{align}
for some $C_P>0$ independent of $u$ and of $t$. Let $G_{0^+}$ be the operator defined by
\begin{align}\label{eq:Gzero}
G_{0^+}(u):=\int_{0}^{+\infty} P_t(u) dt,\quad u \in H_0,
\end{align}
where the above integral is to be understood in the Bochner sense. Then, for all $u\in H_0$, $G_{0^+}(u)$ belongs to $\mathcal{D}(\mathcal{E})$ and, for all $v\in \mathcal{D}(\mathcal{E})$,
\begin{align}\label{eq:SteinKernel}
\mathcal{E} \left(G_{0^+}(u),v\right)=\langle u ; v \rangle_H .
\end{align}
Moreover, for all $u \in H_0$,
\begin{align}\label{ineq:energy}
\mathcal{E}\left(G_{0^+}(u),G_{0^+}(u)\right) \leq \|u\|^2_H C_P. 
\end{align}
\end{thm}

\begin{proof}
First, from \cite[Theorem $1.3.1$]{FOT10}, there is a one to one correspondence between the family of closed symmetric forms on $H$ and the family of non-positive definite self-adjoint operators on $H$. Then, let $\mathcal{A}$, $\{G_\alpha:\, \alpha>0\}$ and $\{P_t:\, t>0\}$ be, respectively, the generator, the strongly continuous resolvent and the strongly continuous semigroup on $H$ associated with $\mathcal{E}$ such that, for all $\alpha >0$ and all $u\in H$,
\begin{align*}
G_\alpha(u)= \int_0^{+\infty}e^{-\alpha t} P_t(u)dt.
\end{align*}
(Again the above integral is to be understood in the Bochner sense.) Then, from \cite[Lemma $1.3.3$]{FOT10}, for all $\alpha >0$, all $u\in H$ and all $v\in \mathcal{D}\left(\mathcal{E}\right)$,
\begin{align}\label{eq:resolvent}
\mathcal{E}(G_\alpha(u),v)+\alpha \langle G_\alpha(u);v\rangle_H=\langle u;v\rangle_H.
\end{align}
Then, in order to establish \eqref{eq:SteinKernel} from \eqref{eq:resolvent}, one needs to pass to the limit in \eqref{eq:resolvent} as $\alpha \longrightarrow 0^+$. First, since \eqref{ineq:Poinc} holds, $G_{0^+}$ given by \eqref{eq:Gzero} is well defined on $H_0$. Moreover, for all $\alpha>0$ and all $u \in H_0$,
\begin{align*}
\|G_\alpha(u) -G_{0^+}(u)\|_H \leq \|u\|_H \int_{0}^{+\infty} (1-e^{-\alpha t})e^{-\frac{t}{C_P}} dt = \|u\|_H \dfrac{\alpha C_P^2}{1+\alpha C_P}.
\end{align*}
Then, $G_\alpha(u)$ converges strongly in $H$ to $G_{0^+}(u)$, as $\alpha$ tends to $0^+$. It therefore follows that, for all $u \in H_0$ and all $v\in H$,
\begin{align*}
\langle G_\alpha(u);v\rangle_H \underset{\alpha \rightarrow 0^+}{\longrightarrow} \langle G_{0^+}(u);v\rangle_H.
\end{align*}
Next, let us prove that, for all $u \in H_0$,
\begin{align}\label{eq:closed}
\mathcal{E}(G_\alpha(u)-G_\beta(u),G_\alpha(u)-G_\beta(u)) \underset{\alpha, \beta \rightarrow 0^+}{\longrightarrow} 0.
\end{align}
First, note that, for all $\alpha, \beta>0$,
\begin{align*}
\mathcal{E}(G_\alpha(u)-G_\beta(u),G_\alpha(u)-G_\beta(u))=\mathcal{E}(G_\alpha(u),G_\alpha(u))+\mathcal{E}(G_\beta(u),G_\beta(u))-2\mathcal{E}(G_\alpha(u),G_\beta(u)).
\end{align*}
Then, from \eqref{eq:resolvent},
\begin{align*}
\mathcal{E}(G_\alpha(u),G_\alpha(u))=\langle u ; G_\alpha(u) \rangle_H -\alpha \langle G_\alpha(u); G_\alpha(u)\rangle_H \underset{\alpha \rightarrow 0^{+}}{\longrightarrow} \langle u ; G_{0^+}(u) \rangle_H,
\end{align*}
and similarly for $\mathcal{E}(G_\beta(u),G_\beta(u))$, as $\beta$ tends to $0^+$. Now, for the crossed term, 
\begin{align*}
\mathcal{E}(G_\alpha(u),G_\beta(u))= \langle u ; G_\beta(u) \rangle_H -\alpha \langle G_\alpha(u); G_\beta(u)\rangle_H \underset{\alpha, \beta \rightarrow 0^+}{\longrightarrow}  \langle u ; G_{0^+}(u) \rangle_H.
\end{align*}
The closedness of $\mathcal{E}$ then ensures that $G_{0^+}(u)$ belongs to $\mathcal{D}(\mathcal{E})$ and that 
\begin{align*}
\mathcal{E}(G_\alpha(u)-G_{0^+}(u),G_\alpha(u)-G_{0^+}(u)) \underset{\alpha \rightarrow 0^+}{\longrightarrow} 0.
\end{align*}
This gives \eqref{eq:SteinKernel}, while the inequality \eqref{ineq:energy} follows from \eqref{eq:SteinKernel}, the Cauchy-Schwarz inequality, the triangle inequality and  \eqref{ineq:Poinc}, concluding the proof of the theorem.
\end{proof}
\noindent
The next remark explores how the absract Theorem \ref{thm:abstract} recovers various known results and provides new ones.

\begin{rem}\label{rem:SteinKernel}
(i) First, let $\gamma$ be the centered Gaussian probability measure on $\bbr^d$ with the identity matrix as its covariance matrix. Let $H$ be the space of $\bbr^d$-valued square-integrable functions on $\bbr^d$ with respect to $\gamma$, let $H_0$ be the functions in $H$ with mean $0$ with respect to $\gamma$ and let $\mathcal{E}$ be the symmetric non-negative definite bilinear form defined, for all $f,g \in \mathcal{C}_c^{\infty}(\bbr^d, \bbr^d)$, by
\begin{align*}
\mathcal{E}(f,g):=\int_{\bbr^d} \langle \nabla(f)(x), \nabla(g)(x) \rangle_{HS} \gamma(dx),
\end{align*}
where $\langle \cdot ;\cdot \rangle_{HS}$ is the Hilbert-Schmidt product for real matrices of size $d\times d$. It is a standard fact of Gaussian analysis that the above form is closable and its closed extension gives rise to the Ornstein-Uhlenbeck generator and its semigroup. Moreover, note that the function, $h(x)=x$, $x\in \bbr^d$, belongs to $H_0$ and that $\gamma$ satisfies the following Poincar\'e inequality: for all smooth $f: \bbr^d\rightarrow \bbr^d$ with $\int_{\bbr^d} f(x) \gamma(dx)=0$,
\begin{align*}
\int_{\bbr^d} \|f(x)\|^2 \gamma(dx) \leq \int_{\bbr^d} \|\nabla(f)(x)\|_{HS}^2 \gamma(dx).
\end{align*}
Then, by Theorem \ref{thm:abstract}, for all $f\in \mathcal{D}(\mathcal{E})$
\begin{align}\label{eq:intbypart}
\mathcal{E}(G_{0^+}(h),f)=\int_{\bbr^d} \langle x; f(x) \rangle \gamma(dx),
\end{align}
where $G_{0^+}(h)$ is given, for all $x\in \bbr^d$, by
\begin{align*}
G_{0^+}(h)(x):=\int_{0}^{+\infty} P_t(h)(x)dt,
\end{align*}
with $(P_t)_{t>0}$ being the Ornstein-Uhlenbeck semigroup. Noting that $P_t(h)(x)=e^{-t} h(x)$, 
\begin{align*}
\mathcal{E}(G_{0^+}(h),f)=\int_{\bbr^d} \langle \nabla(h)(x), \nabla(f)(x) \rangle_{HS} \gamma(dx)= \int_{\bbr^d} \operatorname{div}(f)(x) \gamma(dx)
\end{align*}
where $\operatorname{div}$ is the standard divergence operator. Thus, \eqref{eq:intbypart} is the integration by parts formula associated with $\gamma$.\\
(ii) Let $\mu$ be a centered probability measure on $\bbr^d$ with finite second moment such that, for all smooth $f: \bbr^d\rightarrow \bbr^d$ with $\int_{\bbr^d} f(x) \mu(dx)=0$,
\begin{align*}
\int_{\bbr^d} \|f(x)\|^2 \mu(dx) \leq C_P \int_{\bbr^d} \|\nabla(f)(x)\|_{HS}^2 \mu(dx),
\end{align*}
for some $C_P>0$ independent of $f$. Moreover, assume that the bilinear symmetric non-negative definite form $\mathcal{E}$ defined, for all $f,g \in \mathcal{C}_c^{\infty}(\bbr^d, \bbr^d)$, by
\begin{align*}
\mathcal{E}(f,g)=\int_{\bbr^d} \langle \nabla(f)(x), \nabla(g)(x) \rangle_{HS} \mu(dx),
\end{align*}
is closable (sufficient conditions for the closability of the above form have been addressed in \cite[Chapter $3.1$]{FOT10} and in \cite[Chapter $2.6$]{Bog10}).~Note that the function $h$ defined by, $h(x)=x$, $x\in \bbr^d$, belongs to $H$, the space of square integrable functions on $\bbr^d$ with respect to $\mu$, and that $\int_{\bbr^d} h(x) \mu(dx)=0$. Then, by Theorem \ref{thm:abstract}, for all $f\in \mathcal{D}(\mathcal{E})$,
\begin{align*}
\mathcal{E}(G_{0^+}(h),f)=\int_{\bbr^d} \langle x; f(x) \rangle \mu(dx),
\end{align*}
so that a Gaussian Stein kernel of $\mu$ exists (in the sense of \cite[Definition $2.1$]{CFP19}) and is given by
\begin{align*}
\tau_\mu = \nabla \left(\int_0^{+\infty} P_t(h)dt \right).
\end{align*}
Moreover, with $X\sim \mu$, \eqref{ineq:energy} reads
\begin{align*}
\bbe \|\tau_\mu(X)\|^2_{HS} \leq C_P \bbe\|X\|^2.
\end{align*}
Thus, one retrieves the results of \cite{CFP19}.\\
(iii) Let $\alpha \in (1,2)$ and let $\mu_\alpha$ be a rotationally invariant $\alpha$-stable probability measure on $\bbr^d$ with L\'evy measure defined by
\begin{align*}
\nu_\alpha(du)= \dfrac{c_{\alpha,d}}{\|u\|^{d+\alpha}}du,
\end{align*}
where $c_{\alpha,d}$ is given by \eqref{eq:cstenorm}. Let $H$ be the space of square-integrable functions on $\bbr^d$ with respect to $\mu_\alpha$. Let $\mathcal{E}$ be the symmetric non-negative definite bilinear form defined, for all $f,g \in \mathcal{C}_c^{\infty}(\bbr^d)$, by
\begin{align}\label{eq:canonical}
\mathcal{E}(f,g):=\int_{\bbr^d} \int_{\bbr^d} (f(x+u)-f(x))(g(x+u)-g(x))\nu_\alpha(du)\mu_\alpha(dx).
\end{align}
Since $\nu_\alpha \ast \mu_\alpha$ is absolutely continuous with respect to $\mu_\alpha$, it is standard to check that the above form is closable and its smallest closed extension gives rise (see \cite[Theorem $1.3.1$]{FOT10}) to a non-positive definite self-adjoint operator $\mathcal{A}$ on $H$ with corresponding symmetric contractive semigroup $(P_t)_{t> 0}$ on $H$. Moreover, from Theorem \ref{prop:poincSD}, for all smooth $f: \bbr^d \longrightarrow \bbr$ with $\int_{\bbr^d} f(x) \mu_\alpha(dx)=0$,
\begin{align*}
\int_{\bbr^d} |f(x)|^2 \mu_\alpha(dx) \leq \int_{\bbr^d}\int_{\bbr^d} |f(x+u)-f(x)|^2  \nu_{\alpha}(du) \mu_\alpha(dx).
\end{align*}
Now, \cite[Corollary 1.3.1]{FOT10} together with the above inequality yields
\begin{align*}
\|P_t(f)\|_{H}\leq \exp\left(-t\right)\|f\|_H,\quad t>0,\, f\in H_0,
\end{align*}
where $H_0$ is the space of square-integrable functions on $\bbr^d$ with respect to $\mu_\alpha$ having mean zero. Then, by Theorem \ref{thm:abstract}, for all $f\in \mathcal{D}\left(\mathcal{E}\right)$ and all $h\in H_0$
\begin{align*}
\mathcal{E} \left(G_{0^+}(h),f\right)=\int_{\bbr^d} h(x)f(x)\mu_\alpha(dx).
\end{align*}
Next, observe that the function $h(x)=x$, for $x\in \bbr^d$, does not belong to $L^2(\mu_\alpha)$.
\end{rem}
\noindent
The next technical lemma describes the link between the semigroup of operators obtained from the form $\mathcal{E}$ given by \eqref{eq:canonical} and the semigroup of operators $(P^\nu_t)_{t\geq 0}$ given by \eqref{eq:semSD} with $\nu=\nu_\alpha$ as in \eqref{eq:nustable} and with $\alpha \in (1,2)$. With the help of this lemma, it is then possible to obtain the spectral properties of this semigroup of symmetric operators based on those of $(P^\nu_t)_{t\geq 0}$.

\begin{lem}\label{lem:link}
Let $\alpha \in (1,2)$, let $\nu_\alpha$ be the L\'evy measure given by \eqref{eq:nustable} and let $\mu_\alpha$ be the corresponding rotationally invariant $\alpha$-stable probability measure on $\bbr^d$. Let $\mathcal{E}$ be the smallest closed extension of the symmetric non-negative definite bilinear form given by $\eqref{eq:canonical}$. Let $(P_t)_{t> 0}$ be the strongly continuous semigroup of symmetric contractions on $L^2(\mu_\alpha)$ associated with $\mathcal{E}$. Let $(P^{\nu_\alpha}_t)_{t\geq 0}$ be the semigroup of operators defined by \eqref{eq:semSD} and let $((P^{\nu_\alpha}_t)^*)_{t\geq 0}$ be its dual semigroup in $L^2(\mu_\alpha)$. Then, for all $f\in L^2(\mu_\alpha)$ and all $t>0$,
\begin{align}\label{eq:squareroot}
P_t(f)=\underset{n\rightarrow+\infty}{\lim}\left(P^{\nu_\alpha}_{\frac{t}{n\alpha}} \circ (P^{\nu_\alpha}_{\frac{t}{n\alpha}})^*\right)^n(f)= \underset{n\rightarrow+\infty}{\lim}\left((P^{\nu_\alpha}_{\frac{t}{n\alpha}})^* \circ P^{\nu_\alpha}_{\frac{t}{n\alpha}}\right)^n(f).
\end{align}
Moreover, for all $x\in \bbr^d$ and all $t>0$,
\begin{align*}
P_t(g)(x)= e^{-t} g(x),
\end{align*}
where $g(x)=x$, $x\in \bbr^d$.
\end{lem}

\begin{proof}
Since the form $\mathcal{E}$ is the smallest closed extension of the bilinear symmetric non-negative definite form, given by \eqref{eq:canonical}, on $L^2(\mu_\alpha)$, \cite[Theorem $1.3.1$]{FOT10} ensures the existence of a unique non-positive definite self-adjoint operator $\mathcal{A}$, which moreover, from \cite[Corollary $1.3.1$]{FOT10} is characterized by $\mathcal{D}(\mathcal{A}) \subset  \mathcal{D}(\mathcal{E})$ and by
\begin{align*}
\mathcal{E}(f,g)=\int_{\bbr^d} f(x) (-\mathcal{A})(g)(x) \mu_\alpha(dx),\quad g\in \mathcal{D}(\mathcal{A}), f\in \mathcal{D}(\mathcal{E}),
\end{align*}
where $\mathcal{D}(\mathcal{A})$ is the domain of the operator $\mathcal{A}$. Let us denote by $(P_t)_{t>0}$ the corresponding strongly continuous semigroup on $L^2(\mu_\alpha)$ whose existence and uniqueness is ensured by \cite[Lemma $1.3.2$]{FOT10}. Now, recall that the semigroup of operators $(P^{\nu_\alpha}_t)_{t\geq 0}$ extends to every $L^p\left(\mu_\alpha\right)$, $p \geq 1$, as seen using the representation \eqref{eq:semSD} and the bound,
\begin{align*}
\int_{\bbr^d} |P^{\nu_\alpha}_t(f)(x)|^p \mu_\alpha(dx) \leq \int_{\bbr^d} |f(x)|^p \mu_\alpha (dx),\quad f\in\mathcal{S}(\bbr^d), p\geq 1.
\end{align*}
Moreover, it is a $C_0$-semigroup on $L^p\left(\mu_\alpha\right)$ and its $L^p\left(\mu_\alpha\right)$-generator $\mathcal{A}_{\alpha,p}$ coincides with $\mathcal{A}_{\alpha}$ on $\mathcal{S}(\bbr^d)$ which is now defined, for all $f\in\mathcal{S}(\bbr^d)$ and all $x\in \bbr^d$, by
\begin{align*}
\mathcal{A}_{\alpha}(f)(x)= -\langle x; \nabla(f)(x)\rangle + \int_{\bbr^d} \langle \nabla(f)(x+u)-\nabla(f)(x) ; u\rangle\nu_\alpha(du),
\end{align*}
and for which the following integration by parts formula holds,
\begin{align*}
\mathcal{E}(f,f)= \frac{2}{\alpha} \int_{\bbr^d} f(x)\left(-\mathcal{A}_\alpha \right)(f)(x) \mu_\alpha(dx), \quad f\in\mathcal{S}(\bbr^d).
\end{align*}
Then, by polarization, for all $f,g\in\mathcal{S}(\bbr^d)$,
\begin{align*}
\mathcal{E}(f,g)&=\frac{1}{2}\left(\mathcal{E}(f+g,f+g)-\mathcal{E}(f,f)-\mathcal{E}(g,g)\right)\\
&=\frac{1}{\alpha} \left(\int_{\bbr^d} f(x) \left(-\mathcal{A}_\alpha \right)(g)(x) \mu_\alpha(dx)+ \int_{\bbr^d} g(x) \left(-\mathcal{A}_\alpha \right)(f)(x) \mu_\alpha(dx)\right).
\end{align*}
Moreover, since $\mathcal{S}(\bbr^d)$ is dense in $L^2(\mu_\alpha)$, the adjoint of $\mathcal{A}_{\alpha,2}$ is uniquely defined so that, for all $f \in \mathcal{S}(\bbr^d)$ and all $g \in \mathcal{S}(\bbr^d) \cap \mathcal{D}\left(\mathcal{A}_{\alpha,2}^*\right)$,
\begin{align*}
\mathcal{E}(f,g)=\frac{1}{\alpha} \left(\int_{\bbr^d} f(x) \left(-\mathcal{A}_\alpha \right)(g)(x) \mu_\alpha(dx)+ \int_{\bbr^d} f(x) \left(-\mathcal{A}_{\alpha,2}^* \right)(g)(x) \mu_\alpha(dx)\right),
\end{align*}
where $\mathcal{D}(\mathcal{A}_{\alpha,2}^*)$ is the domain of the operator $\mathcal{A}_{\alpha,2}^*$. Then, $\mathcal{S}(\bbr^d) \cap \mathcal{D}\left(\mathcal{A}_{\alpha,2}^*\right) \subset \mathcal{D}(\mathcal{A})$ and, for all $f\in \mathcal{S}(\bbr^d) \cap \mathcal{D}\left(\mathcal{A}_{\alpha,2}^*\right)$ and all $x\in \bbr^d$,
\begin{align*}
\mathcal{A}(f)(x)= \frac{1}{\alpha}\left(\mathcal{A}_\alpha(f)(x) + \mathcal{A}_{\alpha,2}^*(f)(x)\right).
\end{align*}
Thus, 
\begin{align*}
\mathcal{A}=\frac{1}{\alpha} \left(\mathcal{A}_{\alpha,2} + \mathcal{A}_{\alpha,2}^*\right),
\end{align*}
which implies (thanks to \cite[Theorem $X.51$]{RS2}), for all $t>0$ and all $f\in L^2(\mu_\alpha)$,
\begin{align*}
P_t(f)=\underset{n\rightarrow+\infty}{\lim} \left(P^{\nu_\alpha}_{\frac{t}{n\alpha}} \circ (P^{\nu_\alpha}_{\frac{t}{n\alpha}})^*\right)^n(f)=\underset{n\rightarrow+\infty}{\lim} \left((P^{\nu_\alpha}_{\frac{t}{n\alpha}})^* \circ P^{\nu_\alpha}_{\frac{t}{n\alpha}}\right)^n(f),
\end{align*}
where $(P^{\nu_\alpha}_{\frac{t}{n\alpha}})_{t\geq 0}$ is the extension to $L^2(\mu_\alpha)$ of the semigroup of operators given by \eqref{eq:semSD}, after the time change $t\rightarrow t/(n\alpha)$, while $((P^{\nu_\alpha}_{\frac{t}{n\alpha}})^*)_{t\geq 0}$ is its dual semigroup in $L^2(\mu_\alpha)$ (see, e.g., \cite[Chapter $1.10$]{Pa83}).
Next, by Fourier duality and since $\alpha \in (1,2)$, for all $f\in \mathcal{S}(\bbr^d)$ and all $j\in \{1,\dots, d\}$,
\begin{align*}
\int_{\bbr^d} x_j P_{\frac{t}{n\alpha}}^{\nu_\alpha}(f)(x) \mu_\alpha(dx)&=\int_{\bbr^d} x_j\left( \int_{\bbr^d } \mathcal{F}(f)(\xi) \varphi_{\frac{t}{n\alpha}}(\xi) e^{i \langle x; \xi e^{-\frac{t}{n\alpha}} \rangle} \dfrac{d\xi}{(2\pi)^d}\right) \mu_\alpha(dx)\\
&= \int_{\bbr^d} \mathcal{F}(f)(\xi) \varphi_{\frac{t}{n\alpha}}(\xi) \left( \int_{\bbr^d} x_j e^{i \langle x; \xi e^{-\frac{t}{n\alpha}}\rangle}\mu_\alpha(dx)\right) \frac{d\xi}{(2\pi)^d}\\
&= \int_{\bbr^d} \mathcal{F}(f)(\xi) \varphi_{\frac{t}{n\alpha}}(\xi) \frac{e^{\frac{t}{n\alpha}}}{i} \dfrac{d}{d\xi_j} \left(\varphi(e^{-\frac{t}{n\alpha}}\xi)\right) \frac{d\xi}{(2\pi)^d}\\
&=  \int_{\bbr^d} \mathcal{F}(f)(\xi) \varphi_{\frac{t}{n\alpha}}(\xi) \frac{e^{\frac{t}{n\alpha}}}{i} \dfrac{d}{d\xi_j} \left(\exp\left(-e^{-t/n}\frac{\|\xi\|^\alpha}{2}\right)\right) \frac{d\xi}{(2\pi)^d}\\
&= \int_{\bbr^d} \mathcal{F}(f)(\xi) \varphi_{\frac{t}{n\alpha}}(\xi) \frac{e^{\frac{t}{n\alpha}}}{i} \left(-\frac{\alpha}{2}e^{-t/n} \| \xi \|^{\alpha-2} \xi_j \right) \exp\left(-e^{-t/n}\frac{\|\xi\|^\alpha}{2}\right) \frac{d\xi}{(2\pi)^d}\\
&= \int_{\bbr^d} \mathcal{F}(f)(\xi)  \frac{e^{\frac{t}{n\alpha}}}{i} \left(-\frac{\alpha}{2}e^{-t/n} \| \xi \|^{\alpha-2} \xi_j \right) \exp\left(-\frac{\|\xi\|^\alpha}{2}\right) \frac{d\xi}{(2\pi)^d}\\
&=\int_{\bbr^d} \mathcal{F}(f)(\xi)  \frac{e^{\frac{t}{n\alpha}-\frac{t}{n}}}{i} \dfrac{d}{d\xi_j}\left(\exp\left(-\frac{\|\xi\|^\alpha}{2}\right)\right) \frac{d\xi}{(2\pi)^d}\\
&=\int_{\bbr^d} \mathcal{F}(f)(\xi) e^{\frac{t}{n\alpha}-\frac{t}{n}} \left(\int_{\bbr^d} x_j e^{i \langle x ; \xi\rangle}\mu_\alpha(dx)\right) \frac{d\xi}{(2\pi)^d}\\
&= e^{ -\frac{\alpha-1}{n\alpha}t} \int_{\bbr^d} x_j f(x) \mu_\alpha(dx).
\end{align*}
This implies that, for all $j\in \{1,\dots, d\}$ and for all $t\geq 0$,
\begin{align*}
(P^{\nu_\alpha}_{\frac{t}{n\alpha}})^*(g_j)(x) = e^{ -\frac{\alpha-1}{n\alpha}t} g_j(x),
\end{align*}
where $g_j(x)=x_j$, for $x\in \bbr^d$. This last observation concludes the proof of the lemma.
\end{proof}
\noindent
The following long remark summarizes some basic properties of the semigroups.

\begin{rem}\label{rem:steinkernstable}
(i) First, the measure $\mu_\alpha$ is invariant for $(P_t)_{t>0}$, namely, for all $f\in L^2(\mu_\alpha)$ and all $t>0$, 
\begin{align*}
\int_{\bbr^d} P_t(f)(x) \mu_\alpha(dx) = \int_{\bbr^d} f(x) \mu_\alpha(dx).
\end{align*}
This is a direct consequence of the decomposition \eqref{eq:squareroot} and of the facts that $P^{\nu_\alpha}_{t}(1)=1$ and that $\mu_\alpha$ is invariant for $(P^{\nu_\alpha}_{t})_{t\geq 0}$. Moreover, by a duality argument, $P_t(1)=1$. Finally, $(P_t)_{t>0}$ is positivity preserving since the normal contractions operate on the smallest closed extension of the form $\mathcal{E}$ given by \eqref{eq:canonical} (see, e.g., \cite[Theorem $1.4.1$]{FOT10}).\\
(ii) As mentioned in the above proof, the semigroup $(P^{\nu_\alpha}_t)_{t\geq 0}$ extends to a $C_0$-semigroup on $L^p\left(\mu_\alpha\right)$, for $p \geq 1$. From \cite[Corollary $10.6$]{Pa83}, it follows that the dual semigroup $((P^{\nu_\alpha}_t)^*)_{t\geq 0}$ is then a $C_0$-semigroup on $L^{p^*}\left(\mu_\alpha\right)$, for $p^*=p/(p-1)$ and $1<p<+\infty$. By duality, since $P^{\nu_\alpha}_t$ is a contraction on $L^p\left(\mu_\alpha\right)$, $(P^{\nu_\alpha}_t)^*$ is a contraction on $L^{p^*}\left(\mu_\alpha\right)$. Moreover, the different extensions of $(P^{\nu_\alpha}_t)_{t\geq 0}$ (as well as those of $((P^{\nu_\alpha}_t)^*)_{t\geq 0}$) are compatible in the sense that, for all $t\geq 0$ and all $p\geq q > 1$,
\begin{align*}
(P^{\nu_\alpha}_t)_q \circ i_{p,q}=i_{p,q}\circ (P^{\nu_\alpha}_t)_p,
\end{align*}
where $i_{p,q}$ is the continuous embedding of $L^p(\mu_\alpha)$ in $L^q(\mu_\alpha)$ and where $((P^{\nu_\alpha}_t)_q)_{t\geq 0}$ and $((P^{\nu_\alpha}_t)_p)_{t\geq 0}$ are, respectively, the $L^q(\mu_\alpha)$ and the $L^p(\mu_\alpha)$ extensions of the semigroup $(P^{\nu_\alpha}_t)_{t\geq 0}$.~Moreover, \cite[Theorem $X.55$]{RS2} allows to extend $(P_t)_{t>0}$ as a contraction semigroup on $L^p(\mu_\alpha)$, for $p \in (1,+\infty)$. Moreover, for all $t>0$ and all $f\in L^p(\mu_\alpha)$, $p \in (1,+\infty)$,
\begin{align*}
P_t(f)=\underset{n\rightarrow+\infty}{\lim} \left((P^{\nu_\alpha}_{\frac{t}{n\alpha}})_p \circ ((P^{\nu_\alpha}_{\frac{t}{n\alpha}})^{*})_p\right)^n(f)=\underset{n\rightarrow+\infty}{\lim} \left(((P^{\nu_\alpha}_{\frac{t}{n\alpha}})^{*})_p \circ (P^{\nu_\alpha}_{\frac{t}{n\alpha}})_p\right)^n(f).
\end{align*}
Then, one can consider $P_t(g_j)$ since $g_j \in L^p(\mu_\alpha)$, $p \in (1, \alpha)$, but $g_j \notin L^2(\mu_\alpha)$, for $j \in \{1, \dots, d\}$. To ease the presentation, these extensions are all denoted by $(P^{\nu_\alpha}_t)_{t\geq 0}$ (and similarly for $((P^{\nu_\alpha}_t)^*)_{t\geq 0}$).\\
(iii) Recall that, by Remark \ref{rem:SteinKernel} (iii), for all $t>0$,
\begin{align*}
\underset{\underset{\int_{\bbr^d}f(x) \mu_\alpha(dx)=0}{f \in L^{2}(\mu_\alpha),\, \|f\|_{L^2(\mu_\alpha)}=1}}{\sup} \|P_t(f)\|_{L^2(\mu_\alpha)} \leq e^{-t}.
\end{align*}
Moreover, by (ii) above, for all $p\in (1,+\infty)$ and all $t>0$,
\begin{align*}
\|P_t\|_{L^p(\mu_\alpha)\rightarrow L^p(\mu_\alpha)} \leq 1.
\end{align*}
Now, let $T$ be the operator from $L^p(\mu_\alpha)$ to $L^p(\mu_\alpha)$, $p\geq 1$, such that, for all $f\in L^p(\mu_\alpha)$,
\begin{align*}
T(f)=f-\int_{\bbr^d} f(x) \mu_\alpha(dx).
\end{align*}
Clearly, $\|T\|_{L^p(\mu_\alpha) \rightarrow L^p(\mu_\alpha)}\leq 2$, for $p\geq 1$, and $T(f)=f$ for $f\in L^p(\mu_\alpha)$ with mean $0$. Let $\beta\in (1, \alpha)$. Then, by \cite[Theorem $1.3.4$]{G08}, for any $\theta \in (0,1)$ and all $t>0$,
\begin{align}\label{ineq:inter}
\underset{f \in L^{p_\theta}(\mu_\alpha),\, \|f\|_{L^{p_\theta}(\mu_\alpha)}=1}{\sup} \|P_t(T(f))\|_{L^{p_\theta}(\mu_\alpha)} \leq 2e^{-(1-\theta) t},
\end{align}
where $p_\theta$ belongs to $(\beta, 2)$ with $1/p_\theta=(1-\theta)/2+\theta/\beta$. Now, choosing $\theta \in (0,1)$ in such a way that $p_\theta \in (\beta, \alpha)$, it then follows that $g_j \in L^{p_\theta}(\mu_\alpha)$, for $j \in \{1, \dots, d\}$. Moreover, let $g_{R,j}$, for $R\geq 1$ and $j \in \{1,\dots, d\}$, be a smooth truncation of $g_j$, defined, for all $x\in \bbr^d$, by
\begin{align*}
g_{R,j}(x)=x_j \psi(x/R),
\end{align*} 
where $\psi(x)=\exp\left(-\|x\|^2\right)$, $x\in \bbr^d$. Note that $\int_{\bbr^d}g_{R,j}(x) \mu_\alpha(dx)=0$, for $R\geq 1$ and $j\in\{1, \dots, d\}$. Then, the following crucial estimate holds, for all $j\in \{1,\dots, d\}$,
\begin{align*}
\|G_{0^+}(g_j)-G_{0^+}(g_{R,j})\|_{L^{p_\theta}(\mu_\alpha)} \leq \int_{0}^{+\infty} \| P_t( g_j -g_{R,j})\| _{L^{p_\theta}(\mu_\alpha)}dt \leq \frac{2}{1-\theta} \|g_j - g_{R,j}\|_{L^{p_\theta}(\mu_\alpha)} \underset{R\rightarrow +\infty}{\longrightarrow} 0,
\end{align*}
where the limit follows from Lebesgue's dominated convergence theorem.\\
(iv) As noticed above, for $j \in \{1,\dots, d\}$, the functions $g_j(x)=x_j$, $x\in \bbr^d$, do not belong to $L^2(\mu_\alpha)$ so that Theorem \ref{thm:abstract} does not directly apply with $u=g_j$. To circumvent this fact, one can apply a smooth truncation procedure as in (iii). Thus, by Theorem \ref{thm:abstract}, for all $R\geq 1$, all $j\in \{1,\dots, d\}$ and all $f\in \mathcal{D}(\mathcal{E})$,
\begin{align}\label{eq:covstable}
\mathcal{E} \left(G_{0^+}(g_{R,j}),f\right)=\int_{\bbr^d} g_{R,j}(x)f(x)\mu_\alpha(dx),
\end{align}
and, as $R\longrightarrow +\infty$,
\begin{align*}
\int_{\bbr^d} g_{R,j}(x)f(x)\mu_\alpha(dx) \longrightarrow \int_{\bbr^d} x_j f(x)\mu_\alpha(dx),
\end{align*}
for all $f$ bounded on $\bbr^d$. Moreover, from Lemma \ref{lem:link}, for all $j\in \{1,\dots, d\}$ and all $x\in \bbr^d$, $G_{0^+}(g_{j})(x)= x_j$. Then, since $\mu_\alpha \ast \nu_\alpha <<\mu_\alpha$, as $R\longrightarrow+\infty$, for all $f$ bounded and Lipschitz on $\bbr^d$,
\begin{align*}
\mathcal{E} \left(G_{0^+}(g_{R,j}),f\right) \longrightarrow \int_{\bbr^d}\int_{\bbr^d} u_j (f(x+u)-f(x)) \nu_\alpha(du) \mu_\alpha(dx).
\end{align*}
Putting together these last two facts into \eqref{eq:covstable} gives, for all $f$ bounded and Lipschitz on $\bbr^d$,
\begin{align*}
\int_{\bbr^d}\int_{\bbr^d} u (f(x+u)-f(x)) \nu_\alpha(du)\mu_\alpha(dx) = \int_{\bbr^d} x f(x)\mu_\alpha(dx).
\end{align*}
\end{rem}
\noindent
Next, let us state a result ensuring the existence of a Stein kernel with respect to the rotationally invariant $\alpha$-stable distributions, $\alpha \in (1,2)$, for appropriate probability measures on $\bbr^d$.~Before doing so, recall that a closed, symmetric, bilinear, non-negative definite form on $L^2(\mu)$ is said to be Markovian if \cite[$(\mathcal{E}.4)$]{FOT10} holds. Now, from \cite[Theorem $1.4.1$]{FOT10}, this is equivalent to the fact that the corresponding semigroup $P_t$ is Markovian for all $t>0$, namely, for all $0\leq f\leq 1$, $\mu$-a.e., then $0\leq P_t(f)\leq 1$, $\mu$-a.e.

\begin{thm}\label{thm:ExistStable}
Let $\alpha \in (1,2)$ and let $\nu_\alpha$ be the L\'evy measure given by \eqref{eq:nustable}. Let $\beta \in (1,\alpha)$ and let $\mu$ be a centered probability measure on $\bbr^d$ with $ \int_{\bbr^d} \|x\|^\beta \mu(dx)<+\infty$ and with $\mu\ast \nu_\alpha<< \mu$. Let $\mathcal{E}_\mu$ be the closable, Markovian, symmetric, bilinear, non-negative definite form defined, for all $f,g \in \mathcal{S}(\bbr^d)$, by
\begin{align*}
\mathcal{E}_\mu (f,g) := \int_{\bbr^d} \int_{\bbr^d} (f(x+u)-f(x))(g(x+u)-g(x))\nu_\alpha(du) \mu(dx).
\end{align*}
Let $(P_t)_{t>0}$ be the strongly continuous Markovian semigroup on $L^2(\mu)$ associated with the smallest closed extension of $\mathcal{E}_\mu$ with dense linear domain $\mathcal{D}(\mathcal{E}_\mu)$. Moreover, let there exists $U_\mu>0$ such that, for all $f\in \mathcal{D}(\mathcal{E}_\mu)$ with $\int_{\bbr^d} f(x) \mu(dx)=0$,
\begin{align}\label{ineq:poincmu}
\int_{\bbr^d} |f(x)|^2 \mu(dx) \leq U_\mu \int_{\bbr^d} \int_{\bbr^d} |f(x+u)-f(x)|^2 \nu_{\alpha}(du)\mu(dx).
\end{align}
Let $p \in (1,\beta)$ and let $\theta \in (0,1)$ be such that $p_\theta$, given by $1/p_\theta=(1-\theta)/2+\theta/p$, belongs to $(p,\beta)$. Then, there exists $\tau_\mu \in \mathcal{D}(\mathcal{A}_{\mu})$ such that, for all $f\in \mathcal{D}(\mathcal{A}^*_{\mu})\cap L^{\infty}(\mu)$,
\begin{align*}
\int_{\bbr^d} \tau_\mu(x) (-\mathcal{A}_\mu^*)(f)(x) \mu(dx) = \int_{\bbr^d} x f(x) \mu(dx),
\end{align*}
where $\mathcal{A}_\mu^*$ is the adjoint of $\mathcal{A}_\mu$, the generator of the $L^{p_\theta}(\mu)$-extension of the semigroup $(P_t)_{t>0}$, with respective domains $\mathcal{D}(\mathcal{A}^*_{\mu})$ and $\mathcal{D}(\mathcal{A}_{\mu})$.
\end{thm}

\begin{proof}
Since the form $\mathcal{E}_\mu$ is closable, let us consider its smallest closed extension with dense linear domain $D(\mathcal{E}_\mu)$. Then, let $\mathcal{A}_{\mu,2}$, $(P_t)_{t>0}$ and $(G_\delta)_{\delta>0}$ be, respectively, the corresponding generator, strongly continuous semigroup and strongly continuous resolvent on $L^2(\mu)$ such that, for all $\delta>0$ and all $f \in L^2(\mu)$,
\begin{align*}
G_\delta(f) = \int_0^{+\infty} e^{- \delta t} P_t(f)dt.
\end{align*}
Next, \cite[Corollary 1.3.1]{FOT10} together with the inequality \eqref{ineq:poincmu} yield
\begin{align*}
\|P_t(f)\|_{L^2(\mu)}\leq \exp\left(-\frac{t}{U_\mu}\right)\|f\|_{L^2(\mu)},\quad t>0,\, f\in L^2(\mu), \int_{\bbr^d} f(x) \mu(dx)=0.
\end{align*}
Then, by Theorem \ref{thm:abstract}, for all $f \in L^2(\mu)$ with $\int_{\bbr^d} f(x) \mu(dx)=0$ and all $h \in D(\mathcal{E}_\mu)$,
\begin{align*}
\mathcal{E}_\mu (G_{0^+}(f),h) = \int_{\bbr^d} h(x)f(x) \mu(dx).
\end{align*}
Now, for $x\in \bbr^d$, $j \in \{1,\dots, d\}$ and $R\geq 1$, set $g_j(x)=x_j$ and $g_{R,j}(x)=x_j \psi(x/R)$, with $\psi(x)=\exp(-\|x\|^2)$. Observe that $g_j \notin L^2(\mu)$ but that $g_{R,j} \in L^2(\mu)$. Then, for all $j\in \{1, \dots, d\}$, all $R\geq 1$ and all $h \in D(\mathcal{E}_\mu)$,
\begin{align}\label{eq:approxsteinkernel}
\mathcal{E}_\mu (G_{0^+}(\tilde{g}_{R,j}),h) = \int_{\bbr^d} \tilde{g}_{R,j}(x)h(x) \mu(dx),
\end{align}
where $\tilde{g}_{R,j}(x)=g_{R,j}(x)-\int_{\bbr^d}g_{R,j}(x)\mu(dx)$, $x\in \bbr^d$. A straightforward application of Lebesgue's dominated convergence theorem implies that, as $R\rightarrow +\infty$,
\begin{align*}
\int_{\bbr^d} \tilde{g}_{R,j}(x)h(x) \mu(dx) \longrightarrow \int_{\bbr^d} x_j h(x) \mu(dx),
\end{align*}
for $h$ bounded on $\bbr^d$. Let us study the semigroup $(P_t)_{t>0}$ associated with the form $\mathcal{E}_{\mu}$. First, since it is a strongly continuous semigroup on $L^2(\mu)$, it follows from \cite[Theorem 2.4]{Pa83} that, for all $f\in L^2(\mu)$ and all $t>0$, $\int_0^t P_s(f)ds$ belongs to the domain of $\mathcal{A}_{\mu,2}$ and that
\begin{align*}
\mathcal{A}_{\mu,2} \left(\int_0^t P_s(f)ds\right) =P_t(f)-f.
\end{align*}
Then, for all $f \in L^2(\mu)$ and all $g\in \mathcal{D}(\mathcal{E}_\mu)$,
\begin{align*}
\mathcal{E}_\mu \left(\int_0^t P_s(f)ds, g\right)=-\int_{\bbr^d}(P_t(f)(x)-f(x))g(x)\mu(dx).
\end{align*}
Now, choosing $g=1$, which clearly belongs to $\mathcal{D}(\mathcal{E}_\mu)$, implies
\begin{align*}
\int_{\bbr^d} P_t(f)(x) \mu(dx)= \int_{\bbr^d} f(x) \mu(dx),
\end{align*}
namely, the probability measure $\mu$ is invariant for the semigroup $(P_t)_{t>0}$.Then, for all $g\in L^2(\mu)$,
$\langle P_t(1);g\rangle_{L^2(\mu)}=\langle 1;P_t(g)\rangle_{L^2(\mu)}=\langle 1;g\rangle_{L^2(\mu)}$, so that $P_t(1)=1$. Moreover, since, for $t>0$, $P_t$ is also positivity preserving, it can be extended to a contraction on $L^p(\mu)$, for all $1\leq p <+\infty$. Finally, $(P_t)_{t>0}$ extends to a $C_0$-semigroup on $L^p(\mu)$, $1 \leq p <+\infty$. Then, there exists $p \in (1, \beta)$ such that, for all $t>0$, $\|P_t\|_{L^p(\mu)\rightarrow L^p(\mu)}\leq 1$. 
By an interpolation argument as in Remark \ref{rem:steinkernstable} (iii), for any $\theta \in (0,1)$ and all $t>0$,
\begin{align}\label{ineq:inter}
\underset{f \in L^{p_\theta}(\mu),\, \|f\|_{L^{p_\theta}(\mu)}=1}{\sup} \|P_t(T(f))\|_{L^{p_\theta}(\mu)} \leq 2e^{- \frac{(1-\theta) t}{U_\mu}},
\end{align}
where $p_\theta \in (p, 2)$ is such that $1/p_\theta=(1-\theta)/2+\theta/p$. Now, choose $\theta \in (0,1)$ such that $p_\theta \in (p, \beta)$. Then, for all $R \geq 1$ and all $j \in\{1, \dots, d\}$,
\begin{align*}
\|G_{0^+}(g_j)-G_{0^+}(\tilde{g}_{R,j})\|_{L^{p_\theta}( \mu)} \leq \int_{0}^{+\infty} \| P_t(g_j-\tilde{g}_{R,j})\|_{L^{p_\theta}( \mu)} dt \leq \dfrac{2 U_\mu}{1-\theta} \|g_j-\tilde{g}_{R,j}\|_{L^{p_\theta}( \mu)}.
\end{align*}
Thus, by the Lebesgue dominated convergence theorem, $G_{0^+}(\tilde{g}_{R,j})$ converges strongly in $L^{p_\theta}( \mu)$ to $G_{0^+}(g_{j})$, as $R \rightarrow +\infty$. For $1<p<+\infty$, let us denote by $\mathcal{A}_{\mu,p}$ the generator of the $C_0$-semigroup $(P_t)_{t>0}$ with domain $\mathcal{D}(\mathcal{A}_{\mu,p}) \subset L^p(\mu)$. Recall that, by \cite[Corollary 10.6]{Pa83}, the dual semigroup $(P_t)^*_{t>0}$ is a $C_0$-semigroup on $L^{p^*}(\mu)$ with generator $\mathcal{A}_{\mu,p}^*$, the adjoint of $\mathcal{A}_{\mu,p}$. Moreover, observe that 
\begin{align*}
\mathcal{A}_{\mu,2}^*|_{\mathcal{D}(\mathcal{A}_{\mu,p_\theta}^*)}=\mathcal{A}_{\mu,p_{\theta}}^*,
\end{align*}
since $p_\theta \in (p,\beta)$ with $1<p<\beta<\alpha<2$.~Then, \cite[Corollary $1.3.1$]{FOT10} implies that, for all $h \in \mathcal{D}(\mathcal{A}_{\mu,2}) \subset \mathcal{D}(\mathcal{E}_\mu)$,
\begin{align*}
\mathcal{E}_\mu (G_{0^+}(\tilde{g}_{R,j}),h) &= \int_{\bbr^d} G_{0^+}(\tilde{g}_{R,j})(x) (-\mathcal{A}_{\mu,2})(h)(x) \mu(dx)\\
&= \int_{\bbr^d} G_{0^+}(\tilde{g}_{R,j})(x) (-\mathcal{A}^*_{\mu,2})(h)(x) \mu(dx).
\end{align*} 
Now, taking $h \in \mathcal{D}(\mathcal{A}_{\mu,p_\theta}^*)\subset \mathcal{D}(\mathcal{A}_{\mu,2}) \subset \mathcal{D}(\mathcal{E}_\mu)$,
\begin{align*}
\mathcal{E}_\mu (G_{0^+}(\tilde{g}_{R,j}),h) &=  \int_{\bbr^d} G_{0^+}(\tilde{g}_{R,j})(x) (-\mathcal{A}_{\mu, p_\theta}^*)(h)(x) \mu(dx).
\end{align*}
Let $T>0$. Recall that $\int_0^T P_t(\tilde{g}_{R,j})dt$ belongs to $\mathcal{D}(\mathcal{A}_{\mu,p_\theta})$, for $R\geq 1$ and for $j \in \{1, \dots, d \}$. Moreover, $\int_0^T P_t(\tilde{g}_{R,j})dt$ converges strongly in $L^{p_\theta}( \mu)$ to $G_{0^+}(\tilde{g}_{R,j})$, as $T\rightarrow +\infty$. Finally, for all $R\geq 1$ and all $j \in \{1, \dots, d \}$,
\begin{align*}
\mathcal{A}_{\mu,p_\theta}\left(\int_0^T P_t(\tilde{g}_{R,j})dt\right) = P_T(\tilde{g}_{R,j})-\tilde{g}_{R,j} \underset{T\rightarrow +\infty}{\longrightarrow}-\tilde{g}_{R,j},
\end{align*}
in $L^{p_\theta}(\mu)$. Since $\mathcal{A}_{\mu,p_\theta}$ is closed, $G_{0^+}(\tilde{g}_{R,j}) \in \mathcal{D}(\mathcal{A}_{\mu,p_\theta})$ and $(-\mathcal{A}_{\mu,p_\theta})(G_{0^+}(\tilde{g}_{R,j}))=\tilde{g}_{R,j}$. Finally, a similar argument ensures that $G_{0^+}(g_{j})$ belongs to $\mathcal{D}(\mathcal{A}_{\mu,p_\theta})$ and that $(-\mathcal{A}_{\mu,p_\theta})(G_{0^+}(g_{j}))=g_{j}$. Setting $G_{0^+}(g)=\tau_\mu$ with $g(x)=x$, $x\in \bbr^d$, concludes the proof of the theorem.
\end{proof}
\noindent
To finish this section, a stability result for probability measures on $\bbr^d$ close to the rotationally invariant $\alpha$-stable ones, $\alpha \in (1,2)$, is presented.

\begin{thm}\label{thm:stability}
Let $\psi(x)=\exp(-\|x\|^2)$, $x\in \bbr^d$. Let $\alpha \in (1,2)$, let $\nu_\alpha$ be the L\'evy measure given by \eqref{eq:nustable} and let $\mu_\alpha$ be the associated rotationally invariant $\alpha$-stable distribution. Let $\beta \in (1,\alpha)$ and let $\mu$ be a centered probability measure on $\bbr^d$ with $\int_{\bbr^d} \|x\|^\beta \mu(dx)<+\infty$ and with $\mu\ast \nu_\alpha<< \mu$. Let $\mathcal{E}_\mu$ be the closable, Markovian, symmetric, bilinear, non-negative definite form defined, for all $f,g \in \mathcal{S}(\bbr^d)$, by
\begin{align*}
\mathcal{E}_\mu (f,g) := \int_{\bbr^d} \int_{\bbr^d} (f(x+u)-f(x))(g(x+u)-g(x))\nu_\alpha(du) \mu(dx).
\end{align*}
Moreover, assume that,
\begin{itemize}
\item there exists $U_\mu>0$ such that, for all $f\in \mathcal{D}(\mathcal{E}_\mu)$ with $\int_{\bbr^d} f(x) \mu(dx)=0$,
\begin{align*}
\int_{\bbr^d} |f(x)|^2 \mu(dx) \leq U_\mu \int_{\bbr^d} \int_{\bbr^d} |f(x+u)-f(x)|^2 \nu_{\alpha}(du)\mu(dx),
\end{align*}
\item and, for $\tilde{g}_{R,j}(x)=x_j \psi(x/R)-\int_{\bbr^d}x_j \psi(x/R)\mu(dx)$, $x \in \bbr^d$, $R\geq 1$ and $j\in \{1,\dots, d\}$, 
\begin{align}\label{Weyltype}
\underset{R \rightarrow +\infty}{\lim} \left(\mathcal{E}_{\mu}(U_\mu \tilde{g}_{R,j},\tilde{g}_{R,j})-\langle \tilde{g}_{R,j} ; \tilde{g}_{R,j} \rangle_{L^2(\mu)}\right) = 0.
\end{align}
\end{itemize}
Then, 
\begin{align*}
d_{W_1}(\mu, \mu_\alpha) \leq C_{\alpha,d} \left(\int_{\|u\|\leq 1} \|u\|^2 \nu_\alpha(du)+\int_{\|u\|\geq 1} \|u\| \nu_\alpha(du)\right) |U_\mu-1|,
\end{align*}
for some $C_{\alpha,d}>0$ only depending on $\alpha$ and on $d$.
\end{thm}

\begin{proof}
The proof partly relies on the methodological results contained in \cite{AH18_2}. First, as in \cite[Proposition $3.4$]{AH18_2}, for all $h \in \mathcal{H}_1 \cap \mathcal{C}_c^\infty(\bbr^d)$, let $f_h$, be defined, for all $x\in \bbr^d$, by
\begin{align*}
f_h(x)=-\int_0^{+\infty} (P^{\nu_\alpha}_t(h)(x)- \bbe h(X_\alpha))dt,
\end{align*}
with $(P^{\nu_\alpha}_t)_{t\geq 0}$ given in \eqref{eq:semSD} with $\nu=\nu_\alpha$ and $X_\alpha \sim \mu_\alpha$. Next, let $X\sim \mu$. Then, for all $h\in \mathcal{H}_1 \cap \mathcal{C}_c^\infty(\bbr^d)$,
\begin{align*}
\left |\bbe h(X)-\bbe h(X_\alpha)\right | &= \left| \bbe  \left(-\langle X; \nabla(f_h)(X)\rangle + \int_{\bbr^d} \langle \nabla (f_h)(X+u)- \nabla (f_h)(X); u\rangle \nu_{\alpha}(du)\right) \right|\\
&\leq \sum_{j=1}^d \left|- \langle g_j , \partial_j(f_h) \rangle_{L^2(\mu)}+ \mathcal{E}_\mu(g_j, \partial_j(f_h))\right|,
\end{align*}
where $g_j(x)=x_j$, $x\in \bbr^d$ and $j \in \{1,\dots, d\}$. Let $g_{R,j}$ be the smooth truncation of $g_j$ as defined by \eqref{eq:gRj} with $\psi(x)=\exp(-\|x\|^2)$, $x\in \bbr^d$. Moreover, $\|g_j-g_{R,j}\|_{L^p(\mu)}\rightarrow 0$, as $R$ tends to $+\infty$, for all $p \leq \beta$. Since, (see \cite[Proposition $3.4$]{AH18_2}) $M_1(f_h)\leq 1$, 
\begin{align*}
\left |\bbe h(X)-\bbe h(X_\alpha)\right | & \leq \sum_{j=1}^d \|g_j-g_{R,j}\|_{L^1(\mu)} + \sum_{j=1}^d \left|- \langle g_{R,j} , \partial_j(f_h) \rangle_{L^2(\mu)}+ \mathcal{E}_\mu(g_j, \partial_j(f_h))\right|\\
&\leq \sum_{j=1}^d \|g_j-g_{R,j}\|_{L^1(\mu)} + \sum_{j=1}^d \left|- \langle g_{R,j} , \partial_j(f_h) \rangle_{L^2(\mu)}+ \mathcal{E}_\mu(g_{R,j}, \partial_j(f_h))\right| \\ 
&\quad\quad+ \sum_{j=1}^d \left| \mathcal{E}_\mu(g_j-g_{R,j}, \partial_j(f_h))\right|.
\end{align*}
Now, for all $j \in \{1,\dots, d \}$ and all $R\geq 1$, 
\begin{align*}
 \left| \mathcal{E}_\mu(g_j-g_{R,j}, \partial_j(f_h))\right| &\leq  \int_{\bbr^d}\int_{\bbr^d} |g_j(x+u)-g_{R,j}(x+u)-g_j(x)+g_{R,j}(x)| \\
 &\quad\quad \times\left|\partial_{j}(f_h)(x+u)-\partial_j(f_h)(x)\right| \nu_{\alpha}(du) \mu(dx).
\end{align*} 
Cutting the integral on $u$ into a small jumps part and a big jumps part and using $M_1(f_h)\leq 1$ and $M_2(f_h)\leq C_{\alpha,d}$, for some $C_{\alpha,d}>0$ depending only on $\alpha$ and on $d$, imply
\begin{align*}
 \left| \mathcal{E}_\mu(g_j-g_{R,j}, \partial_j(f_h))\right| &\leq C_{\alpha, d} \int_{\bbr^d} \int_{\{\|u\|\leq 1\}} \|u\| \left|(x_j+u_j)(1-\psi((x+u)/R))-x_j(1-\psi(x/R))\right| \nu_\alpha(du) \mu(dx)\\
& \quad\quad+2  \int_{\bbr^d} \int_{\{\|u\|\geq 1\}} \left|(x_j+u_j)(1-\psi((x+u)/R))-x_j(1-\psi(x/R))\right| \nu_\alpha(du) \mu(dx).
\end{align*}
Since $\|g_j-g_{R,j}\|_{L^p(\mu)}\rightarrow 0$, as $R$ tends to $+\infty$, and since $\mu \ast \nu_\alpha<<\mu$, along a subsequence,
\begin{align*}
g_j(x+u)-g_{R,j}(x+u)-g_j(x)+g_{R,j}(x)\underset{R\rightarrow+\infty}{\longrightarrow} 0, \quad \mu \otimes \nu_\alpha-a.e.
\end{align*}
Now, for all $x\in \bbr^d$, all $u\in \bbr^d$, all $R\geq 1$ and all $j\in \{1, \dots, d\}$,
\begin{align*}
\left|(x_j+u_j)(1-\psi((x+u)/R))-x_j(1-\psi(x/R))\right| \leq C_{j,d}\|u\|, 
\end{align*}
for some constant $C_{j,d}>0$ depending only on $j$ and on $d$. Thus, Lebesgue's dominated convergence theorem implies that, for all $j \in \{1,\dots, d\}$, along a subsequence
\begin{align*}
\underset{h\in \mathcal{H}_1 \cap \mathcal{C}_c^\infty(\bbr^d)}{\sup}\left| \mathcal{E}_\mu(g_j-g_{R,j}, \partial_j(f_h))\right| \underset{R\rightarrow+\infty}{\longrightarrow} 0.
\end{align*}
Finally, for all $R\geq 1$ and all $j \in \{1,\dots, d\}$,
\begin{align}\label{ineq:last}
\left|- \langle \tilde{g}_{R,j} , \partial_j(f_h) \rangle_{L^2(\mu)}+ \mathcal{E}_\mu(\tilde{g}_{R,j}, \partial_j(f_h))\right| &= \left| \mathcal{E}_\mu(\tilde{g}_{R,j}-G_{0^+}(\tilde{g}_{R,j}), \partial_j(f_h)) \right|\nonumber\\
&\leq \left| \mathcal{E}_\mu(\tilde{g}_{R,j}-U_\mu \tilde{g}_{R,j}, \partial_j(f_h)) \right| + \left| \mathcal{E}_\mu(U_\mu\tilde{g}_{R,j}-G_{0^+}(\tilde{g}_{R,j}), \partial_j(f_h)) \right|.
\end{align}
The first term on the right hand-side of \eqref{ineq:last} is bounded, for all $R\geq 1$ and all $j \in \{1, \dots, d\}$, by
\begin{align*}
\left| \mathcal{E}_\mu(\tilde{g}_{R,j}-U_\mu \tilde{g}_{R,j}, \partial_j(f_h)) \right| &\leq |U_\mu-1| \left| \mathcal{E}_\mu(g_{R,j}, \partial_j(f_h)) \right| \\
&\leq C_{\alpha,j,d} \left(\int_{\|u\|\leq 1} \|u\|^2 \nu_\alpha(du)+\int_{\|u\|\geq 1} \|u\| \nu_\alpha(du)\right) |U_\mu-1|.
\end{align*}
To conclude the proof, let us deal with the second term on the right-hand side of \eqref{ineq:last}. Then, by the Cauchy-Schwarz inequality, for all $j \in \{1, \dots, d\}$ and all $R \geq 1$,
\begin{align*}
\left| \mathcal{E}_\mu(U_\mu\tilde{g}_{R,j}-G_{0^+}(\tilde{g}_{R,j}), \partial_j(f_h)) \right| \leq \mathcal{E}_{\mu}(\partial_j(f_h),\partial_j(f_h))^{1/2} \mathcal{E}_{\mu}(U_\mu\tilde{g}_{R,j}-G_{0^+}(\tilde{g}_{R,j}),U_\mu\tilde{g}_{R,j}-G_{0^+}(\tilde{g}_{R,j}))^{1/2}.
\end{align*}
Now, for all $j \in \{1,\dots, d\}$,
\begin{align*}
\mathcal{E}_{\mu}(U_\mu\tilde{g}_{R,j}-G_{0^+}(\tilde{g}_{R,j}),U_\mu\tilde{g}_{R,j}-G_{0^+}(\tilde{g}_{R,j}))&= U_\mu^2 \mathcal{E}_{\mu}(\tilde{g}_{R,j},\tilde{g}_{R,j}) + \mathcal{E}_{\mu}(G_{0^+}(\tilde{g}_{R,j}),G_{0^+}(\tilde{g}_{R,j}))\\
&\quad\quad -2 U_\mu \mathcal{E}_{\mu}(\tilde{g}_{R,j},G_{0^+}(\tilde{g}_{R,j}))\\
&= U_\mu \left( U_\mu\mathcal{E}_{\mu}(\tilde{g}_{R,j},\tilde{g}_{R,j})-\langle \tilde{g}_{R,j} ; \tilde{g}_{R,j}\rangle_{L^2(\mu)}\right)\\
&\quad\quad + \mathcal{E}_{\mu}(G_{0^+}(\tilde{g}_{R,j}),G_{0^+}(\tilde{g}_{R,j})) - U_\mu \langle \tilde{g}_{R,j} ; \tilde{g}_{R,j}\rangle_{L^2(\mu)}\\
&\leq U_\mu \left( U_\mu\mathcal{E}_{\mu}(\tilde{g}_{R,j},\tilde{g}_{R,j})-\langle \tilde{g}_{R,j} ; \tilde{g}_{R,j}\rangle_{L^2(\mu)}\right).
\end{align*}
The condition \eqref{Weyltype} concludes the proof of the theorem.
\end{proof}

\appendix
\section{Appendix}
\label{sec:appendix}
\noindent

\begin{lem}\label{lem:diffallpoints}
Let $\nu$ be a L\'evy measure with polar decomposition given by \eqref{eq:PolarSD2} where the function $k_x(r)$ is continuous in $r\in (0,+\infty)$, is continuous in $x\in \mathbb{S}^{d-1}$ and satisfies \eqref{eq:condlimk}. 
Then, for all $\xi \in \bbr^d$, the function $t\rightarrow \psi_t(\xi)$ where,
\begin{align*}
\psi_t(\xi)&=\exp \bigg(\int_{(0,+\infty)\times \mathbb{S}^{d-1}} \left(e^{i r\langle x;\xi \rangle}-1-i\langle rx;\xi \rangle\bbone_{r\leq 1}\right) \dfrac{k_x(r)}{r}dr \sigma(dx)\\
&\quad\quad-\int_{(0,+\infty)\times \mathbb{S}^{d-1}} \left(e^{i r\langle x;\xi \rangle}-1-i\langle rx;\xi \rangle\bbone_{r\leq e^{-t}}\right) \dfrac{k_x(e^t r)}{r}dr \sigma(dx)\bigg),
\end{align*}
is continuously differentiable on $[0,+\infty)$ and for all $\xi\in \bbr^d$ and all $t \geq 0$,
\begin{align*}
\frac{d}{dt}\left(\psi_t(\xi)\right)=\left( -i e^{-t} \langle \int_{\mathbb{S}^{d-1}}k_x(1) x\sigma(dx);\xi \rangle+ \int_{\mathbb{R}^d} \left(e^{i \langle u;\xi e^{-t}\rangle}-1-i \langle u;\xi e^{-t} \rangle\bbone_{\|u\|\leq 1}\right)\tilde{\nu}(du)\right)\psi_t(\xi),
\end{align*}
where $\tilde{\nu}$ is given by \eqref{eq:tilde}.
\end{lem}

\begin{proof}
First, for all $\xi \in \bbr^d$ and all $t\geq 0$,
\begin{align*}
\psi_t(\xi)&=\exp \bigg(\int_{(0,+\infty)\times \mathbb{S}^{d-1}} \left(e^{i r\langle x;\xi \rangle}-1-i\langle rx;\xi \rangle\bbone_{r\leq 1}\right) \dfrac{k_x(r)-k_x(e^t r)}{r}dr \sigma(dx)\\
&\quad\quad -\int_{(0,+\infty)\times \mathbb{S}^{d-1}}i \langle rx;\xi \rangle\bbone_{e^{-t}<r\leq 1}\dfrac{k_x(e^t r)}{r}dr \sigma(dx)\bigg)
\end{align*}
Now, observe that, for all $\xi \in \bbr^d$ and all $t\geq 0$
\begin{align*}
I&=\int_{(0,+\infty)\times \mathbb{S}^{d-1}}i \langle rx;\xi \rangle\bbone_{e^{-t}<r\leq 1}\dfrac{k_x(e^t r)}{r}dr \sigma(dx)\\
&=i \langle \xi ;\int_{\mathbb{S}^{d-1}}x \left(\int_{e^{-t}}^1 k_x(e^{t}r) dr\right) \sigma(dx) \rangle\\
&=i \langle \xi ;\int_{\mathbb{S}^{d-1}}x e^{-t}\left(\int_{1}^{e^t} k_x(r) dr\right) \sigma(dx) \rangle.
\end{align*}
Then, by Leibniz's integral rule, for all $\xi \in \bbr^d$ and all $t\geq 0$
\begin{align*}
\dfrac{d}{dt} \left(I\right)=-i\langle \xi; \int_{\mathbb{S}^{d-1}}x e^{-t}\left(\int_{1}^{e^t} k_x(r) dr\right) \sigma(dx) \rangle+i \langle \xi ; \int_{\mathbb{S}^{d-1}}x k_x(e^t) \sigma(dx)\rangle.
\end{align*}
Moreover, by Fubini's theorem, for all $\xi \in \bbr^d$ and all $t\geq 0$
\begin{align*}
II&=\int_{(0,+\infty)\times \mathbb{S}^{d-1}} \left(e^{i r\langle x;\xi \rangle}-1-i\langle rx;\xi \rangle\bbone_{r\leq 1}\right) \dfrac{k_x(r)-k_x(e^t r)}{r}dr \sigma(dx)\\
&=\int_{\mathbb{S}^{d-1}} \sigma(dx) \int_0^{+\infty}  \left(e^{i r\langle x;\xi \rangle}-1-i\langle rx;\xi \rangle\bbone_{r\leq 1}\right) \left(\int_{r}^{re^t} (-dk_x(s))\right) \frac{dr}{r}\\
&=\int_{\mathbb{S}^{d-1}} \sigma(dx) \int_0^{+\infty} (-dk_x(s)) \int_{s e^{-t}}^s  \left(e^{i r\langle x;\xi \rangle}-1-i\langle rx;\xi \rangle\bbone_{r\leq 1}\right)\frac{dr}{r}.
\end{align*}
Then, by Leibniz's integral rule, for all $\xi \in \bbr^d$ and all $t\geq 0$
\begin{align*}
\dfrac{d}{dt} \left(II\right)&=\int_{\mathbb{S}^{d-1}} \sigma(dx) \int_0^{+\infty} (-dk_x(s)) \left(e^{i s e^{-t}\langle x;\xi \rangle}-1-i\langle s e^{-t} x;\xi \rangle\bbone_{s\leq e^t}\right)\\
&=\int_{\mathbb{S}^{d-1}} \sigma(dx) \int_0^{+\infty} (-dk_x(s)) \left(e^{i s e^{-t}\langle x;\xi \rangle}-1-i\langle s e^{-t} x;\xi \rangle\bbone_{s\leq 1}\right)\\
&\quad\quad -i\int_{\mathbb{S}^{d-1}} \sigma(dx) \int_0^{+\infty} (-dk_x(s)) \langle s e^{-t} x;\xi \rangle\bbone_{1<s\leq e^t}.
\end{align*}
Now, by integration by parts, for all $\xi \in \bbr^d$ and all $t\geq 0$
\begin{align*}
e^{-t}\int_{\mathbb{S}^{d-1}}\langle x;\xi \rangle \int_1^{e^t} (-dk_x(s))s\sigma(dx)&=e^{-t}\int_{\mathbb{S}^{d-1}}\langle x;\xi \rangle  \left(\int_1^{e^t} k_x(s) ds \right)\sigma(dx)\\
&\quad \quad - e^{-t}\int_{\mathbb{S}^{d-1}}\langle x;\xi \rangle (e^t k_x(e^t)-k_x(1)) \sigma(dx).
\end{align*}
Thus, for all $\xi \in \bbr^d$ and all $t \geq 0$
\begin{align*}
\dfrac{d}{dt}\left(II-I\right)&=\int_{\mathbb{S}^{d-1}} \sigma(dx) \int_0^{+\infty} (-dk_x(s)) \left(e^{i s e^{-t}\langle x;\xi \rangle}-1-i\langle s e^{-t} x;\xi \rangle\bbone_{s\leq 1}\right)\\
&\quad\quad -ie^{-t}\int_{\mathbb{S}^{d-1}}\langle x;\xi \rangle  \left(\int_1^{e^t} k_x(s) ds\right)\sigma(dx)+i e^{-t}\int_{\mathbb{S}^{d-1}}\langle x;\xi \rangle (e^t k_x(e^t)-k_x(1)) \sigma(dx)\\
&\quad\quad+i\langle \xi; \int_{\mathbb{S}^{d-1}}x e^{-t}\left(\int_{1}^{e^t} k_x(r) dr\right) \sigma(dx) \rangle-i \langle \xi ; \int_{\mathbb{S}^{d-1}}x k_x(e^t) \sigma(dx)\rangle\\
&=\int_{\mathbb{S}^{d-1}} \sigma(dx) \int_0^{+\infty} (-dk_x(s)) \left(e^{i s e^{-t}\langle x;\xi \rangle}-1-i\langle s e^{-t} x;\xi \rangle\bbone_{s\leq 1}\right)\\
&\quad\quad-i e^{-t} \langle\xi; \int_{\mathbb{S}^{d-1}} xk_x(1)\sigma(dx) \rangle.
\end{align*}
Finally, straightforward computations conclude the proof of the lemma.
\end{proof}

\begin{lem}\label{lem:GenCau}
Let $X$ be a non-degenerate SD random vector in $\bbr^d$, without Gaussian component, with characteristic function $\varphi$, L\'evy measure $\nu$ having polar decomposition given by \eqref{eq:PolarSD2} where the function $k_x(r)$ is continuous in $r\in(0,+\infty)$, continuous in $x\in \mathbb{S}^{d-1}$ and satisfies \eqref{eq:condlimk}. 
Then,\\
(i) for all $\xi \in \bbr^d$ and all $x\in \bbr^d$,
\begin{align*}
\underset{t\rightarrow 0^+}{\lim}\frac{1}{t}\left(e^{i\langle x;\xi (e^{-t}-1) \rangle}\dfrac{\varphi(\xi)}{\varphi(e^{-t}\xi)}-1\right)&=i\langle b-\int_{\mathbb{S}^{d-1}} k_y(1)y\sigma(dy)-x;\xi\rangle\\
&\quad\quad +\int_{\bbr^d} \left(e^{i\langle u;\xi \rangle}-1-i\langle u;\xi \rangle\bbone_{\|u\|\leq 1}\right) \tilde{\nu}(du),
\end{align*}
where $\tilde{\nu}$ is given by \eqref{eq:tilde}.\\
(ii) For all $\xi \in \bbr^d$ and all $x\in \bbr^d$,
\begin{align*}
\underset{t\in (0,1)}{\sup}\frac{1}{t}\left|e^{i\langle x;\xi (e^{-t}-1) \rangle}\dfrac{\varphi(\xi)}{\varphi(e^{-t}\xi)}-1\right|&\leq \left(\|x\|+\|b\|\right)\|\xi\|+\|\xi\|^2 \int_{\mathbb{S}^{d-1}} \sigma(dx) \int_0^1s^2 (-dk_x(s)) +4  \int_{\mathbb{S}^{d-1}}k_x(1)\sigma(dx)\\
&\quad\quad+2 \|\xi\|^2  \int_{\mathbb{S}^{d-1}}k_x(1)\sigma(dx)+2\int_{\mathbb{S}^{d-1}} \sigma(dx) \int_{1}^{+\infty} (-dk_x(s))\\
&\quad\quad  +\|\xi\| \int_{\mathbb{S}^{d-1}} k_x(1)\sigma(dx).
\end{align*}
\end{lem}

\begin{proof}
Let us start with $(i)$. First, for all $\xi\in \bbr^d$ and all $t\in (0,1)$,
\begin{align*}
\dfrac{\varphi(\xi)}{\varphi(e^{-t}\xi)}&=\exp\bigg(i\langle b;\xi \rangle(1-e^{-t})+\int_{(0,+\infty)\times \mathbb{S}^{d-1}} \left(e^{i r\langle y;\xi \rangle}-1-i\langle ry;\xi \rangle\bbone_{r\leq 1}\right) \dfrac{k_y(r)-k_y(e^t r)}{r}dr \sigma(dy)\\
&\quad\quad -\int_{(0,+\infty)\times \mathbb{S}^{d-1}}i \langle ry;\xi \rangle\bbone_{e^{-t}<r\leq 1}\dfrac{k_y(e^t r)}{r}dr \sigma(dy)\bigg)
\end{align*}
Next, for all $\xi\in \bbr^d$ and all $x\in\bbr^d$,
\begin{align}\label{eq:drift}
\underset{t\rightarrow 0^+}{\lim} \frac{1}{t}\left(e^{i\langle x-b;\xi (e^{-t}-1)\rangle}-1\right)=i\langle b;\xi \rangle-i\langle x;\xi \rangle.
\end{align}
Moreover, by Lemma \ref{lem:diffallpoints}, for all $\xi \in \bbr^d$
\begin{align*}
\underset{t\rightarrow 0^+}{\lim} \frac{1}{t} \left(\psi_t(\xi)-1\right)&=\dfrac{d}{dt} \left(\psi_t(\xi)\right)(0^+)\\
&=\left( -i \langle \int_{\mathbb{S}^{d-1}}k_y(1) y\sigma(dy);\xi \rangle+ \int_{\mathbb{R}^d} \left(e^{i \langle u;\xi \rangle}-1-i \langle u;\xi \rangle\bbone_{\|u\|\leq 1}\right)\tilde{\nu}(du)\right).
\end{align*}
Then, (i) follows. Let us prove (ii). For all $\xi \in \bbr^d$ and all $t\in (0,1)$
\begin{align*}
\frac{1}{t}\left|e^{i\langle x;\xi (e^{-t}-1) \rangle}\dfrac{\varphi(\xi)}{\varphi(e^{-t}\xi)}-1\right| &\leq \frac{1}{t}\left|e^{i\langle x-b;\xi (e^{-t}-1)\rangle}-1\right| +\frac{1}{t}\left|\psi_t(\xi)-1\right|\\
&\leq \left(\|x\|+\|b\|\right) \|\xi\|++\frac{1}{t}\left|\psi_t(\xi)-1\right|.
\end{align*}
Now, observe that $\psi_t$ is the characteristic function of an infinitely divisible distribution since $X$ is SD. Let us denote by $\eta_t(\xi)$ its L\'evy-Khintchine exponent. Thus, \cite[Lemma $7.9$]{S}, for all $t\in (0,1)$ and all $\xi \in \bbr^d$
\begin{align*}
\left|\psi_t(\xi)-1\right|=\left| e^{\eta_t(\xi)}-1\right|=\left|\int_{0}^1 \dfrac{d}{ds}\left(\exp\left(s\eta_t(\xi)\right)\right)ds\right|\leq \left| \eta_t(\xi)\right|.
\end{align*}
Moreover, by Lemma \ref{lem:diffallpoints}, for all $\xi \in \bbr^d$ and all $t>0$
\begin{align*}
\eta_t(\xi)&=\int_{\mathbb{S}^{d-1}} \sigma(dx) \int_0^{+\infty} (-dk_x(s)) \int_{s e^{-t}}^s  \left(e^{i r\langle x;\xi \rangle}-1-i\langle rx;\xi \rangle\bbone_{r\leq 1}\right)\frac{dr}{r}\\
&\quad\quad-i \langle \xi ;\int_{\mathbb{S}^{d-1}}x e^{-t}\left(\int_{1}^{e^t} k_x(r) dr\right) \sigma(dx) \rangle.
\end{align*}
Hence, for all $\xi \in \bbr^d$ and all $t\in (0,1)$
\begin{align}\label{ineq:eta}
\left| \eta_t(\xi)\right| &\leq \frac{\|\xi\|^2}{2} \int_{\mathbb{S}^{d-1}} \sigma(dx) \int_0^1s^2 (-dk_x(s)) \left(1-e^{-2t}\right)\nonumber\\
&\quad\quad+\int_{\mathbb{S}^{d-1}} \sigma(dx)\left| \int_1^{e^t} (-dk_x(s)) \int_{s e^{-t}}^s \left(e^{i r\langle x;\xi \rangle}-1-i\langle rx;\xi \rangle\bbone_{r\leq 1}\right)\frac{dr}{r}\right|\nonumber\\
&\quad\quad+2t\int_{\mathbb{S}^{d-1}} \sigma(dx) \int_{1}^{+\infty} (-dk_x(s)) +\|\xi\| \int_{\mathbb{S}^{d-1}} k_x(1)\sigma(dx) \left(1-e^{-t}\right).
\end{align}
Now, via an integration by parts, for all $t\in (0,1)$, all $\xi \in \bbr^d$ and all $x\in \mathbb{S}^{d-1}$
\begin{align*}
\int_1^{e^t} (-dk_x(s)) \int_{s e^{-t}}^s \left(e^{i r\langle x;\xi \rangle}-1-i\langle rx;\xi \rangle\bbone_{r\leq 1}\right)\frac{dr}{r}&=\int_1^{e^t} \frac{k_x(s)}{s}\bigg(e^{i s\langle x;\xi \rangle}-e^{i s e^{-t}\langle x;\xi \rangle}+i\langle s e^{-t}x;\xi \rangle\bigg)ds\\
&\quad\quad+\bigg(k_x(1) \int_{e^{-t}}^1\left(e^{i r\langle x;\xi \rangle}-1-i\langle rx;\xi \rangle\right)\frac{dr}{r}\\
&\quad\quad -k_x(e^t)\int_{1}^{e^t}\left(e^{i r\langle x;\xi \rangle}-1\right)\frac{dr}{r}\bigg).
\end{align*}
Therefore, for all $t\in (0,1)$, all $\xi \in \bbr^d$ and all $x\in \mathbb{S}^{d-1}$
\begin{align}\label{ineq:bound}
\left|\int_1^{e^t} (-dk_x(s)) \int_{s e^{-t}}^s \left(e^{i r\langle x;\xi \rangle}-1-i\langle rx;\xi \rangle\bbone_{r\leq 1}\right)\frac{dr}{r}\right|&\leq 4 k_x(1) t + \|\xi\|^2 k_x(1) \left(1-e^{-2 t}\right).
\end{align}
Combining \eqref{ineq:eta} with \eqref{ineq:bound}, for all $t\in (0,1)$ and all $\xi \in \bbr^d$
\begin{align*}
\left| \eta_t(\xi)\right| &\leq \frac{\|\xi\|^2}{2} \int_{\mathbb{S}^{d-1}} \sigma(dx) \int_0^1s^2 (-dk_x(s)) \left(1-e^{-2t}\right)+4 \int_{\mathbb{S}^{d-1}}k_x(1)\sigma(dx) t + \|\xi\|^2  \int_{\mathbb{S}^{d-1}}k_x(1)\sigma(dx)\\
&\quad\quad\times \left(1-e^{-2 t}\right)+2t\int_{\mathbb{S}^{d-1}} \sigma(dx) \int_{1}^{+\infty} (-dk_x(s)) +\|\xi\| \int_{\mathbb{S}^{d-1}} k_x(1)\sigma(dx) \left(1-e^{-t}\right).
\end{align*}
This concludes the proof of the lemma.
\end{proof}

\begin{prop}\label{prop:ExpConv}
Let $X$ be a non-degenerate SD random vector in $\bbr^d$ without Gaussian component, with law $\mu_X$, characteristic function $\varphi$ and L\'evy measure $\nu$ with polar decomposition given by \eqref{eq:PolarSD2}. Let there exist $\varepsilon\in (0,1)$ such that $\bbe \|X\|^\varepsilon <+\infty$, and $\beta_1>0,\beta_2>0$, $\beta_3\in (0,1)$ such that
\begin{align}\label{eq:polyexpo}
&\gamma_1=\underset{t\geq 0}{\sup}\, \left(e^{\beta_1 t}\int_{(1,+\infty)\times\mathbb{S}^{d-1}}\frac{k_x(e^t r)}{r}dr\sigma(dx)\right)<+\infty,\nonumber\\
&\gamma_2=\underset{t\geq 0}{\sup}\, \left(e^{\beta_2 t}\int_{(0,1)\times\mathbb{S}^{d-1}} rk_x(e^tr)dr\sigma(dx)\right)<+\infty,
\end{align}
and that
\begin{align}\label{eq:correclog}
\gamma_3=\underset{t\geq 0}{\sup} \left(e^{-(1-\beta_3)t} \left\|\int_{\mathbb{S}^{d-1}}x\left(\int_{1}^{e^t} k_x(r)dr\right)\sigma(dx)\right\|\right)<+\infty,
\end{align}
where $k_x(r)$ is given by \eqref{eq:PolarSD2}.
Let $X_t$, $t>0$, be the ID random vector with law $\mu_t$ defined via its characteristic function $\varphi_t$, by
\begin{align*}
\varphi_t(\xi)=\dfrac{\varphi(\xi)}{\varphi(e^{-t}\xi)},\quad \xi \in \bbr^d.
\end{align*}
Moreover, let
\begin{align*}
\underset{t>0}{\sup}\, \bbe\, \|X_t\|^\varepsilon <+\infty.
\end{align*}
Then, for all $t>0$
\begin{align*}
d_{W_1}(\mu_t,\mu_X)\leq C e^{-c t},
\end{align*}
for some $C>0,c>0$ depending on $d$, $\varepsilon$, $\beta_1$, $\beta_2$, $\beta_3$, $\gamma_1$, $\gamma_2$ and $\gamma_3$.
\end{prop}

\begin{proof}
The strategy of the proof is similar to the one of \cite[Theorem $A.1$]{AH18_2} but without the first moment assumption. The proof of \cite[Theorem $A.1$]{AH18_2} is divided into $3$ steps; the last two depending on the finiteness of the first moment. First of all, from Step $1$ of the proof of \cite[Theorem $A.1$]{AH18_2}, for $Z$ and $Y$ two random vectors of $\bbr^d$ and for all $r\geq 1$,
\begin{align}\label{ineq:W1Wr}
d_{\widetilde{W}_1}(Z,Y)\leq \overline{C}_r \left(d_{\widetilde{W}_r}(Z,Y)\right)^{\frac{1}{2^{r-1}}},
\end{align}
for some $\overline{C}_r>0$ only depending on $r$ and on $d$, where
\begin{align*}
d_{\widetilde{W}_r}(Z,Y):=\underset{h\in \widetilde{\mathcal{H}}_r\cap \mathcal{C}^{\infty}_c(\mathbb{R}^d)}{\sup}\left|\bbe h(Z)-\bbe h(Y)\right|,
\end{align*}
while $\widetilde{\mathcal{H}}_r$ is the set of functions which are $r$-times continuously differentiable on $\mathbb{R}^d$ such that $\|D^\alpha(f)\|_{\infty}\leq 1$, for all $\alpha\in \mathbb{N}^d$ with $0\leq |\alpha|\leq r$.\\
\textit{Step 1}: Let $g$ be an infinitely differentiable function with compact support contained in the closed Euclidean ball centered at the origin and of radius $R+1$, for some $R>0$. Then by Fourier inversion and Fubini theorem, for all $t>0$,
\begin{align}\label{ineq:ball}
\left|\bbe g(X)-\bbe g(X_t)\right|&= \left| \int_{\bbr^d} \mathcal{F}(g)(\xi) \left(\varphi(\xi)-\varphi_t(\xi)\right)\dfrac{d\xi}{(2\pi)^d}\right|\nonumber\\
&\leq \int_{\bbr^d}\left|\mathcal{F}(g)(\xi) \right| \left|\varphi(\xi)-\varphi_t(\xi)\right|\dfrac{d\xi}{(2\pi)^d}.
\end{align}
Next, let us estimate precisely the difference between the two characteristic functions $\varphi$ and $\varphi_t$. For all $t>0$ and all $\xi\in \bbr^d$,
\begin{align*}
\left|\varphi(\xi)-\varphi_t(\xi)\right| =\left|\varphi(e^{-t}\xi)-1 \right|=\left|e^{\omega_t(\xi)}-1\right|,
\end{align*}
where $\omega_t(\xi)=i\langle b;\xi\rangle e^{-t}+\int_{(0,+\infty)\times \mathbb{S}^{d-1}} \left(e^{i\langle\xi; ry \rangle}-1-i\langle ry;\xi \rangle\bbone_{r\leq e^{-t}}\right)\frac{k_y(e^t r)}{r}dr\sigma(dy)$. Now, from \cite[Lemma $7.9$]{S}, the function $\xi \rightarrow e^{s \omega_t(\xi)}$ is a characteristic function for all $s\in(0,+\infty)$. Hence, for all $\xi\in \bbr^d$ and all $t>0$,
\begin{align}\label{ineq:diffchar}
\left|\varphi(\xi)-\varphi_t(\xi)\right| &=\left| \int_0^1 \frac{d}{ds}\left(\exp\left(s \omega_t(\xi)\right)\right) ds\right|\nonumber\\
&\leq \left| \omega_t(\xi)\right|\nonumber\\
&\leq \|b\|\|\xi\| e^{-t}+\left|\int_{\bbr^d} \left( e^{i \langle u;\xi e^{-t} \rangle}-1-i\langle u;\xi e^{-t} \rangle\bbone_{\|u\|\leq 1}\right)\nu(du)\right|\nonumber\\
&\leq \|b\|\|\xi\| e^{-t}+ \left|\int_{\bbr^d} \left( e^{i \langle u;\xi e^{-t} \rangle}-1-i\langle u;\xi e^{-t} \rangle\bbone_{\|u\|\leq e^t}\right)\nu(du)\right|\nonumber\\
&\quad\quad +\left| \int_{\bbr^d}\langle u;\xi e^{-t} \rangle \bbone_{1\leq \|u\|\leq e^{t}} \nu(du)\right|\nonumber\\
&\leq \|b\|\|\xi\| e^{-t}+ \left|\int_{(0,+\infty)\times \mathbb{S}^{d-1}} \left( e^{i \langle r y;\xi \rangle}-1-i\langle r y;\xi \rangle\bbone_{r\leq 1}\right)\dfrac{k_y(e^t r)}{r}dr\sigma(dy)\right|\nonumber\\
&\quad\quad +e^{-t} \left|\langle \xi; \int_{\mathbb{S}^{d-1}}y\left(\int_{1}^{e^t} k_y(r)dr\right)\sigma(dy)\rangle \right|\nonumber\\
&\leq \|b\|\|\xi\| e^{-t}+2\left(\int_{(1,+\infty)\times \mathbb{S}^{d-1}}\frac{k_y(e^t r)}{r}dr\sigma(dy)\right)+\|\xi\|^2  \left(\int_{(0,1)\times \mathbb{S}^{d-1}} rk_y(e^tr)dr\sigma(dy)\right)\nonumber\\
&\quad\quad+e^{-t} \| \xi\| \left\|\int_{\mathbb{S}^{d-1}}y\left(\int_{1}^{e^t} k_y(r)dr\right)\sigma(dy)\right\|.
\end{align}
Plugging \eqref{ineq:diffchar} in \eqref{ineq:ball} implies that, for all $t>0$,
\begin{align*}
\left|\bbe g(X)-\bbe g(X_t)\right|&\leq \int_{\bbr^d}\left|\mathcal{F}(g)(\xi) \right| \Bigg(\|b\|\|\xi\| e^{-t}+2\left(\int_{(1,+\infty)\times \mathbb{S}^{d-1}}\frac{k_y(e^t r)}{r}dr\sigma(dy)\right)\\
&\quad\quad+\|\xi\|^2 \left(\int_{(0,1)\times \mathbb{S}^{d-1}} rk_y(e^tr)dr\sigma(dy)\right)+e^{-t} \| \xi\| \left\|\int_{\mathbb{S}^{d-1}}y\left(\int_{1}^{e^t} k_y(r)dr\right)\sigma(dy)\right\|\Bigg)\dfrac{d\xi}{(2\pi)^d}.
\end{align*}
Now, observe that, for all $p\geq 0$,
\begin{align*}
\int_{\bbr^d}\left|\mathcal{F}(g)(\xi) \right| \|\xi\|^p \frac{d\xi}{(2\pi)^d}&\leq \underset{\xi \in \bbr^d}{\sup}\left(\left|\mathcal{F}(g)(\xi) \right| \left(1+\|\xi\|\right)^{d+p+1}\right) \int_{\bbr^d}\dfrac{\|\xi\|^p}{\left(1+\|\xi\|\right)^{d+p+1}}\frac{d\xi}{(2\pi)^d},\\
&\leq C_{d,p} (R+1)^d \left(\|g\|_\infty+\underset{1\leq j\leq d}{\max} \|\partial^{d+p+1}_j(g)\|_\infty\right),
\end{align*}
for some $C_{d,p}>0$ depending on $d$ and on $p$. Thus,
\begin{align}\label{ineq:Step1}
\left|\bbe g(X)-\bbe g(X_t)\right|&\leq C_{d,1}\|b\|e^{-t}(R+1)^d \left(\|g\|_\infty+\underset{1\leq j\leq d}{\max} \|\partial^{\,d+2}_j(g)\|_\infty\right)\nonumber\\
&\quad+2C_{d,0}(R+1)^d\left(\int_{(1,+\infty)\times \mathbb{S}^{d-1}}\frac{k_y(e^t r)}{r}dr\sigma(dy)\right)\left(\|g\|_\infty+\underset{1\leq j\leq d}{\max} \|\partial^{\,d+1}_j(g)\|_\infty\right)\nonumber\\
&\quad+C_{d,2} (R+1)^d\left(\int_{(0,1)\times \mathbb{S}^{d-1}} rk_y(e^tr)dr\sigma(dy)\right)\left(\|g\|_\infty+\underset{1\leq j\leq d}{\max} \|\partial^{\,d+3}_j(g)\|_\infty\right)\nonumber\\
&\quad+ C_{d,1} e^{-t}  (R+1)^d  \left\|\int_{\mathbb{S}^{d-1}}y\left(\int_{1}^{e^t} k_y(r)dr\right)\sigma(dy)\right\|\nonumber\\
&\quad\times \left(\|g\|_\infty+\underset{1\leq j\leq d}{\max} \|\partial^{\,d+2}_j(g)\|_\infty\right).
\end{align}
The inequality \eqref{ineq:Step1} concludes Step $1$.\\
\textit{Step 2}: This last step also follows the lines of the proof of Step $3$ of \cite[Theorem $A.1$]{AH18_2} so that only the main differences are highlighted. Let $h\in\mathcal{C}_c^\infty(\mathbb{R}^d)\cap \widetilde{\mathcal{H}}_{d+3}$. Let $\Psi_R$ be a compactly supported infinitely differentiable function on $\mathbb{R}^d$, with support contained in the closed Euclidean ball centered at the origin and of radius $R+1$, with values in [0,1], and such that $\Psi_R(x)=1$, for all $x$ such that $\|x\|\leq R$. First, for all $t>0$
\begin{align*}
|\bbe h(X_t)(1-\Psi_R(X_t))|&\leq \int_{\mathbb{R}^d}(1-\Psi_R(x))d\mu_t(x)\\
&\leq \mathbb{P}\left(\|X_t\|\geq R\right)\\
&\leq \dfrac{1}{R^\varepsilon}\underset{t>0}{\sup}\, \bbe\, \|X_t\|^\varepsilon.
\end{align*}
A similar bound holds true for $|\bbe h(X)(1-\Psi_R(X))|$ since $\bbe \|X\|^\varepsilon<+\infty$. Then, combining \eqref{ineq:Step1} together with the previous bounds implies\small{
\begin{align}\label{eq:intermezzo}
\left |\bbe h(X)- \bbe h(X_t)\right| &\leq \frac{\tilde{C}_{d,\varepsilon}}{R^{\varepsilon}}+C_{d,1}\|b\|e^{-t}(R+1)^d \left(\| h\Psi_R \|_\infty+\underset{1\leq j\leq d}{\max} \|\partial^{\,d+2}_j(h\Psi_R)\|_\infty\right)\nonumber\\
&\quad+2C_{d,0}(R+1)^d\left(\int_{(1,+\infty)\times \mathbb{S}^{d-1}}\frac{k_y(e^t r)}{r}dr\sigma(dy)\right)\left(\| h\Psi_R\|_\infty+\underset{1\leq j\leq d}{\max} \|\partial^{\,d+1}_j(h\Psi_R)\|_\infty\right)\nonumber\\
&\quad+C_{d,2} (R+1)^d\left( \int_{(0,1)\times \mathbb{S}^{d-1}} rk_y(e^tr)dr\sigma(dy)\right)\left(\|h\Psi_R\|_\infty+\underset{1\leq j\leq d}{\max} \|\partial^{\,d+3}_j(h\Psi_R)\|_\infty\right)\nonumber\\
&\quad+ C_{d,1} e^{-t}  (R+1)^d  \left\|\int_{\mathbb{S}^{d-1}}y\left(\int_{1}^{e^t} k_y(r)dr\right)\sigma(dy)\right\|\left(\|h\Psi_R\|_\infty+\underset{1\leq j\leq d}{\max} \|\partial^{\,d+2}_j(h\Psi_R)\|_\infty\right),
\end{align}
}for some $\tilde{C}_{d,\varepsilon}>0$ depending only on $d$ and on $\varepsilon$. Next, as in Step $3$ of the proof of \cite[Theorem $A.1$]{AH18_2}, observe that $\| h\Psi_R \|_\infty$, $ \|\partial^{\,d+1}_j(h\Psi_R)\|_\infty$ ,$ \|\partial^{\,d+2}_j(h\Psi_R)\|_\infty$ and $ \|\partial^{\,d+3}_j(h\Psi_R)\|_\infty$ are uniformly bounded in $R$ and in $h$ for $R\geq 1$ and since $h\in \mathcal{C}_c^\infty(\mathbb{R}^d)\cap \widetilde{\mathcal{H}}_{d+3}$ (for an appropriate choice of $\Psi_R$). The last step is an optimization in $R$ which depends on the behavior of 
\begin{align*}
\int_{(1,+\infty)\times \mathbb{S}^{d-1}}\frac{k_y(e^t r)}{r}dr\sigma(dy),\quad \int_{(0,1)\times \mathbb{S}^{d-1}} rk_y(e^tr)dr\sigma(dy),\quad e^{-t} \left\|\int_{\mathbb{S}^{d-1}}y\left(\int_{1}^{e^t} k_y(r)dr\right)\sigma(dy)\right\|,
\end{align*}
with respect to $t$. Using \eqref{eq:polyexpo} and \eqref{eq:correclog}, \eqref{eq:intermezzo} becomes
\begin{align*}
\left |\bbe h(X)- \bbe h(X_t)\right| &\leq \frac{\tilde{C}_{d,\varepsilon}}{R^{\varepsilon}}+\tilde{C}_{d,1}\|b\|e^{-t}(R+1)^d 
+2\tilde{C}_{d,2}(R+1)^d \gamma_1 e^{-\beta_1 t}\\
&+\tilde{C}_{d,3} (R+1)^d \gamma_2 e^{-\beta_2 t}+ \tilde{C}_{d,4}  (R+1)^d   \gamma_3 e^{-\beta_3 t},
\end{align*}
for some $\tilde{C}_{d,\varepsilon}>0, \tilde{C}_{d,1}>0, \tilde{C}_{d,2}>0, \tilde{C}_{d,3}>0$ and $\tilde{C}_{d,4}>0$. Set $\beta=\min \left( 1,\beta_1,\beta_2,\beta_3\right)$. Choosing $R=e^{\beta t/(d+1)}$ and reasoning as in the last lines of \cite[Theorem $A.1$]{AH18_2} concludes the proof of the proposition.
\end{proof}

\end{document}